\documentclass[12pt]{article}

\usepackage{amsmath,amssymb,euscript,amscd}
\usepackage{amsthm}

\newcounter{remark}

\newtheorem{theorem}{Theorem}[section]
\newtheorem{lemma}{Lemma}[section]

\newtheorem{proposition}{Proposition}[section]
\newtheorem{definition}{Definition}[section]
\newtheorem{corollary}{Corollary}[section]
\newtheorem{example}{Example}[section]

\newcommand{\CC}{\mathbb C}

\newcommand{\cC}{{\mathcal C}}

\newcommand{\cL}{{\mathcal L}}

\newcommand{\cO}{{\mathcal O}}

\newcommand{\cR}{{\mathcal R}}

\newcommand{\cV}{{\mathcal V}}
\newcommand{\cW}{{\mathcal W}}
\newcommand{\cX}{{\mathcal X}}

\newcommand{\cZ}{{\mathcal Z}}

\newcommand{\beeq}{\begin{eqnarray*}}
\newcommand{\eeeq}{\end{eqnarray*}}

\newcommand{\ba}{{\bf a}}

\newcommand{\br}{{\bf r}}

\def\qed{\hfill \vrule height 7pt width 7pt depth 0pt
           \medskip}

\def\haR{\hskip .05cm {\widehat{\hskip -.05cm R}}}
\def\dhaR{{\hskip .08cm \widehat{\hskip -.05cm \haR}}}
\def\dhaRsubsup#1#2{{\dhaR}_{#1}^{\hskip
.08cm\lower4pt\hbox{$\scriptstyle#2$}}}

\def\hatheta{\hskip .08cm {\widehat{\hskip -.08cm \theta}}}
\def\dhatheta{{\hskip .08cm \widehat{\hskip -.08cm \hatheta}}}
\def\dhathetasubsup#1#2{{\dhatheta}_{#1}^{\hskip
.08cm\lower4pt\hbox{$\scriptstyle#2$}}}
\def\12{{\frac 12}}

\def\first_indent{\hskip 1.65em}

\def\ba{\begin{array}}
\def\ea{\end{array}}
\def\beq{\begin{equation}}
\def\eeq{\end{equation}}
\def\beqa{\begin{eqnarray}}
\def\eeqa{\end{eqnarray}}
\def\beqn{\begin{eqnarray*}}
\def\eeqn{\end{eqnarray*}}

\def\Im{{\hbox{Im}}}
\def\Re{{\hbox{Re }}}

\begin{document}
\title{\bf State-space approach to zero-modules
of proper transfer functions}

\author{
Gy\"orgy Michaletzky\\
E\"otv\"os Lor\'and University \\
H-1111 P\'azm\'any P\'eter s\'et\'any 1/C,\\
Budapest, Hungary \\
e-mail: michaletzky@caesar.elte.hu\\
   \mbox{}\\
}
\date{\today}

\maketitle

\begin{abstract}
The poles and zeros of a transfer function can be studied by 
various means. The main motivation of the present paper 
is to give a
{\bf state-space} description of the module theoretic definition
of zeros introduced and analyzed by Wyman et al. in 
\cite{WSCP-89} and \cite{WSCP-91}. This analysis is 
carried out for proper transfer functions. 

The obtained
explicit equations determined by the system matrices are 
used for defining two {\it inner functions} to transform the 
original transfer function into a 
{\it square, invertible }
one via multiplication  
eliminating the ``generic'' zeros corresponding to
the kernel and the image of the transfer function. 

As it is well-known the zeros are connected to various
invariant subspaces arising in geometric control, see 
e.g. Aling and Schumacher \cite{AL-SC-84} for a complete 
description. The connections to these subspaces are also
mentioned in the paper.

{\bf Keywords}: zeros of transfer function, realization of 
transfer function, state-space description, proper rational 
function, output-nulling subspace, controlled invariant
subspace, input-containing subspace.

{\bf MSC2000 Numbers}: 30E05, 93B29, 93B30
\end{abstract}

\section{Introduction}
The study of zeros of transfer functions has already a long
history. Various zeros has been defined, various 
approaches has been used to describe them. We are not brave
enough to give a detailed description of this history
but the book written by H. Rosenbrock ('70) should be cited here
\cite{ROS-70} as well as that of T. Kailath ('80)
\cite{KAI-80}. One of the approaches used in these books to define
the zeros of a transfer function is based on the
Smith-McMillan form of these functions. These are the so-called
transmission zeros. C. B. Schrader and M. K. Sain ('89) in 
\cite{SCH-SAI-89} give a
survey on the notions and results of zeros of linear 
time invariant systems,
including invariant zeros, system zeros, input-decoupling zeros,
output-decoupling zeros and input-output-decoupling zeros, as well.
The connection of these zeros to invariant subspaces
appearing in geometric control theory was considered e.g. in
A. S. Morse ('73) \cite{MOR-73} for strictly proper transfer functions,
for proper transfer functions -- not assuming the minimality
of the realization -- in H. Aling and J. M. Schumacher ('84) 
\cite{AL-SC-84} showing that the combined decomposition
of the state space considering Kalman's canonical 
decomposition and Morse's canonical decomposition in
the same lattice diagram corresponds to the various
notions of multivariate zeros.

The book written by J. Ball, I. Gohberg and L. Rodman
\cite{BAL-GOH-ROD-90} uses the concept of left
(and Right) zero pairs. This offers the possibility
of analyzing -- together with the position of the zeros --
the corresponding zero directions, as well.

The zeros play an important role in the theory of
spectral factors. The connection between  the zeros of
spectral factors, splitting subspaces and the algebraic
Riccati-inequality was studied in A. Lindquist et al. ('95)
\cite{LMP-95}. An important aspect of this paper
was further analyzed by P. Fuhrmann and A. Gombani ('98) 
were the concept of {\it externalized} zeros was introduced.
(Interestingly, this concept can be formulated in the
framework of the dilation theory, as was pointed out by 
the author in \cite{MIC-03}.)

The starting point of the present paper is the module-theoretic
approach to the zeros of multivariate transfer functions
defined by B. F. Wyman and M. K. Sain ('83) \cite{WYM-SAI-83},
and further analyzed by Wyman et al. in \cite{WSCP-89},
\cite{WSCP-91}. In this extension the so-called 
Wedderburn-Forney-spaces play an important role.
(Although the published version of the paper written by
G. D. Forney \cite{FOR-75} does not contain an explicit 
definition of this construction, it was in the original manuscript.)
The main result in \cite{WSCP-91} is that the number of
zeros and poles of a rational transfer function coincide
(even in the matrix case) assuming that the zeros are
counted in a {\it right way}. It is well-known that to define
the multiplicity of a finite zero (or even an infinite zero)
the Rosenbrock matrix provides an appropriate tool.
But it is an easy task to construct (non-square)
matrix-valued transfer function with no finite
(infinite) zeros. In such cases it might happen that
there are rational functions mapped to the
{\it identically} zero function by the transfer 
function. Then the functions n the kernel of the 
transfer function form an infinite dimensional
vector space over the space of scalars, but it is
finite dimensional over the field of rational functions.
But defining the multiplicity of this 
{\it zero-function} as the corresponding dimension
of the kernel subspace does not give a satisfactory
result. To this aim the notion of minimal polynomial
bases should by used as in \cite{FOR-75} 
by G. D. Forney.

The main motivation of the present paper is to give a
matrix theoretic description of the corresponding
zero-concepts, i.e. to show how to compute these
zero-modules starting from a state-space realization of 
the transfer function. 

Section \ref{sec:pre_not} gives a short 
introduction to the zero-modules
and minimal polynomial bases.

Section \ref{sec:zeros_fnc} 
first refreshes the fact that the finite
zeros can be described by the Rosenbrock-matrix, namely
if $F(z) = D + C\left(zI-A\right)^{-1}B$ then the
equation
$\rule{0in}{4ex}
\left[\begin{array}{cc} A & B \\ C & D
\end{array}\right]
\left[\begin{array}{c} \Pi \\ H\end{array}\right]
=
\left[\begin{array}{c} \Pi \Lambda \\ 0 \end{array}
\right]$
should be considered. (The minimality of
the realization will not be assumed in the paper, only
the {\it observability} of the pair
$\left(C, A\right)$.) But it turns out that the same equation
describes the zeros corresponding to the kernel-module, as well,
and although there is a possibility to consider a
{\it maximal} solution of this equation, this maximality
is well-defined in terms of $\Im \left(\Pi\right)$ but
in general the matrix $\Lambda$ (and $H$) is not uniquely defined.
Loosely speaking, some part of it can be freely chosen. 
It is shown that for this maximal solution the
subspace $\Im \left(\Pi\right)$ is the 
maximal output-nulling controlled invariant subspace
(denoted by $\cV^{*}\left(\Sigma\right)$, where
$\Sigma$ indicates the system), while the maximal output-nulling
reachability subspace 
(denoted by $\cR^{*}\left(\Sigma\right)$) describes
that part of the matrix $\Lambda$ where it is not
uniquely defined by the system matrices. 
A maximal solution of the equation
$\rule{0in}{4ex}
\left[\begin{array}{cc} A & B \\ C & D
\end{array}\right]
\left[\begin{array}{c} 0 \\ R_0\end{array}\right]
=
\left[\begin{array}{c} \Pi \alpha_0 \\ 0 \end{array}
\right]$
should be considered and $\cR^{*}\left(\Sigma\right)$ 
is given 
$\Im\left( \Pi \left< \Lambda \mid \alpha_0\right>\right)$,
where $\left<\Lambda \mid \alpha_0\right>$ denotes
the minimal $\Lambda$-invariant subspace containing
$\Im \left(\alpha_0\right)$. As a side result, we obtain
that the minimal indices corresponding to the kernel of $F$
coincide with the controllability indices of the pair
$\left(\Lambda, \alpha_0\right)$. It should be noted here
that the correspondence between the various zeros and the
various invariant subspaces was thoroughly investigated
e.g. in \cite{AL-SC-84} by Aling and Schumacher even in the 
general non-minimal case. Especially, they proved that 
$\cR^{*}\left(\Sigma\right)$ corresponds to the kernel of $F$, 
while 
$\left(\cV^{*}\left(\Sigma\right)\cap 
\left<A\mid B\right>\right)/\cR^{*}\left(\Sigma\right)$
to the finite transmission zeros (assuming the observability
of $(C, A)$). But the explicit reference to the equation
above was not given by them.

Using the maximal solutions of the equations above
a matrix valued tall inner
(in continuous time sense) function $K(z)$ is
constructed explicitly with
columns forming a basis (over the field of rational
functions) in the kernel of the transfer function.
Via a square-inner extension $L$ of $K$ 
(i.e. $\left[ K, L\right]$ is a square inner function)
the generic zeros corresponding to the kernel-zero module
can be turned into finite zeros, in other words
the function $F_{\text{\bf r}} = FL$ has already a trivial kernel
(moreover it is left-invertible) 
but still containing the original finite zeros of $F$.
In terms of the language of geometric control theory,
starting from a minimal realization of $F$ and deriving
from this a realization for $F_{\text{\bf r}}$ 
these realizations
share the same maximal output-nulling controlled invariant
subspace: $\cV^{*}\left(\Sigma_{\text{\bf r}}\right) = 
\cV^{*}\left(\Sigma\right)$, but 
$\cR^{*}\left(\Sigma_{\text{\bf r}}\right)$ becomes
trivial.

In order to eliminate the defect in the image space
in Section \ref{sec:left_right} first
the connection between the left and 
right zero-modules is analyzed showing especially that if 
for the same transfer function the roles of the
input signal and the output signal are changed
(i.e. instead of the effect of the right
multiplication $g\rightarrow Fg$ the left multiplication
$h\rightarrow hF$ is considered) then -- assuming a
minimal realization is taken -- the orthogonal complement
of the maximal output-nulling controlled invariant
subspace defined for right multiplication is the minimal
input-containing subspace defined for the left
multiplication. Shortly, 
$\left(\cV^{*}\left(\Sigma\right)\right)^\perp
=\cC^{*}\left(\Sigma\right)_\text{left}$.

Now using an appropriate flat inner function $L'$ the
left kernel-zero module of $F$ can be eliminated. 
Simultaneous application of the inner functions
$L$ and $L'$ leads to the following definition:
$F_{\text{\bf rl}} = L' F L$. Theorem \ref{thm:image_out} 
claims that if the 
poles of $F$ are in the closed left half-plane
while there is no finite zero on the imaginary axis then
the McMillan degree of $F_{\text{\bf rl}}$ 
is the same as that of $F$,
and the function $F_{\text{\bf rl}}$ 
has only finite and possibly
infinite zeros, thus its kernel-zero module
and the zero module corresponding to the defect in
the image space are trivial. The inner functions
$L'$ and $L$ transform the so-called generic zeros
into finite zeros positioned in the
open right half-plane. The function $F_{\text{\bf rl}}$
is a square, invertible function, thus the "squaring"
of $F$ is achieved via left- and right multiplication.
preserving the poles of the original transfer function.
In the paper written by Ntogramatzidis and Prattichizzo
\cite{NTO-PRA-07} this squaring is obtained via
state-feedback and output-injection.

\section{Preliminaries and notation} \label{sec:pre_not}

Let $U$ and $Y$ be vector spaces over $\CC$ of dimensions 
$q$ and $p$, respectively.
As usual, $\CC(z)$ denotes the field of 
rational functions, $\CC[z]$
the ring of polynomials over $\CC$. Set
\[
U(z) = U \otimes_\CC \CC(z)\,, \quad 
Y(z) = Y \otimes_\CC \CC(z)\,.
\]
(these are the sets of vector valued rational functions).

Let $F(z)$ be a {\it transfer function}, i.e.
an $\CC(z)$ linear map
\[
F(z) : U(z) \rightarrow Y(z)\,.
\]
Choosing bases (over $\CC$) in $U$ and $Y$ we obtain bases for
$U(z)$ and $Y(z)$ (over $\CC(z)$) and a $p\times q$ matrix
representation for $F(z)$.

Let us introduce the notations
\[
\Omega U = U \otimes_\CC \CC[z]\,,
\quad \Omega Y = Y \otimes_\CC \CC[z]\,,
\] 
(these are the sets of vector-valued polynomials) and
\[
\Omega_\infty U = U \otimes_\CC \cO_\infty\,, \quad 
\Omega_\infty Y = Y \otimes_\CC \cO_\infty\,,
\]
where $\cO_\infty$ denotes the set of 
proper rational functions 
in $\CC (z)$. Obviously,
\[
z^{-1}\Omega_\infty U\,,\quad
z^{-1}\Omega_\infty Y\,
\]
are the sets of the strictly proper vector-valued rational functions.

Following R. Kalman we might identify the set 
$\Omega U$ with the (finite) \underline{past} 
(with respect to the zero time point) inputs,
and $\Omega Y$ with the (finite) \underline{past} outputs.

\subsection{Zero and pole modules of a transfer function}
\label{subsec:modules}

In this subsection we recall the definition of
 the pole and 
zero modules following Wyman and Sain 
\cite{WYM-SAI-83}.
The finite pole module is given as
\[
X(F) = \frac{\Omega U}{F^{-1}(\Omega Y) \cap \Omega U}\;.
\]
That is, the set of polynomial inputs is 
factorized by the polynomial inputs giving 
rise to polynomial outputs.

Similarly, the infinite pole module is
\[
X_\infty(F) = \frac{z^{-1}\Omega_\infty U}
{F^{-1}(z^{-1}\Omega_\infty Z) 
\cap z^{-1}\Omega_\infty U}\;.
\]
To define the zero module we might start with 
\[
\frac{F^{-1}(\Omega Y)}{F^{-1}(\Omega Y) \cap \Omega U}\;,
\]
(the set of inputs leading to polynomial outputs factorized
by the inputs which are themselves polynomial, in other words
the set of inputs producing no outputs
after time zero where two inputs are considered
to be equivalent if they differ only in the past). 

In those cases, when there are inputs producing identically
zero outputs, in other words the
kernel of the transfer function is nontrivial, then the space
above is infinite dimensional (over $\CC$).
Factorizing out this kernel  
we obtain the "module of finite zeros":
\begin{eqnarray*}
Z(F) &=& \left(
\frac{F^{-1}(\Omega Y)}{F^{-1}(\Omega Y)\cap \Omega U}
\right)
/
\left(
\frac{ker F (z)}{\ker F(z) \cap \Omega U}
\right)
\\
&&\quad = \, \rule{0in}{5ex}
\frac{F^{-1}(\Omega Y) + \Omega U}{\ker F + \Omega U}\;.
\end{eqnarray*}
The infinite zero module is defined similarly
\[
Z_\infty (F) = 
\frac{F^{-1}(\Omega_\infty Y) + \Omega_\infty U}
{\ker F + \Omega_\infty U}\;,
\]

To define a finite-dimensional 
object {\it counting} 
the ``number of zeros'' 
corresponding to the possibly infinite
dimensional (over $\CC$) of $\ker F$ 
there are two possibilities offered by Forney 
\cite{FOR-75}.
The first one is based on the so-called
Wedderburn-Forney  spaces, the second
one uses the notion of ``minimal polynomial bases''.

To define the first one let us start with
introducing a 
mapping $\pi_-$ rendering to any rational 
function its strictly proper part. I.e
\[
\pi_-: \CC(z) \rightarrow z^{-1}\cO_\infty\;.
\]
This can be extended in an obvious manner to a mapping
from $U(z)$ to $z^{-1}\Omega_\infty U$, and also
to a mapping from $Y(z)$ to $z^{-1}\Omega_\infty Y$.
Both these extended mappings 
will be denoted by the same symbol $\pi_-$.
(Similarly, $\pi_+$ denotes the mapping 
producing the polynomial part of 
any rational function.)

Now the kernel subspace $\ker F$ is 
obviously a module over $\CC (z)$.
The Wedderburn-Forney space obtained from it 
is denoted by $\cW (\ker F)$ and defined as
\[
\cW (\ker F) = \frac{\pi_-(\ker F)}
{\ker F \cap z^{-1}\Omega_\infty U}
\]

According to Theorem 5.1 (and Corollary 5.2) in Wyman et al. 
\cite{WSCP-89} for every
rational transfer function the number of poles
and zeros are equal, if they are ``counted''
in an appropriate way. Namely, set
\[
\cX (F) = X (F) \oplus X_\infty (F)\;, \quad
\cZ (F) = Z (F) \oplus Z_\infty (F)\;.
\]
Then for some linear mappings $\alpha$ and $\beta$
the sequence
\[
0 \rightarrow \cZ (F) 
\overset{\alpha}{\longrightarrow} \frac{\cX(F)}{\cW (\ker F)}
\overset{\beta}{\longrightarrow} \cW (\Im F)
\rightarrow 0
\]
forms an exact sequence. The mapping $\alpha$ is 
induced by the mapping 
\[
(u, v) \rightarrow (\pi_+ (u+v), \pi_- (u+v))\;,
\]
from 
$F^{-1} (\Omega Y)\oplus F^{-1} (z^{-1} \Omega_\infty Y)
\rightarrow \Omega U \oplus z^{-1}\Omega_\infty U$,
while $\beta$ is induced by
\[
(u, v) \rightarrow \pi_- \left[F \cdot (u+v)\right]
\]
from 
$\ \Omega U \oplus z^{-1} \Omega_\infty U 
\rightarrow \pi_- \Im F\ $.
(More precisely, to define the factor space 
$\displaystyle\frac{\cX(F)}{\cW (\ker F)}$ 
first we have to apply
the imbedding of $\ \cW (\ker F)\ $ into
$\ \cX (F)\ $ induced by the linear mapping
$\pi_- u \rightarrow (\pi_+ u, \pi_- u)$
from
$\ \ker F \rightarrow \Omega U \oplus z^{-1} \Omega_\infty U$.)

\bigskip
\subsection{Minimal polynomial basis}
\label{subsec:minpol}

Now let us turn to the second possibility based on the notion
of minimal polynomial basis.

If $v=(v_1, \dots , v_k)$ is a k-tuple of polynomials, then
set $\text{deg}\ v = \max_j \text{degree}\ v_j$. If $V$ is a 
$k\times m$ array of polynomials, then its (column)-degree 
is $\nu = \sum_l \nu_l$, where 
$V = \left[ v_1, \dots , v_m\right]$, i.e. $v_l$ denotes the
$l$-th column of $V$, and $\nu_l = \text{deg}\ v_l$. Denote by
$V_h$ the matrix of highest (column) degree coefficients of $V$.

\begin{definition}[Minimal basis]
Let $\cV$ be a finite dimensional 
subspace of $k$-tuples over
the field of rational functions $\CC (z)$.
The array $V$ is a 
\underline{minimal (polynomial) basis} of $\cV$
(over the rational functions), if 
\begin{itemize}
\item it has polynomial entries,
\item its columns form a basis over the rational functions of $\cV$, and 
\item has least (column)-degree.
\end{itemize}
\end{definition}
Denote by $\deg_{\min} \cV$ the degree of any
minimal polynomial basis in $\cV$.

Then according to Corollary 6.5 in Wyman et al.
\cite{WSCP-89} 
\[ 
\dim \cW (\ker F) =  \deg_{\min} \ker F\;.
\]
 
Similarly, 
\[
\dim \cW (\Im F) = \deg_{\min} \Im F\;.
\]

Thus the correct formulation the 
statement about the number of zeros and poles
is as follows:
\[
\dim \cX (F) = \dim \cZ (F) + \deg_{\min} \ker F 
+ 
\deg_{\min} \Im F\;. 
\]

Later on we need the following 
characterization of minimal 
(polynomial) bases given by Forney \cite{FOR-75}.

\begin{theorem}\label{Forney_thm}
Let $\cV$ be an $n$-dimensional subspace of \, $k$-tuples 
of rational functions. Assume that
$V = \left[ v_1, \dots , v_m\right]$ has polynomial 
entries. Then the following 
properties are equivalent:
\begin{itemize}
\item[(i)] $V$ is a minimal basis of $\cV$;
\item[(ii)] if $\xi = V\zeta$ 
is a polynomial $k$-tuple, $\xi\in\cV$, then
$\zeta$ must be a polynomial $m$-tuple, and
$\text{deg}\ \xi = \max_{1\leq l \leq m} 
\left[ \text{deg}\ \zeta_l + \nu_l\right]$;
\item[(iii)] $\dim \cV_d = \sum_l \left(d-\nu_l\right)^+$,
where $\cV_d$ denotes the set 
of polynomials in $\cV$ of 
degree strictly less, then $d$,
\item[(iv)] for any complex number 
$z_0$ the matrix $V(z_0)$ has full 
column rank, and $V_h$ is also 
of full column rank.
\end{itemize}
\end{theorem}

\subsection{State-space realization, invariant subspaces}

In this paper we are going to 
characterize the zero
modules of a proper transfer 
function $F(z)$ based on some linear
equations using the so-called 
Rosenbrock-matrix 
associated to $F$.

Following Rosenbrock we shall use the notation
\[
F(z) \sim \Sigma =
\left(
\begin{array}{c|c} A &   B
\\
\hline \rule{0cm}{.42cm}
   C & D
\end{array}
\right)
\]
indicating that $F(z) = D + C\left(zI - A\right)^{-1} B$.

Although the definition of zeros considered in this paper
of a rational
function does not depend on whether 
a continuous or a discrete time system is associated to it
it will turn out of the analysis later that 
discrete time systems arise in a natural way.
Namely, the system
\begin{equation}\label{system}
\left\{
\begin{array}{rcl}
x(k+1) &=& A x(k) + B u(k)\;, \\
\rule{0in}{3ex}
y(k) &=& C x(k) + D u(k)\;.
\end{array}
\right.
\end{equation}
The subspace $\left< A \mid B\right> = \Im 
\left[ B, AB, A^2B , ...\right]$ is the
reachability subspace of the state-space.

A subspace $\cV$ of the state space is called
{\bf output-nulling controlled-invariant}
if there exists a feedback map $K$ such that 
\[
\left(A+BK\right) \cV \subset
\cV 
\subset
\ker \left(C+DK\right)\;.
\]

It is well-known that there exists a maximal
output-nulling controlled-invariant set --
denoted by $\cV^{*}\left(\Sigma\right)= 
\cV^{*}(A,B,C,D)\ $.
(See for example P. A. Fuhrmann and U. Helmke
 \cite{FUR-HEL-01} where
these sets are characterized using polynomial
and rational models of state-space systems.)
Note that 
$\cV^{*}\left(\Sigma\right) \cap 
\left<A \mid B\right>$ is also
an output-nulling controlled-invariant subspace.
 
\medskip
The set of the {\bf output-nulling reachable} elements 
$\ \cC^{*}\left(\Sigma\right)$ of the system 
(\ref{system}) also 
plays important role in this paper. 
This is defined as follows:
\begin{equation}\label{reachable_set}
\cC^{*}\left(\Sigma\right) =  
\left\{
\begin{array}{c|l}
x \  & \ 
\exists \ 
\left( \dots , 0, u(-k), u(-k+1), \dots , u(0) 
\right) \ \text{input}\ \text{such that}\\ 
 & \ \rule{0in}{3ex}\kern .3in
y(j) = 0\;, \ j\leq 0\quad  \text{and} \quad
x = x(1) 
\end{array}
\right\}
\end{equation}
Obviously 
$\cC^{*}\left(\Sigma\right) = \cC^{*}(A,B,C,D)$ 
forms a subspace
of the state-space. Note that 
$\cC^{*}\left(\Sigma\right)$ coincides
with the {\bf minimal input-containing} subspace. See e.g.
Aling and Schumacher \cite{AL-SC-84}. (A subspace $\cC$ is
called input containing if there exists an output-injection
$L$ such that $\left(A+LC\right) \cC \subset \cC$ and
$\Im (B+LD) \subset \cC$.)

The intersection 
$\cR^{*}\left(\Sigma\right) 
= \cV^{*}\left(\Sigma\right) 
\cap \cC^{*}\left(\Sigma\right)$ is
the {\bf maximal output-nulling reachability} subspace. 
(See again \cite{AL-SC-84}.)

For a matrix $A$ its adjoint 
will be denoted by $A^{*}$, 
while for a matrix valued 
function $F(z)$ the notation
$F^{*}(z)$ refers to its 
para-hermitian conjugate function, 
i.e. $F^{*}(z) = \left( F(-\bar z)\right)^{*}$. 

\medskip
\section{Zeros of proper transfer functions}
\label{sec:zeros_fnc}

\bigskip
As we have seen there are several ingredients of 
the ``zero structure'' of a transfer function. 

Let us recall that to determine the finite zero module $Z (F)$ 
first we have
characterize the set  $F^{-1}(\Omega Y) + \Omega U$, i.e. those functions $h$ for which there
exists a polynomial $q$-tuple $\psi$  (in $\Omega U$) such that
$\phi = F\cdot (h+\psi)$ is a polynomial $p$-tuple.
In order to get $Z(F)$ this should be factorized with respect to 
$\ker F + \Omega_\infty U$. This set contains the functions $h$ for which there
exists a polynomial $q$-tuple $\psi$ such that $F\cdot (h+\psi ) = 0$.

Similarly, to characterize the infinite zero module we have to
consider $F^{-1}(\Omega_\infty Y) + \Omega_\infty U$ with
a similar charaterization as above but instead of polynomials we have to consider
proper functions.
To get $Z_\infty (F)$ this should be factorized with respect to
$\ker F + \Omega_\infty U$.

To obtain the kernel-module $\cW ( \ker F )$ 
we have to first consider $\pi_- \left( \ker F\right)$, i.e. the set of those 
strictly proper functions $h$ for which there exists a polynomial 
$q$-tuple $\psi$ such that $F \cdot \left(h+\psi \right) =0$.
 In order to get $\cW (\ker F )$ this set should be factorized 
with respect to the set of strictly proper functions in the kernel of $F$. 

Similarly to describe $\cW ( \Im F )$ we should first consider $\pi_- (\Im F)$,
i.e. the strictly proper functions $h$ for which there exists a
polynomial $p$-tuple $\phi$ such that $h+\phi \in \Im (F)$. 
Two functions $h_1, h_2$ are considered to be equivalent if
$h_1-h_2 \in \Im (F)$.

\subsection{The finite zero-module
$Z(F)$ and the kernel-module
$\cW (\ker F)$}
\label{subsec:fin_ker}

Let us start with the analysis of 
$F^{-1}(\Omega Y) + \Omega U$
appearing in the definition of $Z(F)$.

Let us point out that
Theorem  1 in Michaletzky-Gombani
\cite{MIC-GOM-2012} essentially characterizes these functions. 
\begin{theorem}\label{zeros_description}
Let $F(z) = D+C\left(zI - A\right)^{-1}B$ 
be a rational function. 
\begin{enumerate}
\item[(i)]
Assume that there exists a - possibly 
matrix-valued - function
$$g (z) = H\left(zI - \Lambda\right)^{-1} G 
+ \psi (z)\;,$$
where $\psi$ is a matrix-valued 
polynomial and $(\Lambda, G)$ is
a controllable pair such that 
\[
F g \ \text{ is analytic at the 
eigenvalues of} \ \Lambda. 
\]
If moreover the pair $(C, A)$ is 
observable, then there exists a 
matrix $\Pi$ such that
$\Im\, \Pi \subset < A\ \mid \ B >$ solving the equation
\begin{equation}\label{matrix_zero_eq}
\left[ \begin{array}{cc} A & B \\ C & D
  \end{array}\right]
\left[ \begin{array}{c}
\Pi \\ H \end{array} \right]
= \left[ \begin{array}{c}  \Pi \Lambda \\  0 
\end{array}\right]\,.
\end{equation}
\item[(ii)]
Assume that the matrices 
$\Lambda, H, \Pi$ satisfy the equation
(\ref{matrix_zero_eq}), where 
$\left( H, \Lambda\right)$ is an
observable pair and 
$\Im\, \Pi \subset < A\ \mid \ B >$. 
Then there exists a matrix polynomial $\psi$ 
such that for 
\begin{equation}\label{function_phi}
g (z) =  H \left(zI - \Lambda\right)^{-1}  + \psi (z)
\end{equation}
the function $F g$ is
analytic at the eigenvalues of $\Lambda$. 

Finally, equation
(\ref{matrix_zero_eq}) implies that 
\begin{equation}\label{matrix_product_eq}
F(z) H \left( zI - \Lambda\right)^{-1}  =
- C\left(zI-A\right)^{-1} \Pi\;,
\end{equation}
which is already analytic at the 
eigenvalues of $\Lambda$ if
$A$ and $\Lambda$ have no common 
eigenvalues.
\end{enumerate}
\end{theorem}
Furthermore, in part (ii) the polynomial $\psi$ can be 
chosen in such a way that
the product $F g$ be a polynomial.
In other words, the columns
of the function $H \left( zI - \Lambda\right)^{-1}$
are in $F^{-1}(\Omega Y) + \Omega U$.

\medskip
This gives the possibility of formulating Theorem 
\ref{zeros_description} in the following way.
\begin{theorem}\label{zeros_description2}
Let $F(z) = D+C\left(zI - A\right)^{-1}B$ 
be a rational function. 
\begin{enumerate}
\item[(i)]
Assume that the pair $(C, A)$ 
is observable and the pair
$(\Lambda, G)$ is controllable. Then, if
the columns of the function
$$g (z) = H\left(zI - \Lambda\right)^{-1} G \;,$$
are in $F^{-1} (\Omega Y) + \Omega U$,
 then there exists a 
matrix $\Pi$ such that $\Im\,\Pi\subset < A\ \mid \ B >$ solving
the equation
\begin{equation}\label{matrix_zero_eq2}
\left[ \begin{array}{cc} A & B \\ C & D
  \end{array}\right]
\left[ \begin{array}{c}
\Pi \\ H \end{array} \right]
= \left[ \begin{array}{c}  \Pi \Lambda \\  0 
\end{array}\right]\,.
\end{equation}
\item[(ii)]
Assume that the matrices 
$\Lambda, H, \Pi$ satisfy the equation
(\ref{matrix_zero_eq2}), where 
$\left( H, \Lambda\right)$ is an
observable pair and 
$\Im\,\Pi\subset < A\ \mid \ B >$. 
Then there exists a matrix polynomial $\psi$ 
such that for 
\begin{equation}\label{function_phi2}
g (z) =  H \left(zI - \Lambda\right)^{-1}  + \psi (z)
\end{equation}
the function $F g$ is 
a polynomial, i.e. the columns of the matrix-valued rational 
function 
$\ H\left(zI-\Lambda \right)^{-1}\ $ are
in $F^{-1}(\Omega Y) + \Omega U$.
\end{enumerate}
\end{theorem}

\begin{corollary}
Assume that $F(z) = D+C\left(zI - A\right)^{-1}B$ is
a minimal realization, and $(H, \Lambda )$ is an 
observable pair.

Then the columns of $H \left( zI - \Lambda \right)^{-1}$
are in the set $F^{-1} (\Omega Y) + \Omega U$
if and only if equation (\ref{matrix_zero_eq2}) holds.
\end{corollary}

Later on we shall utilize to following proposition
which in a sense can be considered as a converse of 
the last statement in Theorem \ref{zeros_description}.

\begin{proposition}\label{product_equation}
Assume that equation
\begin{equation}\label{eliminate_poles}
\left[D+C\left(zI-A\right)^{-1}B\right]\ 
\left[H\left(zI-\Lambda\right)^{-1}G\right] = 
-C\left(zI-A\right)^{-1}S
\end{equation}
holds. 

Then if the pair
$\left(C, \ A\right)$ is observable and the pair
$\left(\Lambda, \ G \right)$ is controllable
then there exists a matrix $\Pi$ such that equation
(\ref{matrix_zero_eq2}) holds, as well. Moreover
\[
\Pi G = S\;.
\]
\end{proposition}

\begin{proof}
Let us observe that if the matrices $A$ and $\Lambda$
have no common eigenvalues then equation 
(\ref{eliminate_poles}) implies that the product is
analytic at the eigenvalues of $\Lambda$, thus 
Theorem \ref{zeros_description} (i) implies 
immediately that equation (\ref{matrix_zero_eq}) holds true.

In the general case let us observe that equation
(\ref{eliminate_poles}) can be written in the 
following form, as well.
\begin{equation}\label{extended_transfer}
\left[C, \ DH\right]
\left(
z\left[\begin{array}{cc} I & 0 \\ 
                         0 & I 
\end{array}\right] -
\left[
\begin{array}{cc}
A & BH \\ 0 & \Lambda
\end{array}
\right]
\right)^{-1}
\left[
\begin{array}{c}
S \\ G
\end{array}
\right] = 0
\end{equation}
In other notation
\[
0 \sim
\left(
\begin{array}{c|c}
\begin{array}{cc}
A & BH \\ 0 & \Lambda
\end{array} & 
\begin{array}{c}
S \\ G
\end{array}\\
\hline \rule{0cm}{.42cm}
\begin{array}{cc} 
C  & DH 
\end{array}
&  0
\end{array}
\right)\;.
\]

Let us first consider the unobservability subspace 
of the realization obtained above
of the identically zero function. Suppose that
the columns of the matrix 
$\left[\begin{array}{c} \alpha \\ \beta \end{array}\right]$
form a basis in the unobservability subspace. Then
\[
\left[ C \quad DH \right]
\left[
\begin{array}{c} 
\alpha \\ \beta 
\end{array}
\right] = 0\;,\quad 
\left[
\begin{array}{cc}
A & BH \\ 0 & \Lambda
\end{array}
\right] 
\left[
\begin{array}{c}
\alpha \\ \beta
\end{array}
\right]
= 
\left[
\begin{array}{c}
\alpha \\ \beta
\end{array}
\right]
\rho\;,
\]
for some matrix $\rho$.

If for some vector $\xi$ the product $\beta\xi = 0$
then $\beta \rho \xi = 0$, as well. Thus the subspace \ 
$\ker \beta$ is $\rho$-invariant. Consider now 
an eigenvector $\xi$ of $\rho$ belonging to this subspace.
Then
\[
\rho \xi = \lambda \xi\;,
\quad C \alpha \xi = 0\;,
\quad
A \alpha \xi = \alpha \rho \xi = \lambda \alpha \xi\;.
\]
The observability of the pair $(C, \ A)$ implies
that $\alpha \xi=0$, as well, thus the columns
of the matrix 
$\left[ \begin{array}{c} \alpha \\ \beta
\end{array}
\right]$ are linearly dependent, contrary to our
assumption. Consequently, the columns of $\beta$ are
linearly independent. Equation
$\Lambda \beta = \beta \rho$ implies that
dimension of the unobservability subspace
cannot be greater than the size of the matrix
$\Lambda$. 

Now let us assume that the row vectors of the matrix
$\left[ \gamma, \ \delta\right]$ form  
a basis in the orthogonal complement of the
controllability subspace. Then
\[
\left[\gamma \quad \delta \right] 
\left[
\begin{array}{c}
S \\ G
\end{array}
\right] = 0 \;, \quad
\left[\gamma \quad \delta \right] 
\left[
\begin{array}{cc}
A & BH \\ 0 & \Lambda
\end{array}
\right] 
= 
\tau
\left[\gamma \quad \delta \right]\;.
\]
Using the controllability of the pair $(\Lambda, G)$
similar reasoning shows that the
codimension of the controllability subspace
cannot be larger than the size of $A$. 
Since these two subspaces
together should generate the whole space 
we obtain that -- comparing the
dimensions of these subspaces with the
sizes of the corresponding matrices -- equalities
should hold. Thus $\beta$ and $\gamma$ 
should be square matrices
with trivial kernels. Applying nonsingular 
transformations
it can be assumed that both are identity
matrices. Then especially
\[
\tau = A \;, \quad S + \delta G = 0 \;,
\quad
BH + \delta \Lambda = A\delta\;. 
\]
Substituting the equation
$BH = \left(zI-\Lambda\right)\delta -
\delta \left(zI-A\right)$
into (\ref{eliminate_poles}) straightforward computation
gives that
\[
\left(DH - C\delta\right)
\left(zI-\Lambda\right)^{-1}G = 0\;.
\]
The controllability of $\left(\Lambda, \ G \right)$
implies that $DH-C\delta = 0$.
Thus the matrix $\Pi=-\delta$ 
satisfies the required equations.
\qed\end{proof}

\medskip
Note that equation (\ref{matrix_zero_eq2}) implies that
there exists a state feedback $K$ such that
\[
(A+BK) \Im \Pi \subset  \Im \Pi \subset 
\ker (C+DK )\,.
\]
Thus the columns of $\Pi$ are in the 
maximal output-nulling  controlled invariant set
of the system (\ref{system}).
I.e.
\[
\Im \left(\Pi\right) \subset 
\cV^{*}\left(\Sigma\right)\;.
\]

\medskip

For the sake of completeness we provide a proof of 
the next obvious statement showing that
$\cV^{*}\left(\Sigma\right) = 
\cV^{*}(A,B,C,D)$ can be characterized via
the "maximal solution" of equation
(\ref{matrix_zero_eq2}).

\begin{lemma}\label{lem:max_pi}
For any system $\Sigma$ determined by the matrices $A, B, C, D$
there exists a maximal solution 
$\left(\Pi_\text{max}, H_\text{max}, \Lambda_\text{max}\right)$
of the equation
\begin{equation}\label{matrix_zero_eq_max}
\left[\begin{array}{cc} A & B \\ C & D
\end{array}\right]
\left[\begin{array}{c} \Pi_\text{max} \\ H_\text{max}
\end{array}\right]
=\
\left[\begin{array}{c} \Pi_\text{max}\Lambda_\text{max} \\
0 
\end{array}\right]
\end{equation}
in the sense that 
\[
\Im \left(\Pi_\text{max} \right) 
\supset
\Im \left( \Pi_1\right)
\]
if $\left(\Pi_1, H_1, \Lambda_1\right)$ is
any other solution of the equation above.

Moreover, for this maximal solution
\[
\cV^{*}\left(\Sigma\right) = \cV^{*}(A,B,C,D) 
= \Im \left(\Pi_\text{max}\right)\;.
\]
\end{lemma}

\begin{proof}
If $\left(\Pi_1, H_1, \Lambda_1\right)$ 
and
$\left(\Pi_2, H_2, \Lambda_2\right)$ are solutions
of the equation then
\[
\left(
\left[\Pi_1,\Pi_2\right], 
\left[H_1, H_2\right],
\left[\begin{array}{cc} \Lambda_1 & 0 \\
0 & \Lambda_2
\end{array}\right]\right)
\]
is a solution, as well.
Since 
\[
\Im \left[\Pi_1, \Pi_2\right] 
=
\Im \Pi_1 \vee \Im \Pi_2
\]
the subspace generated by the ranges of 
all solutions is also the range of a solution, proving that
there exist a maximal solution.

To prove the last statement let us point out
that we have 
already observed that 
$\Im\left(\Pi\right) \subset 
\cV^{*}\left(\Sigma\right)$
for any solution 
$(\Pi, H, \Lambda)$ of 
(\ref{matrix_zero_eq2}). 

For the converse inclusion
consider a matrix $\Pi^{*}$ with column vectors
forming a basis in $\cV^{*}\left(\Sigma\right)$. 
Then 
there exists a feedback 
matrix $K^{*}$ such that the inclusions
\[
\left(A+BK^{*}\right)\Im \Pi^{*} \subset
\Im \Pi^{*} \subset 
\text{ker} \left(C+DK^{*}\right)
\]
hold. In other words
\[
\left[\begin{array}{cc}
A & B \\ C & D 
\end{array}\right]
\left[\begin{array}{c}
\Pi^{*} \\ K^{*}\Pi^{*}
\end{array}\right]
=
\left[\begin{array}{c}
\Pi^{*}\Lambda^{*} \\ 0
\end{array}\right]
\]
for some matrix $\Lambda^{*}$.

The first part of
the lemma implies that 
$\cV^{*}\left(\Sigma\right) = 
\Im \Pi^{*} \subset \Im \Pi_\text{max}$, 
if $(\Pi_\text{max}, H_\text{max}, \Lambda_\text{max})$ is a maximal 
solution, concluding the proof of
the lemma.
\qed\end{proof}

Later on we need the following simple property of 
the subspace $\Im (\Pi_\text{max})$.

\begin{proposition}
If
\begin{eqnarray}\label{next_in_nu}
x_+ &=& Ax + Bu \\
0 &=& Cx + Bu
\end{eqnarray}
and $x_+ \in \Im (\Pi_\text{max})$ then 
$x \in \Im (\Pi_\text{max})$, as well, where
$\Pi_\text{max}$ is a maximal solution of 
equation (\ref{matrix_zero_eq_max}).
\end{proposition}

\begin{proof}
Equations (\ref{next_in_nu}) and 
(\ref{matrix_zero_eq_max}) together can be written
as follows:
\begin{equation}\label{extended_matrix_zero}
\left[ \begin{array}{cc} A & B \\ C & D
  \end{array}\right]
\left[ \begin{array}{cc}
\Pi_\text{max} & x \\ H_\text{max} & u \end{array} \right]
= \left[ \begin{array}{cc}  \Pi_\text{max} \Lambda_\text{max} & x_+
\\  0 & 0 
\end{array}\right]\,.
\end{equation}
According to the assumption there exists a vector
$\zeta$ such that $x_+ = \Pi_\text{zero} \zeta$ 
giving the identity
\[
\left[ \begin{array}{cc}  \Pi_\text{max} \Lambda_\text{max} & x_+
\\  0 & 0 
\end{array}\right] =
\left[ \Pi_\text{max} , x \right]
\left[ \begin{array}{cc}  \Lambda_\text{max} & \zeta
\\  0 & 0 
\end{array}\right]\;.
\]
Substituting this into the left hand side of
(\ref{extended_matrix_zero}) and 
using the maximality
of $\Im (\Pi_\text{max})$ we obtain that 
$x\in Im (\Pi_\text{max})$. 
\qed\end{proof}

\bigskip
Now let us return to the analysis of the space
$F^{-1}(\Omega Y) + \Omega U$.
Theorem \ref{zeros_description2} suggests that
to describe the functions in this space
we have to consider
a maximal -- in some sense -- solution of 
equation (\ref{matrix_zero_eq2}). Lemma 
\ref{lem:max_pi} shows that in terms
of $Im(\Pi)$ there exists a maximal solution.
But the following lemma
shows that in terms of the triplet
$(\Pi, H, \Lambda)$ -- in general -- this is 
not possible.

\begin{lemma}
Let $h$ be a polynomial $q$-tuple. Assume that $Fh = 0$. 
Consider an 
arbitrary complex number $a \in \CC$, and write 
$h(z) = \sum_{j=0}^k h_j (z-a)^j$.
Set 
\[
H = \left[h_0, h_1, \dots , h_{k-1}, h_k  \right]
\]
and 
\[
\Lambda_a =
\left[
\begin{array}{ccccc}
a & 1       &   0    & \cdots & 0 \\
0 & a       &   1    &  & 0\\
0 &   & \ddots & \ddots  & 0 \\
0 & \cdots  & 0      & a & 1 \\
0 & 0 & \cdots & 0& a
\end{array}
\right]\,,
\]
the Jordan-matrix corresponding to the value $a$.

Then --  assuming the
observability of the pair $\left(C, A\right)$ --
the function $g(z) = \left(zI-A\right)^{-1}B h$
is also a polynomial, $g(z) = \sum_{j=0}^{k-1} g_j (z-a)^j$
and equation
\begin{equation}\label{poly_zero_eq}
\left[
\begin{array}{cc}
A & B \\ C & D
\end{array}
\right]
\left[
\begin{array}{c}
G \\ H
\end{array}
\right]
=
\left[
\begin{array}{c}
G\Lambda_a \\ 0 
\end{array}
\right]
\end{equation}
holds where
\[
G = \left[g_0, g_1, \dots , g_{k-1}, 0 \right]\;.
\]
\end{lemma}

\begin{proof}
Set $g(z) = \left(zI-A\right)^{-1} B h(z)$. 
Invoking Lemma 7.1 in \cite{AL-SC-84} we get that
$g$ is a polynomial of degree $k-1$.

Arranging the 
the identity $\left((z-a)I-(A-aI)\right) g(z) = B h(z)$
on the coefficients of $h$ and $g$
into matrix form we arrive at the equation
(\ref{poly_zero_eq})
\qed
\end{proof}

\medskip
Note that according \cite{AL-SC-84} the following more general 
statement holds, as well. Consider a set of polynomials
$h_1, h_2, ... , h_l$ from $\ker F$. Define 
$x_i (z)= (zI-A)^{-1}Bh_i (z)$, $j=1, \dots , l$. 
According to the previous
statement -- under the assumed observability
of the pair $(C, A)$ -- they are also polynomial. Then
$h_1, h_2, ... , h_l$ form a minimal polynomial
basis in $\ker F$ if and only if 
$\left[\begin{array}{c} x_1 \\ h_1\end{array}\right],
\left[\begin{array}{c} x_2 \\ h_2\end{array}\right],
... , 
\left[\begin{array}{c} x_l \\ h_l\end{array}\right]
$ form a minimal polynomial basis in
$\rule{0in}{4ex}
\ker \left[\begin{array}{cc} zI-A & B \\ -C & D
\end{array}\right]$.

\bigskip
\bigskip
Since if there exists a rational function in the
kernel of $F$ then there is also a polynomial
in it (multiplying with the common denominator of
its entries), the previous lemma shows that in this case
there is no {\it largest} matrix $\Lambda$ containing
all ``zeros'' of $F$ as its eigenvalues. But 
to characterize the finite zero module
$Z(F)$ only the equivalence classes
in $F^{-1} (\Omega Y) + \Omega U$
should be taken, where two functions
$g_1, g_2$ in it are considered to
be equivalent if $g_1-g_2 \in \ker F + \Omega U$.
Especially the polynomial part can be eliminated.
In other words,  
\[
g - \pi_- (g) \in \Omega U \subset \ker F + \Omega U\;.
\]
and
\[ 
g \in F^{-1}(\Omega Y) + \Omega U
\quad 
\text{if and only if }
\quad
\pi_- (g) \in F^{-1} (\Omega Y) + \Omega U\,.
\]

Thus our next goal is to characterize the {\it equivalence classes}
in $F^{-1}(\Omega Y) + \Omega U$ via
the equation (\ref{matrix_zero_eq2}).
As we have seen a polynomial $q$-tuple in
$\ker F$ also induces a solution of this equation.
So what we might hope is that this equation
is appropriate for characterizing not only the
elements of $Z(F)$ but also including $\ker (F)$ or
more precisely the elements in $\cW (\ker F)$. 

Let us observe that according to the definition
of $\cW (\ker F)$ a $q$-tuple
$g \in \cW (\ker F)$ if and only if there 
exists a polynomial $\psi$ such that 
$g+\psi \in \ker F$, and two functions $g_1, g_2$
with this property are considered to be equivalent
if $g_1 - g_2 \in \ker F \cap z^{-1} \Omega_\infty U$.

Summarizing these considerations to
describe the elements of 
$Z(F) \oplus \cW (\ker F)$ those rational
$q$-tuples $g$ should be considered  
for which there exists a polynomial $q$-tuple $\psi$
such that $ F (g+ \psi)$ is a polynomial, and
two functions $g_1, g_2$ are taken to be equivalent
if and only if $g_1-g_2 \in 
\left( \ker F\right) \cap z^{-1}\Omega_\infty U$.
Obviously every equivalence class contains a
strictly proper rational function.

\medskip
Since 
\[
\frac{F^{-1}(\Omega Y)+\Omega U}
{\ker F \cap z^{-1}\Omega_\infty U} 
\ \simeq \ 
Z(F) \oplus \cW (\ker F)
\]
is finite dimensional and 
every equivalence class
contains a strictly proper function
there exists a rational function
$H\left(zI-\Lambda\right)^{-1} G$
such that for every strictly proper
rational function 
$g \in F^{-1} (\Omega Y) + \Omega U$
there exists a vector $\alpha$
such that
\[
g - H \left(zI-\Lambda\right)^{-1} G \alpha 
\in \ker F \cap z^{-1} \Omega_\infty  U\;.
\] 
(We might obviously assume that
$(H, \Lambda)$ is an observable, 
while $(\Lambda, G)$ is a controllable
pair.)

\bigskip
\subsubsection{The linear space
$Z(F) \oplus \cW\left(\ker F\right)$}
\label{subsec:zk}

The argument above can be be amplified to
the following theorem.

\begin{theorem}\label{zeros_finite_kernel_thm}
\begin{itemize}
\item[(i)]
Assume that the pair $\left(C, A\right)$ is observable. 
Then there exists a pair 
$\left(H_{fzk}, \Lambda_{fzk}\right)$ and a
matrix $\Pi_{fzk}$ such that for every strictly proper
rational $q$-tuple 
\[
g \in F^{-1} (\Omega Y) + \Omega U
\] 
there
exists a vector $\alpha$ for which
\[
g(z) - H_{fzk}\left(zI-\Lambda_{fzk}\right)^{-1}\alpha
\in
\ker F \cap z^{-1}\Omega_\infty U\;,
\]
furthermore equation
\begin{equation}\label{matrix_zero_eq_fzk}
\left[ \begin{array}{cc} A & B \\ C & D
  \end{array}\right]
\left[ \begin{array}{c}
\Pi_{fzk} \\ H_{fzk} \end{array} \right]
= \left[ \begin{array}{c}  \Pi_{fzk} \Lambda_{fzk} \\  0 
\end{array}\right]\,,
\end{equation}
holds, and the kernel of $\Pi_{fzk}$ is trivial.
\item[(ii)]
If $\left(\Pi_1, H_1, \Lambda_1\right)$
provide a solution of 
(\ref{matrix_zero_eq_fzk})
then there exists a triple
$\left(\Pi^{'}, H^{'}, \Lambda^{'}\right)$
solving equation
(\ref{matrix_zero_eq_fzk})
such that 
\[
\text{ker} \Pi^{'} = \left\{ 0 \right\}\;,\quad
\Im \Pi^{'} = \Im \Pi_1
\]
and the columns of $H^{'}\left(zI-\Lambda^{'}\right)^{-1}$
generate the same equivalence classes with respect to
$\ker F \cap z^{-1}\Omega_\infty U$
as those of $H_1\left(zI-\Lambda_1\right)^{-1}$.

Moreover, if $\left(C, A\right)$ is observable and
$\Im\,\Pi_1\subset < A \ \mid \ B >$ 
then the inclusion
\[
\Im \left(\Pi_1\right) \subset 
\Im \left(\Pi_{fzk}\right)
\]
holds, where $\Pi_{fzk}$ is determined by part (i).
\item[(iii)]
Assume that the pair $\left(C, A\right)$ is
observable. Consider
a triplet $(\Pi_1, H_1, \Lambda_1)$ 
providing a solution of 
(\ref{matrix_zero_eq_fzk}). Let $G_1$ be a column
vector. Then the function $H_1\left(zI-\Lambda_1\right)^{-1}G_1$
is in $F^{-1}\left(\Omega U\right) + \Omega U$ if and only if
\[
\Pi_1 G_1 \in \left< A \mid B \right>\;.
\]
where $< A \mid B >$ denotes the reachability subspace
of the state-space.
\end{itemize}
\end{theorem}

\begin{proof}
(i)
Consider a basis in the
finite-dimensional space
of equivalence classes
\[
\left(F^{-1}(\Omega Y)+\Omega U\right)\ / 
\ \left(\ker F \cap z^{-1}\Omega_\infty U\right)
\] 
and pick up strictly proper rational functions
from their equivalence classes. Using these 
rational $q$-tuples form a matrix-valued
strictly proper rational function with 
minimal realization
\[
\tilde H\left(zI - \tilde \Lambda \right)^{-1} 
\tilde G\,.
\]
Due to the assumption that its columns
generate a basis in $Z(F) \oplus \cW(\ker F)$
for every strictly proper rational function
$g \in F^{-1} (\Omega Y) + \Omega U$
there exists a vector $\alpha$ such that
\[
g(z) - \tilde H
\left(zI-\tilde\Lambda\right)^{-1}\tilde G\alpha
\in \left( \ker F \cap z^{-1}\Omega_\infty U\right)\,.
\]
Theorem \ref{zeros_description} (i) 
-- using the observability of the pair
$\left(C, A \right)$ --
implies that there
exists a matrix $\tilde\Pi$ for which
equation (\ref{matrix_zero_eq2}) holds. 

The identity 
$z\tilde H \left( zI -\tilde\Lambda\right)^{-1}
\tilde G = 
\tilde H \tilde G +
\tilde H \left(zI-\tilde\Lambda\right)^{-1}
\tilde\Lambda\tilde G$ implies that
\[
\tilde H \left(zI-\tilde\Lambda\right)^{-1}
\tilde\Lambda\tilde G
\in 
F^{-1}(\Omega Y)+\Omega U\;.
\]
Thus there exists a matrix $\Lambda_{fzk}$ such that
\[
\tilde H \left(zI-\tilde\Lambda\right)^{-1}
\tilde\Lambda\tilde G - 
\tilde H \left(zI-\tilde\Lambda\right)^{-1}
\tilde G \Lambda_{fzk}
\in \left(\ker G \cap z^{-1}\Omega_\infty U\right)\,.
\]
On the other hand equation 
(\ref{matrix_product_eq}) implies
that
\begin{eqnarray*}
\left[D+C\left(zI-A\right)^{-1}B\right]&&
\kern -.3truein
\left[
\tilde H \left(zI-\tilde\Lambda\right)^{-1}
\tilde\Lambda\tilde G - 
\tilde H \left(zI-\tilde\Lambda\right)^{-1}
\tilde G \Lambda_{fzk}
\right] \\ & =&
-C\left(zI-A\right)^{-1}\tilde\Pi 
\left(\tilde\Lambda\tilde G - \tilde G \Lambda_{fzk}\right)\,.
\end{eqnarray*}
Invoking again the observability of the pair
$(C, A)$ we get that
\[
\tilde\Pi 
\left(\tilde\Lambda\tilde G - \tilde G \Lambda_{fzk}\right)
= 0\;.
\]
Straightforward calculation gives that
\[
\left[D+C\left(zI-A\right)^{-1}B\right]
\left[
\tilde H \left(zI-\tilde\Lambda\right)^{-1}\tilde G
-
\tilde H \tilde G \left(zI-\Lambda_{fzk}\right)^{-1}
\right]
=
0
\]
Set
\[
H_{fzk} = \tilde H \tilde G\, , \quad 
\Pi_{fzk} = \tilde \Pi \tilde G\,.
\]
Then the columns of $H_{fzk}\left(zI-\Lambda_{fzk}\right)^{-1}$
determine the same equivalence classes
as those of 
$\tilde H\left(zI-\tilde\Lambda\right)^{-1}
\tilde G$, thus they form a basis
in $Z(F) \oplus \cW\left(\text{ker} F\right)$, 
and
equation
\[
\left[
\begin{array}{cc}
A & B \\ C & D
\end{array}
\right]
\left[
\begin{array}{c} \Pi_{fzk} \\ H_{fzk}\end{array}
\right]
=
\left[
\begin{array}{c} \Pi_{fzk}\Lambda_{fzk} \\ 0 \end{array}
\right]
\]
is satisfied. 
The fact that the columns of
$H_{fzk}\left(zI-\Lambda_{fzk}\right)^{-1}$ form a basis
implies that for any non-zero vector $\alpha$
\[
\left[D+C\left(zI-A\right)^{-1}B\right]
H_{fzk} \left(zI-\Lambda_{fzk}\right)^{-1}\alpha 
= - C\left(zI-A\right)^{-1}\Pi_{fzk}\alpha \neq 0\;,
\]
thus $\Pi_{fzk} \alpha \neq 0$. In other words
\[
\ker \Pi_{fzk} = \left\{ 0\right\}\;.
\]

(ii) 
Consider any solution of the equation
\begin{equation}\label{matrix_zero_eq3a}
\left[ \begin{array}{cc} A & B \\ C & D
  \end{array}\right]
\left[ \begin{array}{c}
\Pi_1  \\ H_1   \end{array} \right]
= \left[ \begin{array}{c}  
\Pi_1\Lambda_1 \\ 
0 
\end{array}\right]\;.
\end{equation}
Define the matrix $\Pi^{'}$ in such a way
that its column vectors form a basis
in $\Im \Pi_1$. Then there are matrices 
$\alpha, \beta$ such that
\[
\Pi_1 \alpha = \Pi^{'}\;, \quad
\Pi^{'} \beta = \Pi_1\;.
\]
Equation $\Pi^{'}\beta \alpha = \Pi^{'}$
and $\text{ker} \Pi^{'} = \left\{ 0 \right\}$
imply that $\beta\alpha = I$. Set
\[
H^{'} = H_1 \alpha\;,\quad
\Lambda^{'} = \beta \Lambda_1 \alpha\;.
\]
Multiplying equation 
(\ref{matrix_zero_eq3a}) from the right
by $\alpha$ we obtain that the triplet
$\left(\Pi^{'}, H^{'}, \Lambda^{'}\right)$
provides a solution of
(\ref{matrix_zero_eq3a}), as well.

Defining a matrix $\alpha_1$ with column 
vectors forming a basis in $\text{ker} \Pi_1$, 
we get that
the matrix $\left[ \alpha, \alpha_1\right]$ 
is regular.
Now equation (\ref{matrix_zero_eq3a}) implies that
\[
F(z) H_1\left(zI-\Lambda_1\right)^{-1} = 
- C\left(zI-A\right)^{-1}\Pi_1\;,
\]
thus
\[
F(z) H_1\left(zI-\Lambda_1\right)^{-1} \alpha_1 = 0\;,
\]
i.e. the columns of 
$H_1\left(zI-\Lambda_1\right)^{-1}\alpha_1$ are in
$\ \text{ker} F \cap z^{-1}\Omega_\infty U$.
On the other hand
\[
F(z) \left( H^{'}\left(zI-\Lambda^{'}\right)^{-1}
-
H_1\left(zI-\Lambda_1\right)^{-1}\alpha\right) =
-C\left(zI-A\right)^{-1}\Pi^{'}
+C\left(zI-A\right)^{-1}\Pi_1\alpha = 0\;,
\]
proving the first part of (ii).

Now -- using the observability of the pair 
$\left(C, A\right)$ --
we show that $\left(H^{'}, \Lambda^{'}\right)$ 
is observable, as well.
In fact, if for some $\xi$ the 
identities $H^{'}\xi = 0$,
$\Lambda^{'} \xi = \lambda \xi$ holds, then 
equation (\ref{matrix_zero_eq3a}) 
implies that
\[
A \Pi^{'} \xi = \Pi^{'}\Lambda^{'}\xi = 
\lambda \Pi^{'} \xi\;,
\quad
C\Pi^{'}\xi = 0\;.
\]
From the observability of $(C, A)$ we obtain that
$\Pi^{'} \xi =0$, implying that $\xi = 0$, 
proving the observability
$\left(H^{'}, \Lambda^{'}\right)$. 

Since  -- according to our assumption -- 
$\Im\,\Pi^{'} =  \Im\,\Pi_1 \subset < A\ \mid \ B >$ we can apply  
Theorem \ref{zeros_description2} (ii).
From this we obtain that the columns of
$H^{'}\left(zI-\Lambda^{'}\right)^{-1}$ are
in $F^{-1}(\Omega Y) + \Omega U$. Consequently,
there exists a matrix $\alpha^{'}$ such that
\[
H^{'}\left(zI-\Lambda^{'}\right)^{-1} -
H_{fzk}\left(zI-\Lambda_{fzk}\right)^{-1}\alpha^{'} 
\in \ker F \cap z^{-1}\Omega_\infty U\;.
\]
As before, from this it follows that
$\Pi^{'} = \Pi_{fzk} \alpha^{'}$, in other words
\[
\Im \Pi_1 = \Im \Pi^{'} \subset \Im \Pi_{fzk}\;,
\]
proving the maximality of $\Im \Pi_{fzk}$.

(iii) If the function $H_1\left(zI-\Lambda_1\right)^{-1}G_1$
is in $F^{-1}\left(\Omega Y\right) + \Omega U$ then
there exists a polynomial $g$ such that
\[
F(z) \left(H_1\left(zI-\Lambda_1\right)^{-1}G_1 + g(z)\right)
= q(z)
\]
is also a polynomial. Using equation (\ref{matrix_zero_eq3a})
we obtain that
\[
C\left(zI-A\right) \left( B g(z) - \Pi_1 G_1\right) =
q(z) - D g(z)\;.
\]
The observability of the pair $(C, A)$ implies that
$\left(zI-A\right)^{-1} \left( B g(z) - \Pi_1 G_1\right)$
is a polynomial, as well. Denote this by $\psi$.
Rearranging the terms we get that
\[
\Pi_1 G_1 = B g(z) - (zI-A)\psi (z)\;,
\] 
proving that $\Pi_1 G_1 \in < A \mid B >$.

Conversely, if $\Pi_1 G_1 \in < A \mid B >$ then there
exist two polynomials $g, \psi$ such that
$\Pi_1 G_1 = B g(z) - (zI-A)\psi(z)$. 
Straightforward calculation gives that
\[
\left(D+C\left(zI-A\right)^{-1}\right)
\left(H_1\left(zI-\Lambda_1\right)^{-1}G_1 + g(z)\right) =
C \psi (z) + D g(z)\;,
\]
thus $H_1\left(zI-\Lambda_1\right)^{-1}G_1 
\in F^{-1}\left(\Omega Y\right)+ \Omega U$, 
concluding the proof (iii).
\qed
\end{proof}

For later use it is worth summarizing part (i) and (ii)
in the following corollary which was proved e.g. 
partly in Theorem 2
in \cite{AL-SC-84} without the identification of
the zero directions but under more general
assumptions.

\begin{corollary}\label{cor:ZF_kerF}
Assume that the pair $\left(C, A\right)$ is
observable, and $\left(A, B\right)$ is
controllable.
Then for an observable pair $\left(H, \Lambda\right)$
the columns of 
$H\left(zI-\Lambda\right)^{-1}$ form a basis
in $Z(F) \oplus \cW\left(\text{ker} F\right)$ if
and only if
$\Im \Pi$ is maximal and $\ker \Pi = \left\{ 0\right\}$, 
where $\Pi$ (together with
$H, \Lambda$) provide a solution of
(\ref{matrix_zero_eq_max}).
\end{corollary}

Let us note that according to Lemma
\ref{lem:max_pi} the maximality of
$\Im (\Pi)$ can be expressed as
$\Im(\Pi) = \cV^{*}\left(\Sigma\right)$.

More generally, without assuming the controllability of
$(A, B)$.

\begin{corollary}\label{cor:Im_pi}
Assume that the pair $\left(C, A\right)$ is
observable.
Then for an observable pair $\left(H, \Lambda\right)$
the columns of 
$H\left(zI-\Lambda\right)^{-1}$ form a basis
in $Z(F) \oplus \cW\left(\text{ker} F\right)$ if
and only if
\[
\Im \Pi = \cV^{*}\left(\Sigma\right) 
\cap < A \mid B > \quad
\text{and}
\quad 
\ker \Pi = \left\{ 0 \right\}\;,
\] 
where $\Pi$ (together with
$H, \Lambda$) provide a solution of
(\ref{matrix_zero_eq2}).
\end{corollary}

We get immediately -- using the notation introduced above --
that
\begin{equation}\label{pi_nu_fzk}
\cV^{*}\left(\Sigma\right) \cap < A \mid B > 
= \Im \Pi_{fzk}\;.
\end{equation}

\medskip
\begin{remark}
Let us observe that even the maximal 
solution triplet 
$\left(\Pi, H, \Lambda\right)$ of equation
(\ref{matrix_zero_eq_max}) is not unique.  
Although the subspace 
$\Im (\Pi_\text{max}) = \cV^{*}\left(\Sigma\right)
=\cV^{*}(A,B,C,D)$ is
given by the {\it realization}
 of the transfer function $F$, 
so without
loss of generality we might fix a basis in it, 
determining this way the matrix $\Pi_\text{max}$, but even
for a {\it fixed} $\Pi$ the matrices $\Lambda$ and $H$
are not necessarily uniquely defined. 
Obviously, if $\Lambda_1, \Lambda_2$ and
$H_1, H_2$ are two solutions (for the same  
$\Pi$) then equation
\begin{equation}\label{nonunique_eq}
\left[
\begin{array}{c} B \\ D \end{array}
\right]
\left(H_1-H_2\right)
=
\left[
\begin{array}{c} \Pi 
\left(\Lambda_1 - \Lambda_2\right)
\\ 0
\end{array}
\right]
\end{equation}
holds. This equation will play an important role in the characterization
of $\cW (\ker F )$, as we shall see later.
\end{remark}

\bigskip

\subsubsection{The module $\cW\left(\ker F\right)$ and 
the minimal indices of $\ker F$}
\label{subsec:kerF}

To characterize the set $\cW (\ker F)$ we have to 
analyze the space $\ \ker F + \Omega U$.
The next theorem provides a 
description of this set in term of equation 
(\ref{matrix_zero_eq2}) and the set of
{\it maximal output-nulling reachability subspace} 
$\cR^{*} (\Sigma)$.

\begin{theorem}\label{thm:char_ker}
Assume that $\left(C, A\right)$ is observable.

Let $(\Lambda_1, G_1)$ be a controllable pair,
where $G_1$ is a column vector. 

The function 
\[
H_1\left(zI-\Lambda_1 \right)^{-1}G_1\ \in \  
 \left(\ker F + \Omega U\right)
\]
in and only if
there exists a solution $\Pi_1$ of the
equation
\begin{equation}\label{matrix_zero_eq4}
\left[ \begin{array}{cc} A & B \\ C & D
  \end{array}\right]
\left[ \begin{array}{c}
\Pi_1 \\ H_1 \end{array} \right]
= \left[ \begin{array}{c}  \Pi_1 \Lambda_1 \\  0 
\end{array}\right]
\end{equation}
and
\[
\Pi_1 G_1 \in 
\cC^{*}\left(\Sigma\right) 
\cap \cV^{*}\left(\Sigma\right)\;.
\]
\end{theorem}

\medskip
Note that the identification of the ``kernel indices'' to
the subspace 
$\cR^{*}\left(\Sigma\right) = 
\cC^{*}\left(\Sigma\right) 
\cap \cV^{*}\left(\Sigma\right)$ was already
proved in Theorem 5 of \cite{AL-SC-84} for observable
systems and extended in Theorem 6 to general systems.

\medskip
\begin{proof}
If $H_1\left(zI-\Lambda_1 
\right)^{-1}G_1 \in \ker F + \Omega U$
then it is obviously in 
$F^{-1} (\Omega Y) + \Omega U$ and
there exists a polynomial 
$h_0+ h_1 z + \dots + h_j z^j$
such that 
\[
g = H_1 \left(zI-\Lambda_1\right)^{-1}G_1 + 
h_0 + h_z + \dots + h_j z^j \in \ker F\;.
\] 
Applying Theorem 
\ref{zeros_description2} (i)
we get that there exists 
a matrix $\Pi_1$ such that
equation (\ref{matrix_zero_eq4}) holds.

Then
\begin{eqnarray*}
0 &=& \left[D+C\left(zI-A\right)^{-1}B\right]
\left[H_1\left(zI-\Lambda_1\right)^{-1}G_1 + 
h_0 + h_1z + \dots + h_jz^j\right] \\
&=& D\left(h_0 + h_1z + \dots + h_jz^j\right)+
C\left(zI-A\right)^{-1}
\left[
B\left(h_0 + h_1z + \dots + h_jz^j\right)
-\Pi_1 G_1\right]
\end{eqnarray*}
Thus 
$C\left(zI-A\right)^{-1}
\left[
B\left(h_0 + h_1z + \dots + h_jz^j\right)
-\Pi_1 G_1\right]$
is a polynomial. Using the observability of the pair
$\left(C, A\right)$ we get that
\[
\left(zI-A\right)^{-1}
\left[
B\left(h_0 + h_1z + \dots + h_jz^j\right)
-\Pi_1 G_1\right]
=
k_0 + k_1 z + \dots + k_{j-1} z^{j-1}
\]
is also a polynomial.

Writing up the last two equations term by term
we obtain the following set of equations
\begin{equation}\label{ker_zero_eq}
\left[
\begin{array}{cc}
A & B \\ C & D
\end{array}
\right]
\left[
\begin{array}{ccccc}
k_0 & k_1 & \dots & k_{j-1} & 0 \\
h_0 & h_1 & \dots & h_{j-1} & h_j 
\end{array}
\right]
=
\left[
\begin{array}{ccccc}
\Pi_1 G_1 & k_0 & k_1 & \dots & k_{j-1} \\
0     & 0   &  0  & \dots &   0
\end{array}
\right]
\end{equation}
In other words, if in the system
\begin{eqnarray*}
x(k+1) &=& A x(k) + B u(k) \\
y(k) &=& C x(k) + D u(k)
\end{eqnarray*}
starting from the origin the input sequence 
$h_j, h_{j-1}, \dots , h_0$ 
(in this order) is applied
then output sequence during these 
$j+1$ time instants will be zero while the state
vector in the $j+2$ step is exactly $\Pi_1 G_1$.

Thus $\Pi_1 G_1$ is in the minimal input-containing set.
Equation (\ref{matrix_zero_eq4}) 
and Lemma \ref{lem:max_pi} imply that
$\Im \left( \Pi_1\right) 
\subset \cV^{*}\left(\Sigma\right)$, 
consequently 
$\Pi_1G_1 \in \cC^{*}\left(\Sigma\right)
\cap \cV^{*}\left(\Sigma\right)$.

Conversely, 
assume that $\Pi_1$ (together with $H_1, \Lambda_1$)
provides a solution of equation (\ref{matrix_zero_eq4}),
and 
$\Pi_1G_1 \in \cC^{*}\left(\Sigma\right)
\cap \cV^{*}\left(\Sigma\right)$. Since
$\Pi_1 G_1 \in \cC^{*}\left(\Sigma\right)$ 
there exists a finite sequence
denoted by $h_0, h_1, \dots , h_j$ such that when
this is used as an input $h_j, h_{j-1}, \dots , h_0$
(in this order) then the output is zero while
the immediate next state is $\Pi_1 G_1$.

Forming the function 
\[
g(z) = H_1 \left(zI-\Lambda_1\right)^{-1}G_1 + 
h_0 + h_1 z + \dots + h_j z^j
\] 
immediate calculation gives that $F(z) g(z) = 0$.
(In these calculations equation (\ref{matrix_zero_eq4}) 
should be used, as well. )
Thus the columns of $H_1\left(zI-\Lambda_1\right)^{-1}G_1$
are in $\text{ker} F + \Omega U$, 
concluding the proof of the theorem.
\qed
\end{proof}

\begin{remark}
According to Theorem \ref{zeros_finite_kernel_thm} under the
assumptions of the observability of the pair $(C, A)$
there exists a pair $\left(H_{fzk}, \Lambda_{fzk}\right)$
such that the columns of 
$H_{fzk}\left(zI-\Lambda_{fzk}\right)^{-1}$ generate
a basis in $F^{-1}(\Omega Y) + \Omega U$. Now, if
$\Pi_{fzk}$ is given by equation (\ref{matrix_zero_eq_fzk}),
then the obvious inclusion 
$\cC^{*}\left(\Sigma\right) \subset < A\mid B >$
and Corollary \ref{cor:Im_pi} imply that
\[
\cC^{*}\left(\Sigma\right) 
\cap \cV^{*}\left(\Sigma\right) = 
\cC^{*}\left(\Sigma\right) 
\cap \Im \left(\Pi_{fzk}\right)\;.
\]
\end{remark}

According a theorem proven by 
Wyman and Sain \cite{WYM-SAI-83}
in the space $\cW (\ker F)$ the equivalence 
classes (modulo $\ \ker F 
\cap z^{-1}\Omega_\infty U\ $)
of functions
\[
\pi_- \left(z^{-l} v_j\right), 
\ j=1, \dots , \nu_j,\quad
j= 1, \dots , m
\]
form a basis, where the 
$q$-tuples $v_1, \dots , v_m$ 
define a minimal polynomial basis in $\ker F$
and
$\nu_j = \deg v_j$, $j=1, \dots , m$.

\medskip
Now we are going to characterize these 
functions in terms of special solutions
of equation (\ref{matrix_zero_eq2}).

\begin{theorem}\label{thm:char_ker2}
Let $\left(C, A\right)$ be an observable pair.
Assume that the columns of the function
$H_{fzk}\left(zI-\Lambda_{fzk}\right)^{-1}$
provide a basis in 
$Z(F)\oplus \cW\left(\text{ker} F\right)$.
Let $\Pi_{fzk}$ be the corresponding solution of
(\ref{matrix_zero_eq_fzk}). 

Consider now a maximal solution -- in terms
of $\alpha_0$ and $R_0$ -- of the equation
\begin{equation}\label{matrix_kernel_eq}
\left[\begin{array}{cc}
A & B \\ C & D 
\end{array}
\right]
\left[
\begin{array}{c}
0 \\ R_0
\end{array}
\right] =
\left[\begin{array}{c}
\Pi_{fzk} \alpha_0 \\ 0 
\end{array}
\right]
\end{equation}
(the maximality is meant in
the subspace inclusion sense for 
$\Im \ R_0$).
Then the equivalence classes in 
$\cW \left(\ker F \right)$
are determined by the functions
\[
H_{fzk}\left(zI-\Lambda_{fzk}\right)^{-1} \beta
\]
where $\beta$ is any vector in
\[
\left< \Lambda_{fzk} \, \mid \, \alpha_0
\right> = 
\Im \left(
\left[
\alpha_0, \Lambda_{fzk} \alpha_0, 
\Lambda_{fzk}^2 \alpha_0, \dots 
\right]
\right) \;.
\]
\end{theorem}

\medskip
\begin{remark}
Let us note that equation (\ref{matrix_kernel_eq})
in this theorem coincides to the equation
(\ref{nonunique_eq}) describing the non-uniqueness
of the solutions $H, \Lambda$ of (\ref{matrix_zero_eq2})
for a fixed matrix $\Pi$.
\end{remark}

\medskip
\begin{proof}
First we are going to show that the minimal polynomial basis
$v_1, \dots , v_m$ in $\ker F$ generates a
solution of (\ref{matrix_kernel_eq}). 
Set $l_0 = \max_{j=1, \dots , m} \nu_j$. Denote by
\[
R(z) 
= \left[ z^{-\nu_1} v_1, \dots , z^{-\nu_m} v_m\right]
= R_0 + R_1 z^{-1} + \dots + R_l z^{-l}
\;.
\]
Note that for any rational function $g \in \ker F$ there exists a
rational function $h$ such that $g(z) = R(z) h(z)$, and on the other hand 
$R_0 = V_h$, the highest (column) degree
coefficients matrix of the matrix-polinom
$\left[ v_1, \dots , v_m\right]$. Theorem \ref{Forney_thm} (iv)
implies that
it is of maximal column rank.

We claim that
if for a strictly proper rational $q$-tuple
$g$
there exists a polynomial $\psi$ of degree no greater
than $r$ such that 
\[
g+\psi \in \ker F
\] 
then
$\psi$ can be written as a linear combination
of the columns of $\pi_+\left(z^{s} R(z)\right)$, 
$s=0, 1, \dots , r$

In fact, as we have pointed out the elements
$\pi_- (z^{-l}v_j), l=1, \dots , \nu_j$,
$j=1, \dots , m$ form a basis in $\cW (\ker F )$.
In terms of the function $R(z)$ this implies
that the columns of 
$\pi_-\left(z^{s}R(z)\right)$, 
$s=0, 1, \dots ,l_0-1$ induce a generating system
in $\cW (\ker F )$. Thus 
\[
g(z) - \sum_{s=0}^{l_0-1} 
\pi_- \left(z^{s} R(z) c_s\right) 
\in \ker F \cap z^{-1}\Omega_\infty U\;,
\]
for some coefficients $c_s, s=0, \dots , l_0-1$.
Denote by $h(z) = \sum_{s=0}^{l_0-1} z^s c_s$.
Then
\[
\pi_- \left( R(z) h(z)\right)
+ \psi (z) \in \ker F\;.
\]
Since the degree of $R(z)$ in $z^{-1}$ is no greater than $l_0$, 
consequently
$z^{l_0} \left[
\pi_- \left( R(z) h(z)\right)
+ \psi (z)
\right]$ 
is a polynomial in $\ker F$.

According to Theorem \ref{Forney_thm} (ii)
there exist polynomials 
$\phi_1, \phi_2, \dots , \phi_m$
such that
\[
\sum_{j=1}^m v_j \phi_j = 
z^{l_0} \left[
\pi_- \left( R(z) h(z)\right)
+ \psi (z)
\right]
\]
Now the degree of the right hand side is
$l_0+r$, consequently --using again
Theorem \ref{Forney_thm} -- 
\[
\deg \phi_j + \nu_j \leq l_0 + r\;.
\]
Now
\begin{eqnarray*}
\psi (z) &=& \pi_+
\left(
z^{-l_0} \sum_{j=1}^q v_j \phi_j
\right) =
\pi_+
\left( R(z)
\left[ z^{-(l_0-\nu_1)} \phi_1, \dots , 
z^{-(l_0-\nu_m)} \phi_m
\right]
\right) \\
&=&
\pi_+
\left( R(z)\ \pi_+
\left[ z^{-(l_0-\nu_1)} \phi_1, \dots , 
z^{-(l_0-\nu_m)} \phi_m
\right]
\right)\;.
\end{eqnarray*}
Since $\deg \pi_+ \left(z^{-(l_0-\nu_j)} \phi_j\right)
\leq r$
we have obtained that $\psi$ can be written
as linear combinations of the 
columns of $\pi_+ \left(z^s R(z)\right)$,
$s=0, \dots , r$, as claimed.

\medskip
Now denote by
\[
\gamma_0 + \gamma_1 z + \dots + \gamma_r z^r = 
\pi_+
\left[ z^{-(l_0-\nu_1)} \phi_1, \dots , 
z^{-(l_0-\nu_m)} \phi_m\;.
\right]
\]
Then obviously
\begin{equation}\label{strictly_proper_part}
g(z) - \pi_- 
\left( \sum_{j=0}^r z^jR(z)\gamma_j \right) 
\in \ker F\;,
\end{equation}
as well.

\bigskip

Now for any $r = 0, 1, \dots , l_0-1$ the 
columns of
$z^r R(z)$ are in $\ker F$, thus 
\[
\pi_- \left(z^r R(z)\right) \in \ker F + \Omega U 
\subset F^{-1} (\Omega Y) + \Omega U\;.
\]
Theorem \ref{zeros_finite_kernel_thm} (i)
implies that there exist a matrix 
$\alpha_r$ such that
\begin{equation}\label{r_kernel}
H_{fzk}\left(zI-\Lambda_{fzk}\right)^{-1}\alpha_r
- 
\pi_-\left(z^r R(z) \right) \in \ker F\;.
\end{equation}

We are going to show that 
the subspace
$\Im \left( [\alpha_0, \Lambda_{fzk}\alpha_0, \dots , \Lambda_{fzk}^r\alpha_0 ]\right) $
contains the column-vectors of $\alpha_r$.
Adding the function $z^r R(z)$ we get that
\begin{equation}\label{K_r_kernel}
H_{fzk}\left(zI-\Lambda_{fzk}\right)^{-1}\alpha_r + 
R_0 z^r + R_1 z^{r-1} + \dots + R_r \ \in \ 
\ker F\;.
\end{equation}
On the other hand
\begin{multline*}
z^r \pi_-\left(R(z)\right) - 
z^r H_{fzk} \left(zI-\Lambda_{fzk}\right)^{-1}\alpha_0 \\
= z^r R(z) -  R_0 z^r - 
\left(H_{fzk}\left(zI-\Lambda_{fzk}\right)^{-1}
\Lambda_{fzk}^r \alpha_0 + H_{fzk}\Lambda_{fzk}^{r-1}\alpha_0
+ \dots +
H_{fzk}\alpha_0 z^{r-1}\right)  
\in \ker F
\end{multline*}
Taking the difference
\[
H_{fzk}\left(zI-\Lambda_{fzk}\right)^{-1}
\left(\alpha_r-\Lambda_{fzk}^r \alpha_0\right)
+\left(R_r-H_{fzk}\Lambda_{fzk}^{r-1}\alpha_0\right) + 
\dots + 
\left(R_1-H_{fzk}\alpha_0\right) z^{r-1} \in \ker F\;.
\]
In other words by adding a
polynomial of degree no greater than $r-1$
to the strictly proper rational function
$H_{fzk}\left(zI-\Lambda_{fzk}\right)^{-1}
\left(\alpha_r-\Lambda_{fzk}^r\alpha_0\right)$
a function in $\ker F$ is obtained.
Consequently, according to the previous argument
for some vectors 
$c_0, c_1, \dots , c_{r-1}$
\[
H_{fzk}\left(zI-\Lambda_{fzk}\right)^{-1}
\left(\alpha_r-\Lambda_{fzk}^r \alpha_0\right)
- 
\pi_- 
\left( \sum_{j=0}^{r-1} z^jR(z)c_j \right) \in \ker F\;.
\] 
Equation (\ref{r_kernel}) implies that
\[
H_{fzk}\left(zI-\Lambda_{fzk}\right)^{-1}
\left(\alpha_r-\Lambda_{fzk}^r \alpha_0\right)
-
H_{fzk}\left(zI-\Lambda_{fzk}\right)^{-1} \sum_{j=0}^{r-1} \alpha_j c_j
\in \ker F\;.
\]
Due to the fact that the columns of 
$H_{fzk}\left(zI-\Lambda_{fzk}\right)^{-1}$ generate a basis
in the equivalence classes defined
modulo $\ker F \cap z
^{-1}\Omega_\infty U$ we get
that
\[
\alpha_r-\Lambda_{fzk}^r \alpha_0 = 
\sum_{j=0}^{r-1} \alpha_j c_j\;.
\]
Applying this recursively the inclusion
\[
\Im \left(\alpha_r\right) \subset
\Im \left(
\left[\alpha_0, \Lambda_{fzk} \alpha_0, \dots , \Lambda_{fzk}^r\alpha_0
\right]
\right)
\]
can be derived.

Let us remark that the following 
converse statement obviously holds. If for some
vector $\beta$ the identity
\[
\beta = \sum_{j=0}^r \Lambda_{fzk}^j \alpha_0 c_j
\]
holds, then 
\[
H_{fzk}\left(zI-\Lambda_{fzk}\right)^{-1}\beta
-
\pi_- \left(\sum_{j=0}^r z^j R(z) c_j\right)
\in \ker F\;.
\]

It remains to characterize the matrix $\alpha_0$.
Let us recall that the columns of the proper rational function
$H_{fzk}\left(zI-\Lambda_{fzk}\right)^{-1}\alpha_0 
+ R_0$ are in  $\ker F$.
On the other hand -- using equation
(\ref{matrix_zero_eq2})
\begin{multline*}
\left(D+C\left(zI-A\right)^{-1}B\right)
\left(H_{fzk}\left(zI-\Lambda_{fzk}\right)^{-1}\alpha_0 + R_0\right)
\\
= 
DR_0 + C\left(zI-A\right)^{-1}
\left[ BR_0 - \Pi_{fzk}\alpha_0\right]\;,
\end{multline*}
implying that
\begin{eqnarray*}
DR_0 &=& 0 \\
BR_0 - \Pi_{fzk} \alpha_0 &=& 0\;
\end{eqnarray*}
proving that (\ref{matrix_kernel_eq}) holds for the matrices $R_0, \alpha_0$ where 
$R_0=V_h$ and $\alpha_0$ is obtained as the solution of
$H_{fzk} \left(zI - \Lambda_{fzk}\right)^{-1} \alpha_0 - \pi_- \left(R(z)\right)
\in \ker (F)$ (or as  
$R_0 + H_{fzk} \left(zI - \Lambda_{fzk}\right)^{-1} \alpha_0 \in \ker (F)$).

To prove the maximality of $\Im R_0$ observe that, 
if for some vectors $\beta$ and
$\gamma_0$ the equations
\[
D\gamma_0 = 0 \;,\quad B\gamma_0 - \Pi_{fzk} \beta =0
\]
hold then the rational function 
$H_{fzk}\left(zI-\Lambda_{fzk}\right)^{-1}\beta + \gamma_0$
is in $\ker F$,  thus the previous
argument applied for $r=0$ and $g(z) = H_{fzk}\left(zI-\Lambda_{fzk}\right)^{-1}\beta$
gives that
\[
H_{fzk}\left(zI-\Lambda_{fzk}\right)^{-1}\beta
-
\pi_-\left(R(z) c_0\right) \in \ker F
\]
holds, implying that
\[
\beta = \alpha_0 c_0\;.
\] 
(The case $\beta = 0$ corresponds 
to the situation when the constant
vector $\gamma_0$ is in the 
kernel of $F$. Note that 
$\pi_- \left(\gamma_0\right) = 0$, 
consequently in this case the corresponding
equivalence class in 
$\cW\left(\ker F\right)$ is zero.)

Thus -- fixing $\Pi_{fzk}$ -- a maximal solution of
\[
\left[\begin{array}{cc}
A & B \\ C & D 
\end{array}
\right]
\left[
\begin{array}{c}
0 \\ R_0
\end{array}
\right] =
\left[\begin{array}{c}
\Pi_{fzk} \alpha_0 \\ 0 
\end{array}
\right]
\]
should be considered (the maximality is meant in
the subspace inclusion sense for 
$\Im \left(R_0\right)$) and for any vector $\beta$ in
\[
\Im \left(
\left[
\alpha_0, \Lambda_{fzk} \alpha_0, 
\Lambda_{fzk}^2 \alpha_0, \dots 
\right]
\right) 
\]
the strictly proper rational function
\[
H_{fzk}\left(zI-\Lambda_{fzk}\right)^{-1} \beta
\]
generates an equivalence class in
$\cW \left(\ker F \right)$.
\qed
\end{proof}

\medskip
\begin{remark}
Note that an immediate consequence of 
the previous theorem that the minimal indices
$\nu_1, \dots , \nu_m$ of the minimal polynomial basis
in $\ker F$ coincide with the controllability indices
of the pair $\left(\Lambda_{fzk}, \alpha_0\right)$.

This result should be considered in parallel to 
Theorem 5 in \cite{AL-SC-84} where the minimal indices
corresponding to a minimal polynomial basis in $\ker F$
are also identified with the controllability
indices of a pair of suitably chosen 
matrices. There these matrices
are obtained using a feedback transformation. 
In addition to these Corollary 3 of the same paper
shows that these minimal indices corresponding to
the $\ker F$ are invariant under feedback transformation
and output injection, as well. (I.e. for the systems
$(A, B, C, D)$, $(A+BL, B, C+DL, D)$ and 
$(A+LC, B+LD, C, D)$ these minimal indices coincide.
\end{remark}

\medskip
\begin{corollary}\label{cor:nu_r}
Let $\left(C, A\right)$ be an 
observable pair. Assume that the columns of the
function $H_{fzk}\left(zI-\Lambda_{fzk}\right)^{-1}$
provide a basis in $Z(F) \oplus \cW (\ker F )$.
Let $\Pi_{fzk}$ be the corresponding solution
of (\ref{matrix_zero_eq_fzk}).
 
Then
\[
\cC^{*}\left(\Sigma\right) 
\cap \Im \left(\Pi_{fzk}\right) = 
\Pi_{fzk} \left< \Lambda_{fzk} 
\,\mid\, \alpha_0\right>\;.
\]
\end{corollary}

\medskip
It is worth pointing out that the following 
statement which was already present in 
\cite{AND-75} is also an immediate corollary
of the previous theorem. 
\begin{corollary}\label{cor:ker_triv}
Let $\left(C, A\right)$ be an
observable pair. Assume that the columns of the
function $H_{fzk}\left(zI-\Lambda_{fzk}\right)^{-1}$
provide a basis in $Z(F) \oplus \cW (\ker F )$.
Let $\Pi_{fzk}$ be the corresponding solution
of (\ref{matrix_zero_eq_fzk}).

Then the subspace $\cW (\ker F )$ 
is trivial if and only if
\[
\Im (\Pi_{fzk}) \cap 
\left\{
B \eta \ \mid \ D\eta = 0 
\right\} = \{ 0 \}.
\]
\end{corollary}

\begin{proof}
This is immediate from the previous theorem
giving that $\cW (\ker F) = \{ 0 \}$
if and only if the only solution 
of (\ref{matrix_kernel_eq})
is $R_0 = 0$, $\alpha_0 = 0$.
\qed\end{proof}

In some cases the following form of this corollary
can be also of use which follows immediately from
the inclusion $\Im B \subset \left< A \mid B\right>$
and the identity $\Im \Pi_{fzk} =
\Im \Pi_{max} \cap \left< A \mid B\right>$.

\begin{corollary}\label{cor:ker_triv2}
Assume that $(C, A)$ is an observable pair. Consider
a maximal solution $\left(\Pi_\text{max}, 
H_\text{max}, \Lambda_\text{max}\right)$ of
equation (\ref{matrix_zero_eq_max}).

Then the subspace $\cW (\ker F )$ is trivial if and only if
\[
\cV^{*}\left(\Sigma\right) \cap 
\left\{
B \eta \ \mid \ D\eta = 0 
\right\} =
\Im (\Pi_\text{max}) \cap 
\left\{
B \eta \ \mid \ D\eta = 0 
\right\} = \{ 0 \}.
\]
\end{corollary}

\begin{remark}\label{rem:left_inv}
Let us recall (see e.g. \cite{AL-SC-84}) that the function
$F$ is left-invertible if and only if both
$\cW(\ker F)$ and $\ker \left[\begin{array}{c}
B \\ D \end{array}\right]$ are trivial.
\end{remark}

The following corollary is also immediate from the 
previous theorem.

\begin{corollary}
Assume that $\left(C, A\right)$
is observable, $\left(A, B\right)$ is controllable. 
Then the subspace $Z (F)$ is trivial if and only if
the pair $(\Lambda_{fzk}, \alpha_0 )$ from equations
(\ref{matrix_zero_eq_fzk}), (\ref{matrix_kernel_eq})
is controllable.
\end{corollary}

\begin{remark}\label{rem:nonuniq_sol}
Let us emphasize that according to Corollary \ref{cor:Im_pi}
any pair $(H, \Lambda)$ can be used as a starting point in Theorem 
\ref{thm:char_ker2} for which the corresponding solution $\Pi$
of (\ref{matrix_zero_eq2}) satisfies that
$\ker \Pi = \left\{ 0 \right\}$ and $\Im \Pi = \cV^{*} (\Sigma ) \cap
 < A \ \mid \ B >$. For a fixed $\Pi$ with these properties the nonuniqueness 
of maximal solution solution of (\ref{matrix_kernel_eq}) is determined 
by a nonsingular matrix multiplyer from the right. Thus for a fixed $\Pi$ 
(with $\ker \Pi = \left\{ 0 \right\}$, $\Im \Pi = \cV^{*} (\Sigma ) \cap 
< A \ \mid \ B >$) all solution of (\ref{matrix_zero_eq_fzk}) and 
(\ref{matrix_kernel_eq}) can be described as $(H+R_0 \beta, \Lambda + \alpha_0 \beta )$,
and $(R_0 \gamma, \alpha_0 \gamma)$ where $\beta$ is an arbitrary matrix,
$\gamma$ is an arbitrary nonsingular matrix,
$(H, \Lambda)$ and $R_0, \alpha_0$ are particular solutions of these equations.

\end{remark}

\begin{remark}\label{rem:fzk_vs_max}
Furthermore, the identity
$\Im \Pi_{fzk} = \Im \Pi_{\text{max}} \cap 
< A \mid B >$ implies that $R_0$ and $\alpha_0$
can be determined starting with the matrix $\Pi_{max}$ instead of $\Pi_{fzk}$.
In fact, consider a maximal solution $(\Pi_{max}, H_{max}, \Lambda_{max})$ 
of (\ref{matrix_zero_eq2}) with $\ker\left(\Pi_{max}\right) = 
\left\{0 \right\}$ and using $\Pi_{max}$ consider a maximal solution 
$\widetilde{R_0}, \widetilde{\alpha_0}$ of
\begin{equation}\label{matrix_kernel_eq2}
\left[\begin{array}{cc} A & B \\ C & D
\end{array}\right]
\left[\begin{array}{c} 0 \\ \widetilde{R_0}
\end{array}\right]
=
\left[\begin{array}{c}
\Pi_\text{max} \widetilde{\alpha_0} \\ 0
\end{array}\right]\;
\end{equation}
assuming that $\ker \left(\widetilde{R_0}\right) = \left\{0\right\}\;,$
where the maximality is meant in the subspace inclusion sense
for $\Im (\Pi_{max})$ and $\Im (\widetilde {R_0})$.

Multiplying $\Pi_{max}$ from the right by a nonsingular matrix and $\widetilde{\alpha_0}$
from the right by its inverse we might assume that $\Pi_{max}$ has the following form
$\Pi_{max} = [ \Pi_{fzk}, \Pi^{'}]$. Partition $\widetilde {\alpha_0}$ accordingly:
$\widetilde{\alpha_0} = \left[\begin{array}{c} \alpha_1\\ \alpha_2\end{array}\right]$
Now 
\[
B \widetilde{R_0} = \Pi_{fzk} \alpha_1 + \Pi^{'} \alpha_2\,.
\]
The obvious inclusion $\ \Im B \subset < A \mid B >$  
gives that the columns of the matrix above should be in 
$\left\{< A \mid B > \cap\ \Im (\Pi_{max})\right\}=
\Im (\Pi_{fzk})$. Thus $\Pi^{'} \alpha_{2} = 0$. i.e. $\alpha_{2}=0$.
The equations $B \widetilde{R_0} = \Pi_{fzk} \alpha_1$, 
$D\widetilde{R_0} = 0$ and the maximality of the solution $\widetilde{R_0}, 
\widetilde{\alpha_0}$ implies that after a multiplication from the right by a 
nonsingular matrix we can achieve that 
\[
\widetilde{R_0} = R_0\,,\quad \alpha_1 = \alpha_0\,.
\] 

Moreover, applying the same nonsingular matrix multiplication from the right
to the equation (\ref{matrix_zero_eq_max}) we might again assume that
$\Pi_{max} = [ \Pi_{fzk}, \Pi^{'} ]$. Partition $H_{max}$ and $\Lambda_{max}$
accordingly. 
\[
H_{max} = [ H_1, H_2 ]\,, \quad 
\Lambda_{max} = \left[\begin{array}{cc} 
\Lambda_{11} & \Lambda_{12} \\
\Lambda_{21} & \Lambda_{22}
\end{array}\right]
\]
We obtain that
\[
A \Pi_{fzk} + B H_1 = \Pi_{fzk} \Lambda_{11} + \Pi^{'} \Lambda_{21}\;.
\]
But equation (\ref{matrix_zero_eq_fzk}) shows that the columns of
$A\Pi_{fzk}$ are in the subspace generated by $\Im B$ and $\Im \Pi_{fzk}$. 
Similar argument as before gives that $\Pi^{'}\Lambda_{21} = 0$, 
i.e. $\Lambda_{21} = 0$. 
It follows that $H_1, \Lambda_{11}$ provide a solution of (\ref{matrix_zero_eq_fzk}),
consequently they can be denoted by $H_{fzk}, \Lambda_{fzk}$.

Let us observe that 
\begin{multline}\label{K0_max_fzk}
R_0 + H_{max} (zI - \Lambda_{max})^{-1} \widetilde\alpha_0 = \\
R_0 + \left[ H_{fzk}, H_2 \right]
\left(z \left[\begin{array}{cc} I & 0 \\ 0 & I \end{array}\right] -
\left[\begin{array}{cc} \Lambda_{fzk} & \Lambda_{12} \\
0 & \Lambda_{22}\end{array}\right]\right)^{-1} 
\left[\begin{array}{c} \alpha_0 \\ 0 \end{array}\right] = \\
R_0 + H_{fzk}(zI - \Lambda_{fzk})^{-1} \alpha_0\;.
\end{multline}
\end{remark}

\bigskip
\begin{remark}\label{rem:K0}
Let us introduce the notation:
\[
K_0(z) = R_0 + H_{fzk}
\left(zI-\Lambda_{fzk}\right)^{-1}\alpha_0\;.
\]
As we have already pointed out the columns of this function
are in the kernel of $F$. Moreover, the
columns of $\pi_-\left(z^r K_0(z)\right)$, $r\geq 0$
generate $\cW(\ker F)$.

Note that the realization above of $K_0$ 
is -- in general -- non-minimal. Although the 
observability of the pair $(C, A)$ implies
that $\left(H_\text{fzk}, \Lambda_\text{fzk}\right)$
is observable, as well, the controllability of
$\left(\Lambda_\text{fzk}, \alpha_0\right)$
in general does not hold. 

Let us emphasize that the function $K_0$ is defined via fixing 
a particular solution of (\ref{matrix_zero_eq_fzk}) and
also of (\ref{matrix_kernel_eq}). As we have pointed out
in Remark \ref{rem:nonuniq_sol}
all solution can be obtained from these.
Let us observe that -- using the notations from Remark \ref{rem:nonuniq_sol}
\[
\left[
R_0 \gamma + (H_{fzk} + R_0\beta)
\left(zI - (\Lambda_{fzk} + \alpha_0 \beta)\right)^{-1} \alpha_0 \gamma 
\right]
\gamma^{-1} 
\left(I- \beta (zI - \Lambda_{fzk})^{-1}\alpha_0 \right)
=
K_0 (z) \;,
\]
where the function $\gamma^{-1} 
\left(I- \beta (zI - \Lambda_{fzk})^{-1}\alpha_0 \right)$ is proper
with proper inverse.
\end{remark}

\medskip
The following proposition explicitly shows 
how the columns of the function $K_0$ 
generate the kernel of $F$.

\begin{proposition}\label{prop:gen_K0}
Assume that for the proper rational
$q$-tuple $g$ the identity 
\[
F(z) g(z) = 0
\]
holds. Then there exists a proper rational function
$h(z)$ such that
\[
g(z) = K_0(z) h(z)\;,
\]
i.e. the columns of $K_0$ generate the kernel
of $F$ over the field of rational functions.
\end{proposition}

\begin{proof}
Assume that the realization of $F$ given by
\[
F(z) \sim 
\left(
\begin{array}{c|c} A &   B
\\
\hline \rule{0cm}{.42cm}
   C & D
\end{array}
\right)
\]
where $\left(C, A\right)$ is an observable pair,
furthermore the realization of $g$ is given
by
\[
g(z) \sim 
\left(
\begin{array}{c|c} \lambda &   \beta
\\
\hline \rule{0cm}{.42cm}
   \gamma & \delta
\end{array}
\right)\;,
\]
where $\left(\lambda, \beta\right)$ is a
controllable pair.

Evaluating equation $F g = 0$ 
at infinity we obtain that
$D \delta =0$.  Now, the equation can be written in the
following form
\[
\left(D+C\left(zI-A\right)^{-1}B\right) 
\gamma \left(zI-\lambda\right)^{-1} \beta = 
- C\left(zI-A\right)^{-1}B\delta\;.
\]
According to Proposition \ref{product_equation}
there exists a solution $\rho$ of the equation
\[
\left[\begin{array}{cc} A & B \\ 
C & D \end{array}\right] 
\left[\begin{array}{c} \rho \\ \gamma
\end{array}\right] =
\left[\begin{array}{c} 
\rho \lambda \\ 0 
\end{array}\right] 
\]
for which $\rho \beta  = B \delta$ holds. This latter
together with $D\delta = 0$ can be written
in the form
\[
\left[\begin{array}{cc} A & B \\ C & D 
\end{array}\right]
\left[\begin{array}{c} 0 \\ \delta \end{array}\right]
=
\left[\begin{array}{c} \rho \beta \\ 0 \end{array}\right]\;.
\]

Using the maximality of $\Pi_\text{max}$ and
$R_0$ we get that
\[
\rho = \Pi_\text{max} \xi\;, \quad \text{and consequently }
\quad
\delta = R_0 \eta
\]
for some matrix $\xi$ and some vector $\eta$. 
Substituting into $\rho \beta = B\delta $ we obtain that
$\Pi_{max} \xi \beta = B R_0 \eta = \Pi_{max} \widetilde\alpha_0 \eta$
inplying that
\[
\xi \beta = \widetilde\alpha_0 \eta\;.
\]
 
Multiplying the equation (\ref{matrix_zero_eq_max})
from the right by $\xi$ and 
(\ref{matrix_kernel_eq2}) from the right by 
$\eta$ and taking the differences with the
previous equations we arrive at the 
following equations:
\[
\left[\begin{array}{cc} A & B \\
C & D\end{array}\right]
\left[\begin{array}{c} 0 \\
\gamma - H_\text{max}\xi  \end{array}\right]
\left[\begin{array}{c} 
\Pi_\text{max}
\left(\xi \lambda - \Lambda_\text{max} \xi\right) \\
0 \end{array}\right]
\]
implying that there exists 
a matrix $\zeta$ for which
\begin{eqnarray}
\gamma - H_\text{max} \xi &=& R_0 \zeta \\
\xi \lambda - \Lambda_\text{max} \xi 
&=& \widetilde\alpha_0 \zeta
\end{eqnarray}
hold.

Now, define
\[
h (z) = \eta + \zeta 
\left(zI -\lambda\right)^{-1}\beta\;.
\]
Then straightforward calculation -- using the identity
(\ref{K0_max_fzk}) -- gives that
\[
K_0 (z) h(z) = g(z)\;,
\]
concluding the proof of the proposition.
\qed\end{proof}

\medskip
\begin{remark}\label{rem:partition}
Let us return to the non-uniqueness of the solution of 
equation (\ref{matrix_zero_eq_fzk}). Assume that the
pair $(C, A)$ is observable. Consider a
solution $\left(\Pi_{fzk}, H_{fzk}, \Lambda_{fzk}\right)$
of (\ref{matrix_zero_eq_fzk}) for which
$\Im \Pi_{fzk} = \cV^{*}\left(\Sigma\right) 
\cap \left< A \mid B\right>$,
$\ker \Pi_{fzk} = \left\{ 0 \right\}$ and a
maximal solution $(\alpha_0, R_0)$ 
of (\ref{matrix_kernel_eq}). The maximality of this 
solution implies that -- fixing the matrix $\Pi_{fzk}$ --
{\it any} solution of (\ref{matrix_zero_eq_fzk})
has the form 
\[
\left(\Pi_{fzk}, H_{fzk}+R_0\beta, 
\Lambda_{fzk}+\alpha_0\beta\right)\;,
\]
where $\beta$ is an arbitrary 
matrix (of appropriate size).

\medskip
Without loss of generality we might assume 
that the matrices
$\Lambda_{fzk}$ and $\alpha_0$ are of the form
\begin{equation}\label{special_form1}
\Lambda_{fzk} = \left[\begin{array}{cc} 
\Lambda_k & \Lambda_{kf} \\
0 & \Lambda_{f}
\end{array}
\right]\;, \quad
\alpha_0 = \left[
\begin{array}{c} \alpha_k \\ 0
\end{array}
\right]\;,
\end{equation}
where the pair 
$\left(\Lambda_k, \alpha_k\right)$ is
controllable. Accordingly,
\begin{equation}\label{special_form2}
\Pi_{fzk} = \left[ 
\Pi_k, \ \Pi_f
\right]\;, \quad
H_{fzk} = \left[ 
H_k, \ H_f
\right]\:,
\quad\beta = 
\left[ \beta_1, \ \beta_2 \right]
\;.
\end{equation}
(Note, that this transformation does not affect
the choice of $R_0$.)

Observe that
\begin{equation}\label{pi_1}
\left[\begin{array}{cc}
A & B \\ C & D
\end{array}
\right]
\left[\begin{array}{c}
\Pi_k \\ H_k
\end{array}\right] =
\left[\begin{array}{c}
\Pi_k \Lambda_k \\ 0
\end{array}\right]\;,
\end{equation}

and
\begin{equation}\label{alpha1}
\left[\begin{array}{c}
B \\ D
\end{array}\right]
R_0 = 
\left[\begin{array}{c}
\Pi_k \alpha_k \\ 0
\end{array}\right]\;.
\end{equation}
\end{remark}

For later use it is worth 
pointing use that the controllability
of the pair $\left(\Lambda_k, 
\alpha_k\right)$ and Corollary 
\ref{cor:nu_r}
imply that the identity
\begin{equation}\label{eq_im_pi1}
\Im \left(\Pi_k\right) =
\cC^{*}\left(\Sigma\right) \cap 
\Im \left(\Pi_{fzk}\right) =
\cC^{*}\left(\Sigma\right) \cap 
\cV^{*}\left(\Sigma\right) =
\cR^{*}\left(\Sigma\right)
\end{equation}
holds.

Notice also that
\[
\Lambda_{fzk} + \alpha_0 \beta =
\left[\begin{array}{cc} \Lambda_k + \alpha_k \beta_k &
\Lambda_{kf}+\alpha_k\beta_f \\
0 & \Lambda_f
\end{array}\right]
\]
and
\[
H_{fzk} + R_0 \beta = 
\left[ H_k + R_0 \beta_k , H_f + R_0\beta_f\right]\;.
\]
Consequently, invoking Theorem \ref{thm:char_ker2}
we get that the columns of the function
\[
\left(H_k+R_0\beta_k\right)
\left(zI-\left(\Lambda_k+\alpha_k\beta_k\right)\right)^{-1}
\]
generate a basis in $\cW\left(\ker F\right)$,
while (factoring out $\cW\left(\ker F\right)$)
the columns of
\[
\left(H_f+R_0\beta_f\right)
\left(zI- \Lambda_f\right)^{-1}
\]
generate a basis in $Z(F)$.

This latter observation justifies the following definition
\begin{definition}
Under the assumptions of the previous remark the matrix
$\Lambda_f$ is called finite zero matrix of the function $F$.
Its eigenvalues are the so-called finite 
(or transmission) zeros of $F$.

The eigenvalues of $\Lambda_k$ (or of
$\left(\Lambda_k + \alpha_k\beta_k\right)$) are 
called virtual zeros of
$F$.
\end{definition}

The expression {\it virtual zero} refers to 
the fact that choosing the
matrix
$\beta$ in an appropriate way the eigenvalues of 
$\Lambda_k + \alpha_k\beta_k$ can be
moved around in the complex plane.

\begin{remark}\label{rem:Kbeta}
Let us introduce the notation
\[
K_\beta (z) = R_0 + \left(H_{fzk}+R_0 \beta\right)
\left(zI - \left(\Lambda_{fzk}+\alpha_0\beta\right)\right)^{-1}
\alpha_0\;.
\]
This function has obviously the same properties as $K_0$, 
namely, its columns are in the kernel of $F$, and
the columns of $\pi_-\left(z^r K_\beta (z)\right)$, 
$r\geq 0$ generate $\cW(\ker F)$.
\end{remark}

\subsubsection{Choosing $K_\beta$ as a tall inner-function}
\label{subsec:K_inner}

\medskip
The following theorem shows that the rational
function 
$K_0(z) = R_0 + H_{fzk} 
\left( zI-\Lambda_{fzk}\right)^{-1}\alpha_0$
being equivalent to $R(z)$ and playing 
an important role in the proof 
of Theorem \ref{thm:char_ker2} can be chosen
to be a tall inner function. To this aim
we are going to use the property that 
for a fixed matrix $\Pi$ all solutions
of (\ref{matrix_zero_eq_fzk}) can be given in
the form $\Lambda+\alpha_0\,\beta$,
$H+R_0\,\beta$, where $\beta$ is arbitrary.

\begin{theorem}\label{inner_fnc_kernel}
Let $\left(C, A\right)$ be an observable pair.
Assume that the columns of the function
$H_{fzk}\left(zI-\Lambda_{fzk}\right)^{-1}$
provide a basis in $Z(F)\oplus \cW\left(\text{ker} F\right)$.
Let $\Pi_{fzk}$ be the corresponding solution of
(\ref{matrix_zero_eq_fzk}). 
Consider a maximal solution -- in terms
of $\alpha_0$ and $R_0$ -- of the equation
(\ref{matrix_kernel_eq}) assuming 
-- w.l.o.g. -- that
the column-vectors of the matrix $R_0$ are orthonormal.
Then there exists a matrix $\beta$ such that the
function
\[
K_\beta (z) = R_0 + \left( H_{fzk}+ R_0 \beta \right)
\left( zI - \left(\Lambda_{fzk} + \alpha_0 \beta \right)
\right)^{-1}
\alpha_0
\]
is a tall inner (in continuous time sense) 
function.
\end{theorem}

\begin{proof}
We might assume that the matrices are partitioned according
to (\ref{special_form1}) and (\ref{special_form2})
Then
\[
K_\beta(z) = R_0 + \left( H_k+ R_0 \beta_k \right)
\left( zI - \left(\Lambda_k + \alpha_k \beta_k \right)
\right)^{-1}
\alpha_k\;.
\]
(Especially, the value of $\beta_f$ has no effect
on the function $K_\beta$.) 

\medskip
Obviously the equations
\begin{eqnarray}
\sigma \left(\Lambda_k+\alpha_k\beta_k\right)
+
\left(\Lambda_k+\alpha_k\beta_k\right)^{*}\sigma
+
\left(H_k+R_0 \beta_k\right)^{*}
\left(H_k+R_0 \beta_k\right) &=& 0 
\label{inner_eq1}\\
\alpha_k^{*} \sigma + R_0^{*} 
\left(H_k+R_0 \beta_k\right) &=& 0 
\label{inner_eq2}\\
R_0^{*} R_0 &=& I
\label{inner_eq3}
\end{eqnarray}
imply the equation $K_\beta^{*}(z) K_\beta(z) = I$\;.

Due to the fact that the columns of $R_0$ are
orthonormal the third equation trivially holds.

The second equation gives that
\begin{equation}\label{beta_def}
\beta_k = -\alpha_k^{*} \sigma  - R_0^{*} H_k\;.
\end{equation}
Substituting this expression into the first equation
the following Riccati-equation
\begin{equation}\label{ric}
\sigma \left(\Lambda_k - \alpha_k R_0^{*} H_k\right)
+
\left(\Lambda_k - \alpha_k R_0^{*} H_k\right)^{*} \sigma
- \sigma\ \alpha_k \ \alpha_k^{*}\ \sigma 
+
H_k^{*} \left(I-R_0 R_0^{*}\right) H_k
=0
\end{equation}
is obtained. 

The controllability of the pair 
$(\Lambda_k, \alpha_k)$
implies that equation (\ref{ric}) 
has a unique
positive semidefinite solution.

Next we prove that any solution $\sigma$ 
of this equation is invertible.
Obviously, if $\xi \in \ker \sigma$ then 
-- multiplying by $\xi^{*}$ from the left and by $\xi$
from the right the equation
\[
\left(I-R_0R_0^{*}\right) H_k \xi = 0
\]
is obtained. Now multiplying only from the right
by $\xi$ we get that $\ker \sigma$ is 
$\left(\Lambda_k - \alpha_k R_0^{*} H_k\right)$-invariant.
Choosing $\xi$ to be an eigenvector of this matrix
\[
\left(\Lambda_k - \alpha_k R_0^{*} 
H_k\right)\xi = \lambda \xi\;,
\]
and using (\ref{pi_1}) and 
(\ref{alpha1}) we obtain that
\[
\left[\begin{array}{cc}
A & B \\ C & D
\end{array}
\right]
\left[\begin{array}{c}
\Pi_k \xi \\ (H_k-R_0R_0^{*}H_k)\xi
\end{array}\right]
=
\left[\begin{array}{c}
\Pi_k (\Lambda_k - \alpha_k R_0^{*}H_k)\xi \\
0
\end{array}\right] =
\left[\begin{array}{c}
\lambda \Pi_k \xi \\ 0
\end{array}\right]\;.
\]
Invoking the observability of the
pair $(C, A)$ we get that $\Pi_k \xi = 0$.
But according to our assumption the column vectors of
$\Pi_k$ are linearly independent, thus $\xi = 0$,
proving the invertibility of $\sigma$.

It remains to prove the 
analyticity of $K_\beta$ on the right
half plane. If $\xi$ is an eigenvector of 
$\Lambda_k + \alpha_k\beta_k$ i.e.
\[
\left(\Lambda_k + \alpha_k\beta_k\right)
\xi = \lambda \xi\;,
\]
then
\[
2 \Re \lambda \ \xi^{*} \sigma \xi +
\xi^{*} \sigma \alpha_k\alpha_k^{*}\sigma \xi
+
\xi^{*} H_k^{*}\left(I-R_0R_0^{*}\right)H_k\xi = 0\;. 
\]
Thus $\Re \lambda \leq 0$. If $\Re \lambda = 0$, then
$\left(I-R_0R_0^{*}\right)H_k\xi = 0$ and
$\alpha_k^{*}\sigma \xi = 0$. 
Using again equations
(\ref{pi_1}) and 
(\ref{alpha1}) we obtain that
\[
\left[\begin{array}{cc}
A & B \\ C & D
\end{array}
\right]
\left[\begin{array}{c}
\Pi_k \xi \\ H_k\xi-R_0
\left(\alpha_k^{*}\sigma +R_0^{*}H_k\right)\xi
\end{array}\right]
=
\left[\begin{array}{c}
\Pi_k \left(\Lambda_k - \alpha_k 
\left(\alpha_k^{*}\sigma + R_0^{*}H_k\right)\right)\xi \\
0
\end{array}\right] =
\left[\begin{array}{c}
\lambda \Pi_k \xi \\ 0
\end{array}\right]\;.
\]
Invoking the observability of the
pair $(C, A)$ and $\ker \Pi_k =
\left\{ 0 \right\}$ we get -- similarly as before --
$\xi=0$. Thus the matrix 
$\Lambda_k + \alpha_k\beta_k$ is
asymptotically stable (in continuous time sense), 
proving that the function 
$K_\beta$ is inner.
\qed\end{proof}

The inner function $K_\beta$ will be called as 
``right-kernel'' inner function.

\bigskip
\subsubsection{Zero structure of a tall inner function}

As an immediate application of Theorems 
\ref{zeros_finite_kernel_thm}
and \ref{thm:char_ker2} let us consider the zero structure
of a tall inner function. In order to emphasize that in this
subsection a special case is considered let us denote this
tall inner function by 
$Q(z) = D_Q+C_Q\left(zI-A_Q\right)^{-1}B_Q$. 
Assume that the pair $\left(C_Q, A_Q\right)$ is
observable and all the eigenvalues of the
matrix $A_Q$ have negative real part.
Consider 
a square inner extension of $Q$ in the form 
$\left[ Q, \widetilde Q\right]$ assuming that a realization of
$\widetilde Q$ is given as
$\widetilde Q (z) = 
\widetilde D_Q + C_Q\left(zI-A_Q\right)^{-1} \widetilde B_Q$.

As it is well-known this extension can be obtained
in the following way.
Consider the solution $P$ of
the Lyapunov equation
\begin{equation}\label{eqn_ljap}
PA_Q + A_Q^{*}P + C_Q^{*}C_Q = 0\;.
\end{equation}
$P$ is uniquely determined and positive definite.
The matrix $\widetilde  D_Q$ provides a unitary 
extension of $D_Q$, i.e. $\left[ D_Q, \widetilde D_Q\right]$
is a unitary matrix. (In other words, the orthonormal 
vectors formed by the columns of $D_Q$ are extended 
to an orthonormal basis.)
Then 
\[
\widetilde B_Q = - P^{-1}C_Q^{*} \widetilde D_Q\;.
\]
Note that $B_Q = -P^{-1}C_Q^{*} D_Q$.
In other words the identity
\[
\left[ 
B_Q, \widetilde B_Q
\right]
\left[
\begin{array}{c}
D_Q \\ \widetilde D_Q
\end{array}
\right]^{*} +
P^{-1}C_Q^{*} = 0
\]
holds.

\begin{proposition}\label{prop:zeros_inner}
Let $Q$ be an tall inner function (in continuous time sense)
with the realization above. Assume that the pair
$\left( C_Q, A_Q\right)$ is observable, and
all the eigenvalues of $A_Q$ have negative real part.

Consider a square inner extension 
$\left[Q, \widetilde Q\right]$ of $Q$ with the realization above. 
Then
\begin{itemize}
\item[(i)]
the maximal solution of (\ref{matrix_zero_eq_max})
is given by the triplet $\left(\Pi, H, \Lambda\right)$
where $ P \Im \left(\Pi\right)$ is the 
the orthogonal complement of 
the reachability subspace
$\left< A_Q \mid \widetilde B_Q\right)$.

\[
H = - D_Q^{*}C_Q\Pi\]
and $\Lambda$ is determined by
the equation
\[
\left(A_Q-B_QD_Q^{*}C_Q\right)\Pi = \Pi \Lambda\;.
\]
The matrix $\Lambda$ is the finite zero matrix of $Q$.

\item[(ii)]
The module $\cW\left( \ker Q\right)$ is trivial.
\item[(iii)]
On the subspace $\Im \left(\Pi\right)$ 
the matrices $-P^{-1}A_Q^{*}P$ and 
$\left(A_Q-B_QD_Q^{*}C_Q\right)$ coincide.
\end{itemize}
\end{proposition}

\begin{proof}
According the Theorem \ref{zeros_finite_kernel_thm}
first a maximal solution of equation
\begin{equation}\label{eqn:matrix_zero_inner}
\left[\begin{array}{cc} A_Q & B_Q \\
    C_Q & D_Q
    \end{array}
    \right]
\left[\begin{array}{c}
\Pi \\ H\end{array}
\right]
=
\left[ \begin{array}{c}
\Pi \Lambda \\ 0 \end{array}\right]
\end{equation}
should be considered.

Multiplying the second equation from the left by 
$D_Q^{*}$ we get that
$H = - D_Q^{*} C_Q \Pi$.
Substituting this values into the first and 
the second equation we arrive at
the following equations:
\begin{eqnarray*}
\left( A_Q - B_Q D_Q^{*} C_Q \right) \Pi &=&
\Pi \Lambda \;,\\
\left(I-D_Q D_Q^{*}\right) C_Q \Pi &=& 0\;.
\end{eqnarray*}

Now $B_Q D_Q^{*} C_Q = - P^{-1} C_Q^{*} D_QD_Q^{*} C_Q$,
consequently the Lyapunov-equation (\ref{eqn_ljap})
above can be written as
\[
P \left(A_Q-B_QD_Q^{*}C_Q\right) + A_Q^{*} P 
+ C_Q^{*} \left(I-D_QD_Q^{*}\right) C_Q = 0\;.
\]
Multiplying form the right by $\Pi$ and from
the left by $P^{-1}$ we obtain
that
\[
P^{-1} A_Q^{*} P \Pi = - \left(A_Q-B_QD_Q^{*}C_Q\right)\Pi 
= - \Pi \Lambda\;. 
\]
Thus the subspace $\Im \left(\Pi \right)$ should be
 $\left(A_Q-B_QD_Q^{*}C_Q\right)$-invariant
and on it the matrices
$-P^{-1}A_Q^{*}P$ and 
$\left(A_Q-B_QD_Q^{*}C_Q\right)$ coincide.
(This proves (iii).)

The identities 
\begin{eqnarray*}
\left(I-D_Q D_Q^{*}\right) C_Q &=&
\widetilde D_Q \widetilde D_Q^{*} C_Q = 
-\widetilde D_Q \widetilde B_Q^{*} P\\
A_Q^{*} P \Pi &=& - P \Pi \Lambda
\end{eqnarray*}
imply that the subspace
$P \Im \left(\Pi \right)$ should be orthogonal
to the reachability subspace
$\left< A_Q \mid \widetilde B_Q\right>$.

Conversely, consider 
orthogonal complement of 
$\left< A_Q\mid \widetilde B_Q\right>$ and choose the 
matirx $\Pi$ is such a way that the columns of $P \Pi$ span the
this subspace. In this case
then 
\begin{eqnarray*}
\Pi^{*} P \widetilde{B_Q} &=& 0 \\
\Pi^{*} P A_Q &=& -\Lambda \Pi^{*} P
\end{eqnarray*}
for some matrix $\Lambda$. Using the equations above
we get that
\[
\left(I-D_Q D_Q^{*}\right) C_Q \Pi = 0\;,
\]
and from the Lyapunov-equation (\ref{eqn_ljap}) we obtain that
\[
\left(A_Q-B_QD_Q^{*}C_Q\right)\Pi 
=  \Pi \Lambda
\]

Thus defining 
$H= - D_Q^{*} C_Q \Pi$ we get that 
$\left(\Pi, H, \Lambda\right)$ provide
a solution of (\ref{eqn:matrix_zero_inner}),
proving the first part of (i).

To identify the corresponding $\Lambda$ as
the finite zero matrix -- using 
Corollary \ref{cor:Im_pi} and Remark
\ref{rem:partition} -- we have to
prove that $\left< A \mid B\right> 
\supset \Im \left(\Pi\right)$ and
$\cW\left( \ker Q\right)$ is trivial.

To this aim first consider solutions
of the Lyapunov-equations
\begin{eqnarray*}
A_QP_1 + P_1A_Q^{*} + B_QB_Q^{*} &=& 0 \;,\\
A_QP_2 + P_2A_Q^{*} + \widetilde B_Q \widetilde B_Q^{*} &=& 0\;.
\end{eqnarray*}
Then invoking that $\left[Q, \widetilde Q\right]$
is a square inner function we get that
\begin{equation}\label{eqn:contr_gram}
P_! + P_2 = P^{-1}\;.
\end{equation}
The kernel of $P_2$ determines the orthogonal
complement of the reachability subspace
$\left< A\mid \widetilde B\right>$, while the
image of $P_!$ gives 
$\left< A\mid B\right>$. Now --as we have seen -- 
for the
maximal solution of (\ref{eqn:matrix_zero_inner})
the identity 
\[
\Im \left(P\Pi\right) = \ker \left( P_2\right)
\]
holds. In other words
\[
\Im \left(\Pi\right) = P^{-1} \ker \left(P_2\right)\;. 
\]
But the equation (\ref{eqn:contr_gram}) implies that
if $\xi \in \ker \left(P_2\right)$ then
$\xi = P P_1 \xi$, thus
\[
\Im \left(\Pi\right) = P^{-1} \ker \left(P_2\right) 
\subset \Im \left(P_1\right)
=
\left< A \mid B \right>\;.
\]

Now consider a maximal solution of
(\ref{eqn:matrix_zero_inner}) and solve the 
equation
\[
\left[\begin{array}{c} B_Q \\ D_Q
\end{array}\right] R_0 =
\left[\begin{array}{c} \Pi \alpha \\ 0
\end{array}\right]\;.
\]
But the identiy $D_Q^{*} D_Q = I$ implies
that $R_0 = 0$, i.e. according
to Theorem \ref{thm:char_ker2} the
module $\cW \left( \ker Q \right)$ is trivial,
proving (iii) and finishing the proof of (i), 
thus
concluding the proof of the proposition.
\qed\end{proof}

\begin{remark}
Let us point out two special cases
of the proposition
above. 

\begin{itemize}
\item[(i)]
The finite zero module $Z(Q)$ is trivial,
if the pair
$\left( A_Q, \widetilde B_Q\right)$ is
controllable,
\item[(ii)]
The finite zero matrix of $Q$ is given
by $A_Q-B_QD_Q^{*}C_Q$, if $Q$ is a square inner function.
\end{itemize}
\end{remark}

\bigskip
\subsubsection{Eliminating $\cW\left(\ker F\right)$
via factorization}
\label{sct:elim_ker}

\medskip
The ``right-kernel'' inner function 
$K_\beta$ constructed in Theorem 
\ref{inner_fnc_kernel} is a tall 
inner function. 
Consider its square inner extension. 
Straightforward
computation gives that the function
\begin{eqnarray}
K_{\beta , ext} &=&
\left[ K_\beta, L_\beta\right]
\label{eq:Kinner}\\
&=& \left[ R_0, L_0\right]
+ \left(H_k+R_0\beta_k\right)
\left(zI- \left(\Lambda_k + 
\alpha_k\beta_k\right)\right)^{-1}
\left[ \alpha_k, 
-\sigma^{-1}\left(H_k+R_0\beta_k\right)^{*}L_0
\right]\nonumber
\end{eqnarray}
-- where the matrix $L_0$ is chosen in such a way
that the matrix $\left[ R_0, L_0\right]$
be unitary, and $\sigma$ is the positive definite
solution of the Riccati-equation (\ref{ric}) 
-- is a
square inner function. 

\begin{remark}
Let us
observe that Proposition \ref{prop:zeros_inner}
implies (using that the pair
$\left( \Lambda_k, \alpha_k\right)$ is reachable)
that the finite zero module 
$Z\left(L_\beta\right)$ and the
kernel module 
$\cW\left(\ker L_\beta\right)$
of $L_\beta$ is trivial.
\end{remark}

\medskip
Now define the function $F_{\text{\bf r}}$ as follows.
\begin{equation}\label{f1_def}
F_{\text{\bf r}} = F L_\beta\;.
\end{equation}
Then $F_{\text{\bf r}} L_\beta^{*} = F L_\beta L_\beta^{*} = 
F \left(K_\beta K_\beta^{*} 
+ L_\beta L_\beta^{*}\right) = F$,
using that $F K_\beta = 0$ and 
$K_\beta K_\beta^{*} + L_\beta L_\beta^{*} = I$.

The following theorem essentially shows that 
$Z(F_{\text{\bf r}}) = Z(F) \oplus \cW \left(\ker F\right)$
(they are isomorphic as vector spaces).

\begin{theorem}\label{thm:f1}
Let $\left(C, A\right)$ be an observable pair.
Assume that the columns of the function
$H_{fzk}\left(zI-\Lambda_{fzk}\right)^{-1}$
provide a basis in $Z(F)\oplus \cW\left(\text{ker} F\right)$.
Let $\Pi_{fzk}$ be the corresponding solution of
(\ref{matrix_zero_eq_fzk}). 
Consider a maximal solution -- in terms
of $\alpha_0$ and $R_0$ -- of the equation
(\ref{matrix_kernel_eq}) assuming 
-- w.l.o.g. -- that
the column-vectors of the matrix $R_0$ are orthonormal
and the matrices are partitioned according
to (\ref{special_form1}) and 
(\ref{special_form2}).

Consider the function $F_{\text{\bf r}}$ defined in
(\ref{f1_def}).  
Then
\begin{itemize}
\item[(i)]
$F_{\text{\bf r}}$ has the following (in general non-minimal)
realization
\begin{equation}\label{f1_real}
F_{\text{\bf r}} \sim \Sigma_\text{\bf r} = 
\left(
\begin{array}{c|c} A &  
\left( B+\Pi_k\sigma^{-1}H_k^{*}\right)L_0
\\
\hline \rule{0cm}{.42cm}
   C & D L_0
\end{array}
\right)
\end{equation}
where $\sigma$ is the positive definite solution of the
Riccati-equation
(\ref{ric}).
\item[(ii)]
denoting by $\cV^{*}\left(\Sigma_{\text{\bf r}}\right)$,  
$\cC^{*}\left(\Sigma_{\text{\bf r}}\right)$ 
the maximal output-nulling controlled
invariant subspace and the minimal input-containing
subspace, respectively, of the {\bf realization} of
$F_{\text{\bf r}}$ provided in (\ref{f1_real}) we get that
\begin{eqnarray}
\cV^{*}\left(\Sigma_{\text{\bf r}}\right) &=& 
\cV^{*}\left(\Sigma\right)\label{v_v1}\\
\cV^{*}\left(\Sigma_{\text{\bf r}}\right) \cap
\cC^{*}\left(\Sigma_{\text{\bf r}}\right) &=& \left\{ 0 \right\}
\label{v1_r1}\\
\left(\cV^{*}\left(\Sigma\right) 
\cap \cC^{*}\left(\Sigma\right)\right) 
\vee
\cC^{*}\left(\Sigma_{\text{\bf r}}\right) &=&
\cC^{*}\left(\Sigma\right)
\label{vr_v1r1}\\
\noalign{and}
\cW(\ker F_{\text{\bf r}} ) &=& \left\{ 0 \right\}\;,
\end{eqnarray}
\item[(iii)] 
the reachability subspace of the given realization of $F$
contains that of (\ref{f1_real}), i.e.
\[
 \left< A \mid\left(B+\Pi_k\sigma^{-1}H_k^{*}\right)L_0 
\right>
\subset
\left< A \mid B \right>
\]
and if \begin{itemize}
\item[a)] 
all the eigenvalues of the matrix $A$ have
non-positive real part, or
\item[b)]
the matrices
\begin{eqnarray}\label{eqn:disjoint_sp}
A \quad \text{and}\quad -A^{*}&&
\text{have no common eigenvalues, and}\\
\label{eqn:reach_AB_ACbar}
\text{the pair}\quad 
( A, \overline C^{*} )
 &&
\text{is stabilizable (in continuous time sense)}
\end{eqnarray}
where $\overline C = C P + D B^{*}$ and
$P$ is the solution of the Lyapunov-equation 
\begin{equation}\label{eqn:ljap_F_P}
AP + PA^{*} + BB^{*} = 0\;,
\end{equation}
\end{itemize}
then the 
reachability subspaces of the given
realizations of $F$ and $F_{\text{\bf r}}$ coincide, i.e.
\begin{equation}\label{eqn:reach_F_Fr}
\left< A \mid B \right> = \left< A \mid
\left(B+\Pi_k\sigma^{-1}H_k^{*}\right)L_0 
\right>
\end{equation}
\item[(iv)]
if the reachability subspaces above coincide then 
the finite zero matrix of $F_{\text{\bf r}}$
(will be denoted by $\Lambda_f(F_{\text{\bf r}})$) is  
given by 
\begin{equation}\label{eqn:finite_zero_Fr}
\Lambda_f (F_{\text{\bf r}}) = 
\left[
\begin{array}{cc}
-\sigma^{-1}\left(\Lambda_k+\alpha_k\beta_k\right)^{*}
\sigma & 
\Lambda_{kf} + \sigma^{-1}
\left(\beta_k^{*}R_0^{*}+H_k^{*}\right)H_f \\
0 & \Lambda_f
\end{array}
\right]\;.
\end{equation}
\end{itemize}
\end{theorem}

Let us observe that the theorem shows that
as a result of factoring out
$\cW (\ker F)$ the {\it virtual zeros}
of $F$ are materialized as {\it finite zeros}
of $F_{\text{\bf r}}$ appearing on the right half plane of
$\CC$ together with preserving the original finite zeros
of $F$.

\begin{proof}
(i)
Let us first compute a realization of $F_{\text{\bf r}}$.
\begin{eqnarray}
F_{\text{\bf r}}(z)& = & F(z) L_\beta (z) 
\nonumber \\
&=& \left(D+C\left(zI-A\right)^{-1}B\right)
\left(
L_0
- \left(H_k+R_0\beta_k\right)
\left(zI- \left(\Lambda_k + 
\alpha_k\beta_k\right)\right)^{-1}
\sigma^{-1}H_k^{*}L_0
\right) 
\nonumber \\
&=& 
DL_0 + C\left (zI-A\right)^{-1} 
\left(BL_0 + \Pi_k \sigma^{-1} H_k^{*}L_0\right)\;,
\label{eq:real_f1}
\end{eqnarray}
using the identities
\[
B\left(H_k+R_0\beta_k\right) = 
\left(zI-A\right)\Pi_k -
\Pi_k \left(zI-\Lambda_k - \alpha_k\beta_k\right)
\]
and 
\[
C \Pi_k + D\left(H_k+R_0\beta_k\right) = 0\;,
\]
proving part (i)

(ii)
According to Lemma \ref{lem:max_pi} 
to characterize the space 
$\cV^{*}\left(\Sigma_{\text{\bf r}}\right)$
a maximal solution of equation (\ref{matrix_zero_eq_max})
should be considered.
To this aim compute the following product:
\[
\left[
\begin{array}{cc}
A & \left(B+\Pi_k\sigma^{-1}H_k^{*}\right)L_0 \\
C & DL_0
\end{array}
\right]
\left[
\begin{array}{c}
\Pi_\text{max} \\ L_0^{*} H_\text{max}
\end{array}
\right]\;.
\]
Let us take the first element:
\begin{multline}
A\Pi_\text{max} + \left(B+\Pi_k\sigma^{-1}H_k^{*}\right) L_0
L_0^{*} H_\text{max} \\
= \Pi_\text{max}\Lambda_\text{max} - BH_\text{max} 
+ B L_0L_0^{*} H_\text{max} + 
\Pi_k\sigma^{-1}H_k^{*} L_0L_0^{*}H_\text{max} \\
\rule{0in}{3ex}
= \Pi_\text{max} \Lambda_\text{max} - BR_0R_0^{*}H_\text{max} + 
\Pi_k\sigma^{-1}H_k^{*} L_0L_0^{*}H_\text{max} \\
= \Pi_\text{max} \Lambda_\text{max} -
\Pi_k \alpha_k R_0^{*} H_\text{max} +
\Pi_k \sigma^{-1} H_k^{*} L_0 L_0^{*} H_\text{max}\;. 
\end{multline}

On the other hand
\[
C\Pi_\text{max} + DL_0 L_0^{*}H_\text{max} = 
-DH_\text{max}
+ D\left(I-R_0R_0^{*}\right) H_\text{max} = 0\;,
\]
proving that equation
\begin{equation}\label{pre_pi_in_pibar}
\left[
\begin{array}{cc}
A & \left(B+\Pi_k\sigma^{-1}H_k^{*}\right)L_0 \\
C & DL_0
\end{array}
\right]
\left[
\begin{array}{c}
\Pi_\text{max} \\ L_0^{*} H_\text{max}
\end{array}
\right]
 = 
\left[
\begin{array}{c}
\Pi_\text{max} \Lambda_e\\
0
\end{array}
\right]
\end{equation}
holds, for some matrix $\Lambda_e$ using that
$\Im \Pi_k \subset \Im \Pi_\text{max}$ and thus
proving that
$\Im \Pi_\text{max} \subset 
\cV^{*}\left(\Sigma_{\text{\bf r}}\right)$.

\bigskip
To prove the converse inclusion 
let us assume that the matrices $\bar\Pi$,
$\bar H$, $\bar\Lambda$ provide a maximal
solution of the equation
\begin{equation}\label{max_pi_f1}
\left[
\begin{array}{cc}
A & \left(B+\Pi_k\sigma^{-1}H_k^{*} \right)L_0 \\
C & DL_0
\end{array}
\right]
\left[
\begin{array}{c}
\bar\Pi \\ \bar H
\end{array}
\right] = 
\left[
\begin{array}{c}
\bar \Pi \bar\Lambda \\ 0
\end{array}
\right]\;.
\end{equation}
Due to the maximality we already have that 
$\Im \Pi_k \subset \Im \Pi_{fzk} \subset \Im \Pi$.
Rearranging the terms in
(\ref{max_pi_f1}) we get that
\[
\left[
\begin{array}{cc}
A & B \\
C & D
\end{array}
\right]
\left[
\begin{array}{c}
\bar\Pi \\ L_0 \bar H
\end{array}
\right] = 
\left[
\begin{array}{c}
\bar \Pi \bar\Lambda -
\Pi_k \sigma^{-1}H_k^{*}L_0^{*}\bar H\\ 0
\end{array}
\right] = 
\left[
\begin{array}{c}
\bar \Pi \bar\Lambda' \\ 0
\end{array}
\right]
\]
(where $\bar\Lambda'$ defined in an obvious way), 
giving that
\[
\Im \left(\bar \Pi\right) \subset 
\Im \left(\Pi_\text{max}\right)\;,
\]
thus
\[
\cV^{*}\left(\Sigma_{\text{\bf r}}\right) = 
\Im \left(\bar\Pi\right) =
\Im \left(\Pi_\text{max}\right) = 
\cV^{*}\left(\Sigma\right)\;,
\]
proving (\ref{v_v1}).

\bigskip
To prove that $\cV^{*}\left(\Sigma_{\text{\bf r}}\right) 
\cap \cC^{*}\left(\Sigma_{\text{\bf r}}\right) =
\left\{ 0 \right \}$ Corollary \ref{cor:ker_triv2}
can be applied giving that solutions of
\[
\left[
\begin{array}{c}
\left(B + \Pi_k\sigma^{-1} H_k^{*}\right)
L_0 \\ D L_0
\end{array}
\right] \xi =
\left[
\begin{array}{c}
\bar \Pi \eta \\ 0
\end{array}
\right]
\]
should be considered. Rearranging 
the first equation
we obtain that
\[
BL_0\xi = \bar \Pi \eta 
- \Pi_k \sigma^{-1} H_k^{*}L_0\xi \in 
\Im \bar \Pi = \Im \Pi_\text{max}\;.
\]
Using the second equation: 
$D \left(L_0\xi\right) = 0$, 
the maximality of 
$R_0$ in (\ref{matrix_kernel_eq}) 
gives that $L_0\xi \in \Im \left(R_0\right)$.
This implies that $\xi = 0$, 
consequently $\eta = 0$.
Thus
\[
\cV^{*}\left(\Sigma_{\text{\bf r}}\right) 
\cap
\cC^{*}\left(\Sigma_{\text{\bf r}}\right) 
= \left\{ 0 \right\}\;,
\]
or in other words the module
$\cW \left(\ker F_{\text{\bf r}}\right)$ is trivial.

\bigskip
To prove (\ref{vr_v1r1}) we first 
verify the inclusion
$\cC^{*}\left(\Sigma_{\text{\bf r}}\right) 
\subset \cC^{*}\left(\Sigma\right)$.
Since the elements of 
$\cC^{*}\left(\Sigma_{\text{\bf r}}\right)$ are those
vectors in the state-space 
which are reachable from
the origin via a trajectory 
producing no output 
we might apply
an induction argument. Obviously,
$0 \in \cC^{*}\left(\Sigma_{\text{\bf r}}\right)
\cap \cC^{*}\left(\Sigma\right)$.  

Now, if  
$\xi \in \cC^{*}\left(\Sigma_{\text{\bf r}}\right) 
\cap \cC^{*}\left(\Sigma\right)$ 
and equations
\begin{eqnarray*}
\eta &=& A \xi 
+ \left(BL_0 + \Pi_k\sigma^{-1}H_k^{*}
L_0\right) u \\
0 &= & C\xi + DL_0 u
\end{eqnarray*}
hold, then 
$\eta \in \cC^{*}\left(\Sigma_{\text{\bf r}}\right)$
and $A\xi + B L_0u \in \cC^{*}\left(\Sigma\right)$, while
$\Pi_k \sigma^{-1}H_k^{*} L_0 u \in \Im \Pi_k =
\cV^{*}\left(\Sigma\right)\cap 
\cC^{*}\left(\Sigma\right)$. 
Thus $\eta \in \cC^{*}\left(\Sigma\right)$, as well.

By induction this proves that
\[
\cC^{*}\left(\Sigma_{\text{\bf r}}\right) 
\subset \cC^{*}\left(\Sigma\right)\;.
\]

\medskip
Conversely, if 
$\xi \in \cC^{*}\left(\Sigma\right) \cap
\left(\cC^{*}\left(\Sigma_{\text{\bf r}}\right) \vee
\Im \Pi_k\right)$
and
\begin{eqnarray*}
\eta &=& A \xi + B u \\
0 &= & C\xi + D u
\end{eqnarray*}
then $\eta \in \cC^{*}\left(\Sigma\right)$. Introducing
the notation $u_1 = R_0^{*} u$,
$u_2 = L_0^{*}u$ we get that
$u = R_0 u_1 + L_0 u_2$.

The assumption 
$\xi \in \left(\cC^{*}\left(\Sigma_{\text{\bf r}}\right) \vee
\Im \Pi_k\right)$ implies that
\[
\xi = \xi_1 + \Pi_k v\;,
\]
for some vectors $\xi_1, v$, where
$\xi_1 \in \cC^{*}\left(\Sigma_{\text{\bf r}}\right)$.
Now
\begin{eqnarray*}
C\xi_1 + D\left(L_0 u_2 - H_k v\right) &=& 
C\xi_1 + D L_0 u_2 - D H_k v + D R_0R_0^{*} H_k v \\
&=& 
C\xi_1 + DL_0 u_2 + C\Pi_k v \\
&=& C\xi + Du = 0
\end{eqnarray*}
implying that
\[
A \xi_1 + \left(B + \Pi_k\sigma^{-1}H_k^{*}\right)
L_0
\left(u_2 - L_0^{*} H_k v \right) \in
\cC^{*}\left(\Sigma_{\text{\bf r}}\right)\;.
\]
On the other hand 
\begin{eqnarray*}
\eta &=& A\xi + B u \\
&=& A\xi_1 + A\Pi_k v + B R_0 u_1 + B L_0 u_2 \\
&=& A\xi_1 + \left(B+\Pi_k\sigma^{-1}H_k^{*}\right)
L_0 \left( u_2 - L_0^{*}H_k v\right) 
- \Pi_k \sigma^{-1} H_k^{*}
L_0\left( u_2 - L_0^{*}H_1 v\right) \\
&& \qquad + B L_0L_0^{*} H_k v 
+ A\Pi_k v + B R_0 u_1 \\
&=&
A\xi_1 + \left(B+\Pi_k\sigma^{-1}H_k^{*}\right)
L_0 \left( u_2 - L_0^{*}H_k v\right)
\\
&& \qquad - \Pi_k \sigma^{-1} H_k^{*}
L_0 \left(u_2 - L_0^{*}H_k v\right) 
+\Pi_k\Lambda_k v 
+ \Pi_k \alpha_k \left(u_2 - R_0^{*}H_k v\right)\;.
\end{eqnarray*}
thus
\[
\eta \in \left(\cC^{*}\left(\Sigma_{\text{\bf r}}\right)
\vee \Im \left(\Pi_k\right)\right)\;.
\]

Induction argument gives that
\[
\cC^{*}\left(\Sigma\right) \subset 
\left(\cC^{*}\left(\Sigma_{\text{\bf r}}\right) \vee 
\Im \Pi_k\right)
\]
Consequently
\[
\cC^{*}\left(\Sigma\right)
=
\Im \Pi_k \vee \cC^{*}\left(\Sigma_{\text{\bf r}}\right)  =
\left(\cC^{*}(\Sigma)\cap \cV^{*}(\Sigma)\right) \vee \cC^{*}\left(\Sigma_{\bf r}\right)
\;.
\]
Finally the identity (\ref{v_v1}) and 
Corollary \ref{cor:ker_triv2}
imply that 
\[
\cW\left(\ker F_{\text{\bf r}}\right) = \left\{ 0 \right\}\;,
\]
concluding the proof of part (ii).

(iii)
To prove the first part 
let us recall that $\Im \left(\Pi_k\right) \subset 
\Im \left(\Pi_{fzk}\right) = \cV^{*}\left(\sigma\right) \cap < A \mid B >$.
Thus if the column-vectors of the matrix $\xi$ form a basis
in the orthogonal complement of the reachability
subspace of $\left< A\mid B\right>$ i.e. 
$\xi^{*} B = 0$ and $\xi^{*} A = \kappa \xi^{*}$ holds for
some matrix $\kappa$ then $\xi^{*} \Pi_k = 0$, as well.
Consequently, $\xi^{*} \left(B+\Pi_k \sigma^{-1}H_k^{*}\right)L_0 =0$
giving that the columns of $\xi$ are orthogonal to the
elements of $\left< A\mid \left(B+\Pi_k\sigma^{-1}H_k^{*}\right)L_0
\right>$ proving the first inclusion.

To prove the second part of (iii) let us first consider 
two identities.
\begin{eqnarray}
B &=& \left( B + \Pi_k \sigma^{-1}H_k^{*}\right)
\left(L_0L_0^{*} + R_0 R_0^{*}\right) -
\Pi_k \sigma^{-1} H_k^{*} \nonumber \\
&=& \left(B+\Pi_k\sigma^{-1}H_k^{*}\right) L_0L_0^{*} 
+
\Pi_k \left(\sigma^{-1} H_k^{*} R_0R_0^{*} +\alpha_k R_0^{*}
-\sigma^{-1} H_k^{*}\right) \nonumber \\
\label{eqn:B_form_Pik}
&=& \left(B+\Pi_k\sigma^{-1}H_k^{*}\right) L_0L_0^{*}
-\Pi_k \sigma^{-1}\left(\beta_k^{*} R_0^{*} + H_k^{*}\right)
\end{eqnarray}
using that $BR_0 = \Pi_k \alpha_k$, and (\ref{beta_def})
and
\begin{eqnarray}
A\Pi_k + \left(B+\Pi_k\sigma^{-1}H_k^{*}\right)
L_0L_0^{*} H_k &=&
A\Pi_k + B H_k + \Pi_k\sigma^{-1}
\left(\beta_k^{*} R_0^{*} + H_k^{*}\right) H_k 
\nonumber \\
&=&
\Pi_k\Lambda_k + \Pi_k \sigma^{-1} 
\left(\beta_k^{*}R_0^{*} + H_k^{*}\right) H_k 
\nonumber \\
&=&
\Pi_k \left(\Lambda_k + \sigma^{-1}
\left(\beta_k^{*} R_0^{*} + H_k^{*}\right) H_k\right) 
\nonumber \\
&=& - \Pi_k \sigma^{-1} 
\left(\Lambda_k + \alpha_k \beta_k\right)^{*} \sigma\;,
\label{eqn:AB_form_Pik}
\end{eqnarray}
using that $A\Pi_k + B H_k = \Pi_k\Lambda_k$ and equations
(\ref{inner_eq1}), (\ref{inner_eq2}). 

Consider first the assumption formulated in a),
i.e. if all the eigenvalues of the matrix $A$
have non-positive real part then the reachability
subspaces of the given realizations of $F$ and $F_{\text{\bf r}}$
coincide. 

If the column-vectors of the matrix 
$\xi$ form a basis in the
orthogonal complement of the 
reachability subspace of
the realization above of $F_{\text{\bf r}}$, then
$\xi^{*} A = \kappa \xi^{*}$ for 
some matrix $\kappa$
and $\xi^{*} \left(B+\Pi_k\sigma^{-1}
H_k^{*}\right)L_0 = 0$. The eigenvalues
of the matrix $\kappa$ should form a
subset of those of $A$. 

Equations (\ref{eqn:AB_form_Pik}) and
(\ref{eqn:B_form_Pik}) imply that
$\kappa \xi^{*} \Pi_k = 
- \xi^{*} \Pi_k \sigma^{-1}
\left(\Lambda_k + \alpha_k \beta_k\right)^{*} \sigma$
showing in particular that 
$\Im \left( \xi^{*} \Pi_k\right)$
is $\kappa$-invariant.
Since according to the proof of 
Theorem \ref{inner_fnc_kernel}
the matrix $\Lambda_k+\alpha_k\beta_k$ is
asymptotically stable we get
that on the subspace
$\Im \left( \xi^{*} \Pi_k\right)$
the eigenvalues of the matrix $\kappa$
have positive real part.

But according to the assumption
the spectrum of the matrix $A$ 
is in the closed left half plane, 
consequently the eigenvalues of $\kappa$
should have non-positive real part.
Thus
the equation $\xi^{*}\Pi_k = 0$ holds true
implying that $\xi^{*} B =0$, as well. 
So the columns of $\xi$
are orthogonal to the reachability subspace of the
realization of $F$. I.e.
\[
\left( A\mid B\right> \subset
\left< A\mid \left(B+\Pi_k\sigma^{-1}H_k^{*}
\right)L_0\right>\;,
\]
proving the second part of (iii) using the
assumption formulated in a).

To prove the converse inclusion based on the
assumption b) assume again that
the column-vectors of the matrix 
$\xi$ form a basis in the
orthogonal complement of the 
reachability subspace of
the realization above of $F_{\text{\bf r}}$, then
$\xi^{*} A = \kappa \xi^{*}$ for 
some matrix $\kappa$
and $\xi^{*} \left(B+\Pi_k\sigma^{-1}
H_k^{*}\right)L_0 = 0$.
Now from (\ref{eqn:B_form_Pik}) we get
that
\[
\xi^{*} B = - \xi^{*} \Pi_k\sigma^{-1} 
\left(H_k + R_0 \beta_k\right)^{*}\;.
\]
Multiplying from the right by $D^{*}$ and using 
equation (\ref{alpha1}) and (\ref{pi_1})
we obtain that
\[
\xi ^{*} B D^{*} = \xi^{*} \Pi_k \sigma^{-1} 
\Pi_k^{*} C^{*}\;.
\]
On the other hand multiplying the Lyapunov-equation (\ref{eqn:ljap_F_P})
above from the left by $\xi^{*}$ and using
again equations (\ref{alpha1}),  
(\ref{pi_1}) and (\ref{eqn:AB_form_Pik})
we arrive at the following 
equation
\begin{eqnarray*}
- \xi^{*} P A^{*} &=& \xi^{*} AP  + 
\xi^{*} BB^{*} \\
&=& \kappa \xi^{*} P 
-\xi^{*} \Pi_k\sigma^{-1} 
\left(H_k + R_0 \beta_k\right)^{*} B^{*} \\
&=& 
\kappa \xi^{*} P 
-\xi^{*} \Pi_k\sigma^{-1} 
\left( \left(\Lambda_k+\alpha_k\beta_k\right)^{*} \Pi_k^{*}
- \Pi_k^{*} A^{*}\right) \\
&=& 
\kappa \xi^{*} P  
+ \kappa \xi^{*} \Pi_k \sigma^{-1}\Pi_k^{*} 
+ \xi^{*} \Pi_k \sigma^{-1}\Pi_k^{*}  A^{*}\;.
\end{eqnarray*}
Rearranging it
\[
\kappa \xi^{*}\left( P + \Pi_k\sigma^{-1}\Pi_k^{*}\right) = 
-\xi^{*}\left( P + \Pi_k\sigma^{-1}\Pi_k^{*}\right) A^{*}\;.
\]
Since according to our assumption the spectra of $A$ and
$-A^{*}$ are disjoint but the spectrum of $\kappa$ should be
a subset of that of $A$ we find that
\[
\xi^{*}\left(P + \Pi_k \sigma^{-1}\Pi_k^{*}\right) = 0\;.
\]
Consequently,
\[
\xi^{*} \overline C^{*} = 
\xi^{*} \left( P C^{*} + B D^{*}\right)
= \xi^{*} \left(-\Pi_k \sigma^{-1}\Pi_k^{*} C^{*}
+ 
\Pi_k \sigma^{-1}\Pi_k^{*} C^{*}\right) = 0\;.
\]
Thus the eigenvalues of $\kappa$ belong to
the uncontrollable (with respect to the pair
$\left(A, \overline C^{*}\right)$) eigenvalues
of $A$. According to the assumption these
eigenvalues have non-positive real part, but
equation (\ref{eqn:AB_form_Pik}) 
implies that on the subspace
$\Im \left(\xi^{*} \Pi_k\right)$ the
matrix $(-\kappa)$ should be asymptotically stable.
Thus $\xi^{*} \Pi_k = 0$. Consequently,
\[
\xi^{*} B = - \xi^{*} \Pi_k\sigma^{-1} 
\left(H_k + R_0 \beta_k\right)^{*} = 0\;.
\]
Thus
\[
\left< A\mid B\right> \subset
\left< A\mid 
\left(B+\Pi_k\sigma^{-1}H_k^{*}\right)L_0\right>\;,
\]
proving in this case, as well, 
that these two reachability subspaces coincide.

\medskip
(iv)
To conclude the proof of the theorem the finite zero matrix of
$F_{\bf r}$ should be computed.
According to Theorem  \ref{thm:char_ker}
equation (\ref{v1_r1}) gives that
$\cW\left(\ker F_{\text{\bf r}}\right) = \left\{ 0 \right\}$,
consequently from Corollary \ref{cor:Im_pi} it
follows that to identify the finite zero matrix
of $F_{\text{\bf r}}$ a basis in 
$\cV^{*}\left(\Sigma_{\text{\bf r}}\right) \cap 
\left< A \mid \left(B+\Pi_k\sigma^{-1}
H_k^{*}\right)L_0\right>$
should be considered and taken as 
the matrix ``$\Pi$''
in the corresponding form of the equation
(\ref{matrix_zero_eq_fzk}). Now equation 
(\ref{v_v1}) in part (ii) and the assumption concerning the
reachability subspaces imply that
\[
\cV^{*}\left(\Sigma_{\text{\bf r}}\right) \cap 
\left< A \mid \left(B+\Pi_k\sigma^{-1}H_k^{*}
\right)L_0\right> =
\cV^{*}\left(\Sigma\right) \cap \left< A \mid B\right>
\]  
so the columns of $\Pi_{fzk}$ form a basis 
in this subspace.
Thus it is reasonable to compute the product
\begin{equation}\label{guess_pi}
\left[\begin{array}{cc}
A & \left(B+\Pi_k\sigma^{-1}H_k^{*}\right)L_0 \\
C & DL_0
\end{array}\right]
\left[\begin{array}{c}
\Pi_{fzk} \\
L_0^{*} H_{fzk}
\end{array}\right]\;.
\end{equation}
Let us take the first element:
\begin{multline}
A \Pi_{fzk} +\left(B+\Pi_k\sigma^{-1}H_k^{*}\right)
L_0 L_o^{*} H_{fzk} \\
=
\Pi_{fzk}\Lambda_{fzk} - B H_{fzk} + BL_0L_0^{*} H_{fzk}
+ \Pi_k\sigma^{-1}H_k^{*}L_0L_0^{*}H_{fzk} \\
= \Pi_{fzk} \Lambda_{fzk} - BR_0R_0^{*} H_{fzk} 
+ \Pi_k \sigma^{-1}H_k^{*} L_0L_0^{*} H_{fzk} \\
= 
\Pi_{fzk} \Lambda_{fzk} - \Pi_k \alpha_k R_0^{*} H_{fzk}
+\Pi_k \sigma^{-1} H_k^{*} L_0L_0^{*} H_{fzk}\;.
\end{multline}
Taking the partitioned form of these matrices 
we obtain that the first
block is
\begin{eqnarray*}
\Pi_k \Lambda_k &-& \Pi_k\alpha_kR_0^{*}H_k 
+ \Pi_k \sigma^{-1} H_k^{*} L_0L_0^{*}
H_k \\
&=& \Pi_k \sigma^{-1} \left( \sigma
\left(\Lambda_k - \alpha_k
      R_0^{*}H_k\right) +
H_k^{*}\left(I-R_0R_0^{*}\right)H_k\right) \\
&=&
-\Pi_k\sigma^{-1}\left(\left(\Lambda_k 
-\alpha_k R_0^{*}H_k\right)^{*}
\sigma - \sigma \alpha_k 
\alpha_k^{*} \sigma \right) \\
&=& -\Pi_k \sigma^{-1} 
\left(\Lambda_k + \alpha_k\beta_k
\right)^{*}\sigma\;,
\end{eqnarray*}
and also
\[
\Pi_k \Lambda_k - \Pi_k \alpha_k R_0^{*} H_k 
+\Pi_k \sigma^{-1} H_k^{*} L_0L_0^{*} H_k =
\Pi_k \left(\Lambda_k +
\sigma^{-1}\left(\beta_k^{*} R_0^{*} 
+ H_k^{*}\right) H_k\right)
\]
using the Riccati-equation (\ref{ric}) and the identity 
(\ref{beta_def}).

Let us compute the second block:
\begin{multline}
\Pi_k\Lambda_{kf} + \Pi_f \Lambda_f - \Pi_k \alpha_k R_0^{*} H_f
+ \Pi_k \sigma^{-1} H_k^{*} L_0L_0^{*} H_f  \\
=
\Pi_k \left( \Lambda_{kf} + \sigma^{-1}
\left(\beta_k^{*} R_0^{*} + H_k^{*}\right) H_f \right) + 
\Pi_f \Lambda_f\;.
\end{multline}
On the other hand the second element in (\ref{guess_pi}):
\[
C\Pi_{fzk} + D L_0 L_0^{*} H_{fzk} =
- D H_{fzk} + D \left(I-R_0R_0^{*}\right)H_{fzk} = 0\;,
\]
proving that equation
\begin{equation}\label{pi_in_pibar}
\left[\begin{array}{cc}
A & \left(B+\Pi_k\sigma^{-1}H_k^{*}\right)L_0 \\
C & DL_0
\end{array}\right]
\left[\begin{array}{c}
\Pi_{fzk} \\
L_0^{*} H_{fzk}
\end{array}\right]
=
\left[\begin{array}{c}
\Pi_{fzk}\left(\Lambda_{fzk} 
+ \Gamma H_{fzk}\right) \\
0
\end{array}\right]
\end{equation}
holds, where
\begin{equation}\label{eq:def_gamma}
\Gamma = 
\left[\begin{array}{c}
\sigma^{-1} \left(\beta_k^{*} R_0^{*} 
+ H_k^{*}\right) \\
0
\end{array}\right]\;.
\end{equation}
Thus -- using Corollary \ref{cor:Im_pi} and 
$\cW(\ker F_{\text{\bf r}} ) = \left\{ 0 \right\}$ from part (ii) -- the matrix
\[
\Lambda_{fzk} + \Gamma H_{fzk} =
\left[\begin{array}{cc}
-\sigma^{-1}\left(\Lambda_k 
+\alpha_k\beta_k\right)^{*} \sigma &
\Lambda_{kf} + \sigma^{-1} 
\left(\beta_k^{*} R_0^{*} + H_k^{*}\right)
H_f \\
0 & \Lambda_f
\end{array}\right]
\]
is the finite zero matrix of $F_{\text{\bf r}}$, 
concluding the proof of (iv) and that of the theorem.
\qed\end{proof}

\begin{remark}\label{rem:F_r_left_inv}
Let us note that the function $F_{\text{\bf r}}$ is 
{\it left-invertible}. 
In fact, according to the Remark \ref{rem:left_inv}
and equation (\ref{v1_r1}) it 
remains only to check the
kernel of 
$\left[\begin{array}{c} 
\left(B+\Pi_k\sigma^{-1}H_k^{*}\right)L_0 \\
D L_0
\end{array}\right]$. Now if for some vector $\xi$
the identity 
$\left(B+\Pi_k\sigma^{-1}H_k^{*}\right)L_0 \xi = 0$
holds, then obviously 
 $BL_0 \xi \in \Im \Pi_k \subset \Im \Pi$. 
If moreover
$DL_0 \xi =0$, as well, then -- using 
the maximality of $R_0$
-- $L_0 \xi \in \Im R_0$ should hold. 
But this implies that
$L_0\xi = 0$, so $\xi =0$. I.e. both 
conditions for the
left-invertibility hold.
\end{remark}
\bigskip

\begin{remark}\label{rem:red_pik}
Let us point out that even in the case when there is a reduction in the reachability
subspace the finite zero matrix $\Lambda_f$ of $F$ appears in the finite zero matrix 
of $F_{\bf r}$.

In fact, we are going to show that
\begin{multline}
\dim \left[ \cV^{*}(\Sigma) \cap < A \mid B>\right] -
\dim \left[\cV^{*}(\Sigma _{\bf r}) \cap 
\left< A \mid \left(B+\Pi_k \sigma^{-1}H_k^{*}\right)L_0\right>
\right] \\
=
\dim \left(\Im \Pi_k\right) -
 \dim \left[\Im (\Pi_k) \cap 
\left< A \mid \left(B+\Pi_k \sigma^{-1}H_k^{*}\right)L_0\right> 
\right]
\end{multline}
i.e. the "reduction" affects only the subspace $\cC^{*}(\sigma) \cap \cV^{*}(\sigma) = 
\cR^{*}(\sigma)$

Let us observe that the inclusion $\Im (\Pi_k) \subset \cV^{*}(\Sigma ) =
\cV^{*} (\Sigma_{\bf r})$ implies that the inequality $\geq$ holds trivially.

To prove the converse inequality let us consider a matrix $\xi$ with columns forming a basis in 
the orthogonal complement of the reachability subspace
$\left< A \mid (B+\Pi_k\sigma^{-1}H_k^{*})L_0\right>$. Then 
\[
\text{rank } \xi^{*} \Pi_k = \dim \left(\Im \Pi_k\right) -
 \dim \left[\Im (\Pi_k) \cap 
\left< A \mid \left(B+\Pi_k \sigma^{-1}H_k^{*}\right)L_0\right>\right] \;.
\]
We are going to show that the inclusions
$\Im\, \left[\xi^{*} A^jB\right] \subset\Im\,  \xi^{*} \Pi_k$ hold, for all $j \geq 0$ proving that
$\text{rank } \xi^{*} [ B, AB, A^2B, \dots ] \leq \text{rank} \xi^{*} \Pi_k$.
In fact, equation (\ref{eqn:B_form_Pik}) gives that 
$\xi^{*} B =  \xi^{*} \Pi_k \sigma^{-1} \left(\beta_k^{*}R_0^{*} + H_k^{*}\right)$,
thus $\Im \,\xi^{*} B \subset \Im\,\xi^{*} \Pi_k$. On the other hand
(\ref{eqn:AB_form_Pik}) gives immediately that
$\Im\, \xi^{*}A\Pi_k \subset \Im\, \xi^{*}\Pi_k$. Starting form these observation
we shall prove by induction that 
$\Im\,\xi^{*} A^j\Pi_k \subset \Im\, \xi^{*} \Pi_k$ and
$\Im\,\xi^{*} A^jB \subset \Im\,\xi^{*}\Pi_k$ for all $j\geq 0$.

Equations (\ref{eqn:AB_form_Pik}) and (\ref{eqn:B_form_Pik}) imply that
\begin{eqnarray*}
\xi^{*}AB &= & -\xi^{*} A 
\Pi_k \sigma^{-1} \left(\beta_k^{*}R_0^{*} + H_k^{*}\right) \\
&=& 
\xi^{*}\Pi_k \sigma^{-1} \left(\Lambda_k + \alpha_k\beta_k\right)^{*} 
  \left(\beta_k^{*}R_0^{*} + H_k^{*}\right)
\end{eqnarray*}
Using the equations (\ref{eqn:B_form_Pik}) and $A\Pi_k = -BH_k + \Pi_k \Lambda_k$ we can write
\begin{eqnarray*}
\xi^{*} A^j B &=& -\xi^{*} A^j \Pi_k  \sigma^{-1} \left(\beta_k^{*}R_0^{*} + H_k^{*}\right) \\
&=&
-\xi^{*} A^{j-1}(-B H_k  + \Pi_k \Lambda_k) 
\sigma^{-1} \left(\beta_k^{*}R_0^{*} + H_k^{*}\right)\;.
\end{eqnarray*}
Thus 
$\Im\,\xi^{*} A^jB \subset \Im\,\xi^{*} A^j\Pi_k \subset
\Im\,\xi^{*} A^{j-1}B \vee \Im\,\xi^{*} A^{j-1}\Pi_k
\subset \Im\,\xi^{*}\Pi_k$
by induction.

Consequently,
\begin{multline*}
\dim \left[ \cV^{*}(\Sigma) \cap < A \mid B>\right] -
\dim \left[\cV^{*}(\Sigma _{\bf r}) \cap 
\left< A \mid \left(B+\Pi_k \sigma^{-1}H_k^{*}\right)L_0\right>
\right] \\
=
\dim \left[ \cV^{*}(\Sigma) \cap < A \mid B>\right] -
\dim \left[\cV^{*}(\Sigma ) \cap 
\left< A \mid \left(B+\Pi_k \sigma^{-1}H_k^{*}\right)L_0\right>
\right]  \\
\leq 
\dim <A \mid B > - \dim < A \mid (B+\Pi_k\sigma^{-1}H_k^{*})L_0 > 
\leq  \text{rank } \xi^{*}\Pi_k \\
= 
\dim (\Im \Pi_k ) -
 \dim \left[\Im \Pi_k \cap < A \mid (B+\Pi_k\sigma^{-1}H_k^{*})L_0 >\right]\;,
\end{multline*}
proving the converse inequality, as well.
\end{remark}

\bigskip
\subsection{The zero module $\cW\left(\Im \ F \right)$}

Now let us turn to the analysis of the space
\[
\cW (\Im F)
=\frac{\pi_- (\Im F)}{\Im F \cap z^{-1}\Omega_\infty Y}\,.
\]
A $p$-tuple $h$ is in $\pi_- (\Im F)$ if 
it is strictly proper and there
exists a polynomial $p$-tuple $\phi$ such that
$h+\phi \in \Im (F)$. Two such functions $h_1, h_2$
are considered to be equivalent 
if $h_1-h_2 \in \Im (F)$.

Based on these observations 
the following theorem gives a ``state-space''
characterization of the elements in
$\cW (\Im F)$.

\begin{theorem}\label{thm:char_im}
Assume that the pair $\left(C, A\right)$ is 
observable. Then
the equivalence classes of 
$\cW\left(\Im \ F\right)$
are determined by the functions
\[
C\left(zI-A\right)^{-1}\beta
\]
where $\beta \in \left< A \mid B \right>$
and two functions -- given 
by the vectors $\beta_1, \beta_2$ --
are considered to be equivalent if
\[
\beta_1-\beta_2 \in 
\cV^{*}\left(\Sigma\right) \vee 
\cC^{*}\left(\Sigma\right)\;.
\]
\end{theorem}

\begin{proof}
Consider a rational $q$-tuple $g(z)$. Assume that
\[
g(z) = H_1\left(zI-\Lambda_1\right)^{-1}G_1 +
g_0 + g_1z + \dots + g_k z^k\;.
\]
Then the observation that 
$z^{j} \left(zI-A\right)^{-1} - 
A^j \left(zI-A\right)^{-1}$ 
is a polynomial implies that
\[
\pi_- \left(F(z) g(z)\right) = 
F(z) H_1\left(zI-\Lambda_1\right)^{-1}G_1
+ C\left(zI-A\right)^{-1} 
\sum_{l=0}^k A^lB g_l\;.
\]
The first term is strictly proper and in the space
$\Im F$ thus the second term
determines the corresponding equivalence class.
By definition any function of the
form $C\left(zI-A\right)^{-1} \beta$ where 
$\beta\in \left< A \mid B\right>$ 
can be obtained this way.
But possibly different $\beta$ vectors might
generate the same equivalence class.

Thus we have to characterize those 
$\beta\in \left< A \mid B\right>$ vectors
for which $C\left(zI-A\right)^{-1}\beta \in
\Im F \cap z^{-1}\Omega_\infty Y$.
To this aim assume that $F g = C\left(zI-A\right)^{-1} \beta$
is strictly proper for some rational function $g$ with the form
given above.

Straightforward calculation gives the 
polynomial part of the product.
Namely it is
\[
D g_k z^k + \sum_{j=0}^{k-1} \left( D g_j + 
\sum_{l=0}^{k-1-j} C A^l B g_{l+j+1}
\right)z^j\;.
\]
This should be zero. Since the 
polynomial part of $g$
gives rise to $C\left(zI-A\right)^{-1} 
\sum_{l=0}^k A^l B g_l$
in the strictly proper part of $F g$, 
we get that
\[
\beta_1 = \sum_{l=0}^k A^l B g_l
\]
 should be an output-nulling reachable element,
or in other words 
$\beta_1 \in \cC^{*}\left(\Sigma\right)$. 
Introducing the 
notation $\beta_2 = \beta - \beta_1$, we obtain
that
$F(z) H_1\left(zI-\Lambda_1\right)^{-1}G_1 =
C\left(zI-A\right)^{-1}\beta_2$.
Using the observability of the
pair $(C, A)$ and Proposition \ref{product_equation}
we obtain that 
\[
\beta_2 = \Pi_1 G_1\;,
\]
where $\Pi_1$ is a solution of the equation
\[
\left[\begin{array}{cc}
A & B \\ C & D
\end{array}\right]
\left[\begin{array}{c} \Pi_1 \\ H_1
\end{array}
\right]
=
\left[
\begin{array}{c}
\Pi_1\Lambda_1 \\ 0
\end{array}
\right]\;.
\]
I.e. $\beta_2$ is in the maximal output-nulling
controlled invariant set, 
$\beta_2 \in \cV^{*}\left(\Sigma\right)$.

Conversely, if $\beta \in \left< A \mid B\right>$,
and $\beta=\beta_1+\beta_2$, 
$\beta_1 \in \cC^{*}\left(\Sigma\right) 
\subset \left<A\mid B\right>$,
$\beta_2 \in \cV^{*}\left(\Sigma\right)$, then $\beta_1$
can be written in the form
\[
\beta_1 = \sum_{l=0}^ k A^lB g_l
\]
in such a way that for the polynomial
$g(z) = \sum_{l=0}^k g_l z^l$ the identity
\[
F(z)g(z) = C\left(zI-A\right)^{-1} \beta_1
\]
holds true. On the other hand assume that the 
triplet $\left(\Pi_\text{max}, H_\text{max}, 
\Lambda_\text{max}\right)$
forms a maximal solution of (\ref{matrix_zero_eq_max}).
Then $\beta_2\in \cV^{*}\left(\Sigma\right)$ implies that
\[
\beta_2 = -\Pi_\text{max} G
\]
for some vector $G$. Now immediate calculation gives that
\[
F(z) H_\text{max}
\left(zI-\Lambda_\text{max}\right)^{-1}G = 
-C(zI-A)^{-1}\Pi_\text{max} G
= C(zI-A )^{-1} \beta_2\;.
\]
Consequently,
\[
F(z) \left(H_\text{max}\left(zI-\Lambda_\text{max}\right)^{-1}G
+ g(z)\right) = C\left(zI-A\right)^{-1} \beta\;,
\]
thus it is in the space $\Im F \cap z^{-1}\Omega_\infty Y$,
concluding the
proof of the theorem.
\qed
\end{proof}

\begin{remark}
The identification of the co-range of the function
$F$ to the factor-space $\left< A\mid B \right>\ / \ 
\left(\left<A \mid B\right> \cap 
\left(\cV^{*}\left(\Sigma\right) 
\vee \cC^{*}\left(\Sigma\right)\right)\right)$
can be found e.g. in \cite{AL-SC-84} (even without
the assumption of the observability of $(C, A)$) but 
without explicitly identifying the functions
in the equivalence classes of $\cW\left(\Im\ F\right)$. 
\end{remark}

\begin{remark}\label{rem:Im_fr_f}Assume that the pair
$(C, A)$ is observable, and the eigenvalues of 
$A$ are in the closed
left half plane or  -- more generally -- conditions (\ref{eqn:disjoint_sp} and
(\ref{eqn:reach_AB_ACbar}) hold.
Consider the function $F_{\text{\bf r}}$ defined in
(\ref{f1_def}). Due to the fact that
it has the same "$\left(C, A\right)$" pair
as the function $F$, the previous theorem 
together with part (ii) and (iii) of 
Theorem \ref{thm:f1} imply
that
\[
\cW\left(\Im \ F_{\text{\bf r}}\right) = \cW\left(\Im \ F \right)\;.
\]
\end{remark}

\subsection{Zeros at infinity}

Let us recall the definition of 
the zero module at infinity:

\[
Z_\infty (F) = \frac{F^{-1}(z^{-1}\Omega_\infty Y)
+ z^{-1}\Omega_\infty U}{\ker F + z^{-1} \Omega_\infty U}\;.
\]
I.e. the $q$-tuples of rational functions $g$ should be
considered for which there exist a strictly
proper rational $q$-tuple $h$ such that
\begin{equation}\label{zero_infty_1}
F (g+h) \quad \text{is strictly proper},
\end{equation}
and $g_1, g_2$ with this property are considered
to be equivalent if for some strictly proper
$q$-tuple $h$ the identity
\[
F (g_1-g_2+h) = 0\;.
\]

\begin{theorem}\label{zero_infty_thm}
Assume that the pair $(C, A)$ is observable.
Then the equivalence classes in
$Z_\infty (F)$ are
determined by the vectors in 
$\cC^{*}\left(\Sigma\right)$ in 
the sense that for any
$\beta \in \cC^{*}\left(\Sigma\right)$ there exists a finite
input sequence producing no output but 
giving $\beta$ as the next immediate state-vector.
The input sequence gives the coefficients of a
polynomial in $F^{-1}\left(z^{-1}\Omega_\infty Y\right)
+ z^{-1}\Omega_\infty U$.

Two polynomials are taken to be equivalent if 
the difference of the corresponding $\beta$
vectors are in $\cR^{*}\left(\Sigma\right) 
= \cV^{*}\left(\Sigma\right) \cap \cC^{*}\left(\Sigma\right) =
\Im \left(\Pi_k\right) $\;, see (\ref{eq_im_pi1}).
\end{theorem}

\begin{proof}
Since $F$ is assumed to be proper the function
$F h$ is strictly proper if $h$ is strictly proper. Thus the 
condition in (\ref{zero_infty_1})
states that $F g$ should be strictly proper.
Due to our assumption that the function $F$ is
proper this is equivalent to 
\[
\pi_+ \left(F \pi_+ (g) \right) = 0\;.
\]
Using the notation
\[
\pi_+ \left( g \right) = 
g_0 + g_1 z + \dots g_k z^k\;,
\]
we get that the sequence
$g_k, g_{k-1}, \dots , g_0$ gives an 
output-nulling input sequence, so it
takes the origin into some state-vector
$\beta \in \cC^{*}\left(\Sigma\right)$.

\medskip
Two such sequences are considered to be equivalent if
adding to their difference an appropriate 
strictly proper function a function in $\ker F$ is
obtained. So let us assume that
\[
g (z) = H_1\left(zI - \Lambda_1\right)^{-1}G_1 + 
g_0 + g_1z + \dots + g_kz^k \in \ker F\;.
\]
Under the assumption that the input sequence
$g_k, g_{k-1}, \dots , g_0$ produces no
output we get that the polynomial part of the
product $Fg$ is zero. 
Thus, computing the strictly proper part of 
$F g$ the equation
\[
\left(D+C\left(zI-A\right)^{-1}B\right)
H_1\left(zI-\Lambda_1\right)^{-1}G_1 +
C\left(zI-A\right)^{-1}
\sum_{j=0}^k A^j B g_j = 0
\]
is obtained. Proposition \ref{product_equation}
implies that there exists a matrix $\Pi_1$ such that
equation
\[
\left[\begin{array}{cc}
A & B \\ C & D \end{array}
\right]
\left[
\begin{array}{c}
\Pi_1 \\ H_1 \end{array}
\right] = 
\left[\begin{array}{c} 
\Pi_1 \Lambda_1 \\ 0 
\end{array}
\right]
\;, \quad \Pi_1 G_1 = \sum_{j=0}^k A^j B g_j
\]
hold. 

Conversely, if 
\[
\beta = \sum_{j=0}^k A^j B g_j \in 
\cV^{*}\left(\Sigma\right)\;,
\]
for an output-nulling input sequence
$g_k, g_{k-1}, \dots , g_0$
then there exists a vector $G$ such that
$\beta = \Pi_\text{max} G$. 
Using the identity
\[
F(z) H_\text{max} \left(zI - \Lambda_\text{max}\right)^{-1}
= - C \left(zI-A\right)^{-1} \Pi_\text{max}
\]
straightforward
computation gives that 
\[
H_\text{max} \left(zI - \Lambda_\text{max}\right)^{-1} G
+ \sum_{j=0}^k g_j z^j \in \ker F\;.
\]

Thus the 
polynomial $g_0 + g_1 z + \dots + g_k z^k$
(with ouput-nulling input sequence coefficients)
is considered to be equivalent to zero
if and only if the
state vector 
$\beta = \sum_{j=0}^k A^j B g_j$
is in $\ \Im \ \Pi_\text{max} = 
\cV^{*}\left(\Sigma\right)$.
I.e. $\beta \in \cC^{*}\left(\Sigma\right)
\cap \cV^{*}\left(\Sigma\right)$.
\qed 
\end{proof}

\medskip
\begin{remark}
Again this Theorem should be compared to Theorem 4
in \cite{AL-SC-84}.
\end{remark}
\medskip
\begin{corollary}
Assume that the pair $(C, A)$ is observable. Then the 
subspace $Z_\infty (F)$ is trivial if and only if
\begin{equation}\label{no_zero_infty}
\left\{
B \eta \ \mid \ D \eta = 0
\right\} \subset \cV^{*}\left(\Sigma\right) = \Im \left(\Pi_{max}\right)\,.
\end{equation}
\end{corollary}

\begin{proof}
The previous theorem implies that $Z_\infty (F)$ is
trivial if and only if 
$\cC^{*}\left(\Sigma\right) \subset \cV^{*}\left(\Sigma\right)$.
Since the set $\cC^{*}\left(\Sigma\right)$ contains those vectors which
are reachable from the origin with zero output, and the
set $\left\{ B \eta \ \mid \ D \eta = 0\right\}$ 
contains those vectors which can be reached 
from the origin in one step with zero output, we obtain that
if $Z_\infty (F)$ is trivial then
$\left\{
B \eta \ \mid \ D \eta = 0
\right\} \subset \cV^{*}\left(\Sigma\right)$.

Conversely, assume that $\left\{
B \eta \ \mid \ D \eta = 0
\right\} \subset \cV^{*}\left(\Sigma\right)$. We show by induction
that in this case $\cC^{*}\left(\Sigma\right) 
\subset \cV^{*}\left(\Sigma\right)$.
Consider a maximal solution 
$\left(\Pi_\text{max}, H_\text{max}, 
\Lambda_\text{max}\right)$ 
of (\ref{matrix_zero_eq_max}). According to Lemma 
\ref{lem:max_pi} $\cV^{*}\left(\Sigma\right) = 
\Im \left(\Pi_\text{max}\right)$.
Assume that $x \in \Im (\Pi_\text{max})$, i.e.
$x = \Pi_\text{max} \xi$ for some $\xi$, and
equations
\begin{eqnarray*}
x_+ &=& Ax + Bu \\
0 &=& Cx + Du
\end{eqnarray*}
hold true. Equation (\ref{matrix_zero_eq_max}) gives that
\[
\left[\begin{array}{cc} A & B \\ C & D
\end{array}\right]
\left[\begin{array}{c} \Pi_\text{max} \xi \\ 
H_\text{max}\xi
\end{array}\right] 
=
\left[\begin{array}{c} \Pi_\text{max}\Lambda_\text{max} \xi\\ 0
\end{array}\right]\;.
\]
Taking the difference
\[
\left[\begin{array}{c} B \\ D \end{array}\right]
(u - H_\text{max}\xi ) = 
\left[\begin{array}{c} x_+ 
- \Pi_\text{max}\Lambda_\text{max} \xi \\ 0
\end{array}\right]\;.
\]
The assumption implies that 
$x_+ -\Pi_\text{max}\Lambda_\text{max} \xi = 
B \left(u - H_\text{max}\xi\right) 
\in \Im (\Pi_\text{max})$,
giving that $x_+ \in \Im (\Pi_\text{max})$
and concluding the proof of the corollary.
\qed\end{proof}

\medskip
\begin{remark}
Assume that the pair $(C, A)$ is observable.
Consider the function $F_{\text{\bf r}}$ defined in
(\ref{f1_def}). The previous theorem 
together with part (ii) of 
Theorem \ref{thm:f1} implies
that
\[
Z_\infty \left( F_{\text{\bf r}}\right) 
= Z_\infty\left( F \right)\;.
\]
\end{remark}

\medskip
\subsection{Zero modules of $F$ vs. $F_{\text{\bf r}}$}
\label{subsec:zero_F_Fr}
\medskip
It is worth summarizing the 
connections between the various
zero modules of $F$ and $F_{\text{\bf r}}$. 
This is the subject of the
next proposition.

\begin{proposition}
Assume that $F$ has the realization
\[
F(z) \sim 
\left(
\begin{array}{c|c} A &   B
\\
\hline \rule{0cm}{.42cm}
   C & D
\end{array}
\right)
\]
where $(C, A)$ is an observable pair.

Then the function $F$ has the following factorization
\[
F = F_{\text{\bf r}} L_\beta^{*}
\]
where $L_\beta$ is a tall inner function, and
\begin{itemize}
\item[(i)]
\[
\cW\left(\ker F_{\text{\bf r}}\right) 
= \left\{ 0 \right\}\;;
\]
\item[(ii)]
\[
Z_\infty \left( F_{\text{\bf r}} \right) 
= Z_\infty ( F )\;;
\]
\item[(iii)]
if all the eigenvalues of $A$ are in the closed
left half-plane or conditions 
(\ref{eqn:disjoint_sp}) and 
(\ref{eqn:reach_AB_ACbar}) hold
then  \\
\centerline{the McMillan-degrees of $F$ and $F_{\bf r}$ are equal.}
\item[(iv)]
if the McMillan-degrees of $F$ and $F_{\bf r}$ are equal then
\begin{itemize}
\item[(a)]
\[
\cW\left(\Im\ F_{\text{\bf r}} \right) = \cW ( \Im \ F )
\]
and
\item[(b)]
the finite zero 
matrix of $F_{\text{\bf r}}$ is given as
(using the notation given in Theorem \ref{thm:f1}):
\[
\Lambda_f (F_{\text{\bf r}}) = 
\left[
\begin{array}{cc}
-\sigma^{-1}\left(\Lambda_k+\alpha_k\beta_k\right)^{*}
\sigma & 
\Lambda_{kf} + \sigma^{-1}
\left(\beta_k^{*}R_0^{*}+H_k^{*}\right)H_f \\
0 & \Lambda_f
\end{array}
\right]
\]
i.e. the finite zero matrix $\Lambda_f$ of $F$
is extended.
\end{itemize}
\end{itemize}
\end{proposition}

\bigskip
\section{Connections between the left and right
zero module spaces}\label{sec:left_right}

\medskip
In the previous sections the zero module spaces were
defined with respect to the transformation
$h \rightarrow Fh$. For a fixed matrix valued 
rational function
we might consider the left multiplication, i.e.
$g \rightarrow gF$, and define the 
corresponding zero modules
accordingly. The previous theorems and 
propositions can be
carried over to cover this case almost 
without any changes.

For example -- assuming that the realization of $F$
provided by the matrices $(A, B, C, D)$ is
minimal -- according to Corollary \ref{cor:ZF_kerF}
to characterize the spaces 
$Z_{\text{left}} (F) \oplus 
\cW (\ker_{\text{left}} F)$
(where the subtext ``left'' indicates 
that these spaces
are defined with respect to the left 
multiplication)
maximal solution of the equation
\begin{equation}\label{matrix_left_zero}
\left[ \Pi'_\text{max} , H'_\text{max} \right]
\left[\begin{array}{cc} A & B \\ C & D \end{array}
\right]
=
\left[ \Lambda'_\text{max} \Pi_\text{max}' , 0 \right]
\end{equation}
should be considered.

The following theorem connects various ``left'' and 
``right'' subspaces.

\begin{theorem}\label{thm:left_right}
Assume that $F$ has the realization
\[
F(z) \sim 
\left(
\begin{array}{c|c} A &   B
\\
\hline \rule{0cm}{.42cm}
   C & D
\end{array}
\right)\;.
\]
Consider a maximal solution of the equation
(\ref{matrix_left_zero}). Then
\[
\ker \Pi_\text{max}' = \cC^*\left(\Sigma\right)\, ,
\]
in other words $\left(\cV^{*}_{\text{left}}\right)^\perp\left(\Sigma\right) 
= \cC^{*}\left(\Sigma\right)$,
(with the obvious meaning of the notation 
$\cV^{*}_{\text{left}}\left(\Sigma\right)$).
\end{theorem}

\begin{proof}
First we show that 
$\ker \Pi_\text{max}' \subset \cC^{*}\left(\Sigma\right)$. 
To this aim we use
the following well-known construction from 
geometric control theory: define recursively 
the following
subspaces of row vectors
\[
\cL^r = \left\{ z \ \mid \ \exists \eta \ 
\text{such that }\ zA+\eta C \in \cL^{r-1}, 
\ \text{and} \ zB +\eta D = 0 \right\}\;,
\]
($\cL^0 = \CC^n$.)
Then $\cL^r \subset \cL^{(r-1)}$ and
$\cap_r \cL^r$ equals to the space spanned by
the rows of $\Pi_\text{max}^{'}$. 

We prove by induction that for any $r$ the 
vectors orthogonal to
the subspace $\cL^r$ are in the subspace 
$\cC^{*}\left(\Sigma\right)$.
Obviously, 
\[
\cL^1 = \left\{ z \ \mid \ \exists \eta: \ 
zB+\eta D = 0 \right\}\;.
\]
Now if the vector $\alpha$ is orthogonal to 
the elements of
$\cL^1$ then the equation
\[
\left[ z, \eta \right] \left[\begin{array}{c} B \\ D
\end{array}\right] = 0
\]
implies that 
\[
\left[ z, \eta \right] \left[\begin{array}{c} \alpha \\ 0
\end{array}\right] = 0\;.
\]
Consequently, there exists a vector $\zeta$ such that
\[
\left[\begin{array}{c} B \\ D \end{array}\right] \zeta = 
\left[\begin{array}{c} \alpha \\ 0 \end{array}\right]\;.
\]
In other words $\alpha = B\zeta$ can be reached from
the origin in one step with zero output, thus 
$\alpha  \in \cC^{*}\left(\Sigma\right)$.

For each $r$ consider a basis in $\cL^r$ 
and form
the matrix $\Pi^r$ containing the 
basis-vectors as its rows.
Then 
\[
\cL^r = \left\{ z \ \mid \ 
\exists \eta, \lambda \ \text{such that } \ 
zA + \eta C = \lambda \Pi^{(r-1)}, 
\ \text{and} \ zB + \eta D =0 \right\}\;.
\]
Assume that the vectors orthogonal to 
$\cL^{(r-1)}$ are
in the subspace $\cC^{*}\left(\Sigma\right)$.
Now if $\alpha$ is orthogonal to the elements of
$\cL^r$ then the equation
\[
\left[ z, \eta, -\lambda\right]
\left[\begin{array}{cc} A & B \\ C & D \\ 
\Pi^{(r-1)} & 0
\end{array}\right] = \left[ 0, 0 \right]
\]
should imply that
\[
\left[ z, \eta, -\lambda \right] 
\left[\begin{array}{c} \alpha \\ 0 \\ 0 
\end{array}\right]= 0\;.
\]
Consequently, there exist $\zeta, \xi$ such that
\[
\left[\begin{array}{cc} A & B \\ C & D \\ 
\Pi^{(r-1)} & 0
\end{array}\right]
\left[ \begin{array}{c} \zeta \\ \xi 
\end{array}\right]
=
\left[\begin{array}{c} \alpha \\ 0 \\ 0 
\end{array}\right]\;.
\]
In details, $\Pi^{(r-1)} \zeta = 0$ thus 
$\zeta \in \cC^{*}\left(\Sigma\right)$.
Also, $A\zeta + B\xi = \alpha$, $C\zeta + D\xi = 0$, so
$\alpha$ can be reached from $\zeta$ in 
one step with zero
output. The induction hypothesis gives 
that $\alpha \in \cC^{*}\left(\Sigma\right)$,
as well. The identity $\ker \Pi_\text{max}' 
= \cup \ker\Pi^r$, and
$\ \ker \Pi^r \supset \ker \Pi^{(r-1)}$ 
implies that
\[
\ker \Pi_\text{max}' \subset \cC^{*}\left(\Sigma\right)\;.
\]

Conversely, assume that the vector 
$\alpha \in \cC^{*}\left(\Sigma\right)$.
We are going to show that $\alpha \in \ker \Pi'$, 
where $\left(\Pi', H', \Lambda'\right)$ is any
solution of (\ref{matrix_left_zero}) implying
that $\cC^{*}\left(\Sigma\right) \subset \ker \Pi_\text{max}'$,
especially $\cC^{*}\left(\Sigma\right) \subset \ker \Pi_\text{max}'$.
According to the definition of $\cC^{*}\left(\Sigma\right)$ 
there exists a 
finite input sequence producing zero output and directing
the origin to the vector $\alpha$. Denoting by
$\eta_0, \eta_1, \dots , \eta_j$ this sequence of inputs
and by $\xi_0, \xi_1, \dots , \xi_{j-1}$ the sequence of 
state vectors produced by using this input sequence
the following system of equations holds:
\[
\left[\begin{array}{cc} A & B \\ C & D 
\end{array}\right]
\left[
\begin{array}{ccccc} \xi_0 & \xi_1 & \dots & \xi_{j-1}& 0\\
\eta_0 & \eta_1 & \dots & \eta_{j-1} & \eta_j
\end{array}
\right]
=
\left[\begin{array}{cccc} \alpha & \xi_0 & 
\dots & \xi_{j-1} \\
0 & 0 & \dots & 0
\end{array}
\right]\;.
\] 
Multiplying this equation by $\left[ \Pi' , H'\right]$
from the left and using 
(\ref{matrix_left_zero}) we obtain 
that
\[
\left[ \Pi'\alpha, \Pi'\xi_0, 
\dots , \Pi'\xi_{j-1}\right]
=
\left[\Lambda' \Pi'\xi_0, \Lambda'\Pi'\xi_1, \dots ,
\Lambda'\Pi'\xi_{j-1}, 0\right]\;.
\]
Consequently, 
\[
\Pi' \xi_{j-1} = 0 , \ \dots , \ \Pi'\xi_0 = 0, \ 
\Pi'\alpha =0\;.
\]
I.e. $\cC^{*}\left(\Sigma\right) \subset \ker \Pi'$, 
concluding the proof
of the theorem.
\qed\end{proof}

Similar proof gives the following statement:
\[
\left(\cC_{\text{left}}^{*}\right)^\perp\left(\Sigma\right) 
= \cV^{*}\left(\Sigma\right)\;,
\]
implying the following corollary:
\begin{corollary}
Assume that the realization 
\[
F(z) \sim 
\left(
\begin{array}{c|c} A &   B
\\
\hline \rule{0cm}{.42cm}
   C & D
\end{array}
\right)
\]
is minimal. Then
\begin{eqnarray}
\dim Z(F) &=& \dim Z_\text{left}(F)\;,\\
\dim \cW(\ker F) &=& \dim \cW (\Im_\text{left} F)\;,\\
\dim Z_\infty (F ) &=& \dim Z_{\infty, \text{left}}(F)\;,\\
\dim \cW (\Im F) &=& \dim \cW (\ker_\text{left} F)\;.
\end{eqnarray}
\end{corollary}

\bigskip
The following theorem shows that there is a
deeper connection between the left and 
right finite zero spaces of
$F$.

\begin{theorem}
Assume that the realization of $F$ given by
\[
F(z) \sim 
\left(
\begin{array}{c|c} A &   B
\\
\hline \rule{0cm}{.42cm}
   C & D
\end{array}
\right)
\]
is minimal. Then the left and right finite zero
matrices $\Lambda_f$, $\Lambda_{f,\text{left}}$ are
similar.
\end{theorem}

\begin{proof}
Consider maximal solutions of equations
(\ref{matrix_zero_eq_fzk}), 
(\ref{matrix_kernel_eq}) and
(\ref{matrix_left_zero}) and the
corresponding ``left'' version of
(\ref{matrix_kernel_eq}).

\begin{eqnarray}\label{eq:matrix_kerright}
\left[
\begin{array}{cc} A & B \\ C & D
\end{array}\right]
\left[\begin{array}{c} \Pi_{fzk} \\ H_{fzk}
\end{array}\right] &=& 
\left[\begin{array}{c} \Pi_{fzk} \Lambda_{fzk} \\ 0
\end{array}\right]\;,\quad 
\left[\begin{array}{c} B \\ D 
\end{array}\right]
R_0 = 
\left[\begin{array}{c}
\Pi_{fzk} \alpha_0 \\ 0 
\end{array}\right]\;,\\
\rule{0in}{4ex}\label{eq:matrix_kerleft}
\left[ \Pi_{fzk}'\;, \; H_{fzk}'\right]
\left[
\begin{array}{cc} A & B \\ C & D
\end{array}\right] &=&
\left[\Lambda_{fzk}'\Pi_{fzk}'\;,\; 0 \right]\;,\quad 
R'_0\left[ C \;, D\right]= 
\left[\alpha'_0\Pi_{fzk}'\;,\ 0\right]\;,
\end{eqnarray}
where for the sake of simplicity the ``left'' is
indicated by the notation $\cdot^{'}$

Without loss of generality we might 
assume that these matrices are
partitioned as it is
described in Remark \ref{rem:partition}.
(Applying it also to the ``left'' structure, as well.):
\begin{eqnarray}
\Lambda_{fzk} &=& \left[\begin{array}{cc}
\Lambda_k & \Lambda_{kf} \\ 0 &
\Lambda_f
\end{array}\right]\ 
\alpha_0 = \left[\begin{array}{c}
\alpha_k \\ 0 \end{array}\right] \\
\Lambda_{fzk}' &=& \left[\begin{array}{cc}
\Lambda'_k & 0 \\ 
\Lambda'_{kf} & \Lambda'_f
\end{array}\right]\;,\; 
\alpha'_0 = \left[\alpha'_k \;,\; 0 \right]\;,
\end{eqnarray}

where the pair $(\Lambda_k, \alpha_k)$ 
is controllable,
$(\alpha'_k, \Lambda'_k)$ is observable.

Partitioning the matrices 
$\Pi_{fzk}, H_{fzk}, \Pi_{fzk}', H_{fzk}'$ 
accordingly, we get that

\begin{eqnarray}
\Im \Pi_k = \cC^{*}\left(\Sigma\right) 
\cap \cV^{*}\left(\Sigma\right)\;, && 
\Im \Pi_{fzk} = \cV^{*}\left(\Sigma\right)\;, \\
\Im_{\text{left}} \Pi'_k = \cC_{\text{left}}^{*}\left(\Sigma\right)
\cap \cV_{\text{left}}^{*}\left(\Sigma\right)\;, &&
\Im_{\text{left}} \Pi_{fzk}' = \cV_{\text{left}}^{*}\left(\Sigma\right)\;,
\end{eqnarray}

and Theorem \ref{thm:left_right} implies that
\[
\Pi_{fzk}' \Pi_k = 0 \;,\quad \Pi'_k \Pi_{fzk} = 0
\]
Multiplying the first equation in 
(\ref{eq:matrix_kerright})
from the left by $\left[ \Pi_{fzk}', H_{fzk}'\right]$ 
and using the
first equation in (\ref{eq:matrix_kerleft})
we obtain that
\[
\left[\begin{array}{c} 
\Pi'_k \\ \Pi'_f \end{array}\right]
\left[ \Pi_k\;,\ \Pi_f\right]
\left[\begin{array}{cc}
\Lambda_k & \Lambda_{kf} \\
0 & \Lambda_f
\end{array}\right] =
\left[\begin{array}{cc}
\Lambda'_k & 0 \\
\Lambda'_{kf} & \Lambda'_f
\end{array}\right] 
\left[\begin{array}{c}
\Pi'_k \\ \Pi'_f
\end{array}\right]
\left[ \Pi_k\;,\ \Pi_f\right]\;.
\]
Shortly
\[
\Pi'_f\Pi_f \Lambda_f = 
\Lambda'_f \Pi'_f\Pi_f\;.
\]
Since Theorem \ref{thm:left_right} gives also that
the matrix $\Pi'_f\Pi_f$ is square and nonsingular
the similarity of the matrix
$\Lambda_f$ and $\Lambda'_f$ is obtained.
In fact
\[
\left(\Pi'_f\Pi_f\right) \Lambda_f 
\left(\Pi'_f\Pi_f\right)^{-1} = 
\Lambda'_f \;.
\] 

In other words, the finite left and 
right zero matrices of
the function $F$ are similar to each other.
\qed\end{proof}

\bigskip

\subsection{Connection between the values of
$F$ and $K_0$ at a given point $\lambda\in\CC$}
\label{subsec:interp_vs_zero}

Assume that the matrices $(A, B, C, D)$ 
provide a {\it minimal} realization
of $F$ and consider the function $K_0$
given in Remark \ref{rem:K0} 
(or $K_\beta$ defined in Remark \ref{rem:Kbeta})
``generating'' the kernel of
$F$ (in the sense that for any 
q-tuple $g$ of rational functions
for which $Fg \equiv 0$ holds there 
exists a (vector-valued) rational 
function $h$ such that $g = K_0 h$. 
The converse statement
obviously holds)

Since $F K_0=0$, if both functions $F$ 
and $K_0$ are analytic at a
given
$\lambda'\in \CC$,  the same connection holds for the
values of these functions taken at $\lambda'$. I.e.  
\[
F (\lambda') K_0 (\lambda') = 0\;.
\]
In other words the row-vectors of $F(\lambda')$
are orthogonal to the column-vectors of $K_0(\lambda')$.

\medskip
More generally, consider a solution of the set
of equations:
\begin{equation}\label{left_interpol}
\left[ Y'\;,\ Z'\right]
\left[\begin{array}{cc} A & B \\
C & D \end{array}\right] =
\left[ \lambda' Y'\;, h'\right]
\end{equation}
(where now $\lambda'$ can also be a matrix) 
implying obviously that
\[
\left(zI - \lambda'\right)^{-1}
\left(Z'F(z) - h'\right) = 
- Y'\left(zI-A\right)^{-1} B\;.
\] 
Thus, if the spectra of $A$ and $\lambda'$ 
are disjoint then $h'$
determines the ``directional''
values and derivatives of $F$ taken at the 
eigenvalues of $\lambda'$.

Now multiplying from the right by 
$\left[\begin{array}{c} \Pi_{k} \\ H_{k}
\end{array}\right]$ and
by 
$\left[\begin{array}{c} 0 \\ R_0
\end{array}\right]$ we obtain the 
following equations
\begin{eqnarray*}
\lambda' Y'\Pi_{k} + h' H_{k} &=& Y'\Pi_{k} \Lambda_{k}\\
h' R_0 & = & Y'\Pi_{k} \alpha_k\;.
\end{eqnarray*}
In other words
\begin{equation}\label{right_zero}
\left[ Y'\Pi_{k}\;, \ -h'\right]
\left[\begin{array}{cc} \Lambda_{k} & \alpha_0 \\
H_{k} & R_0 \end{array}\right] =
\left[ \lambda' Y'\Pi_{k}\;, \ 0\right]\;,
\end{equation}
thus
\[
\left(zI-\lambda'\right)^{-1} h' 
\left(R_0+H_k\left(zI-\Lambda_k\right)^{-1}
\alpha_k\right)
= Y'\Pi_k \left(zI-\Lambda_k\right)^{-1}\alpha_k\;.
\]
Shortly
\[
\left(zI-\lambda'\right)^{-1} h' 
K_0(z)
= Y'\Pi_k \left(zI-\Lambda_k\right)^{-1}\alpha_k\;.
\]
Now, if the spectra of $\lambda'$ 
and $\Lambda_k$ are disjoint
then the pair $(\lambda', h')$ is a right-zero pair
of $K_0$.

Summarizing these considerations: 
if the spectra of $\lambda'$ and that of $A$
and $\Lambda_k$ are disjoint then
the assumption $\left(zI-\lambda'\right)^{-1}
\left(Z' F(z) - h'\right)$ is analytic on the 
set of eigenvalues of $\lambda'$ implies that
$\left(zI-\lambda'\right)^{-1}h' K_0(z)$ 
is analytic there.

In that special case, when $\lambda'$ is a matrix 
in Jordan-form, then 
equations (\ref{left_interpol}) 
and (\ref{right_zero}) establish connections
between the ``directional''derivatives of $F$ 
and $K_0$ taken at  
the eigenvalues of $\lambda'$.

The following theorem shows 
that the converse statement also
holds true. Under some conditions, 
if a pair is a right zero pair
of the function $K_0$, then the same pair determines
also interpolation values of the function $F$, i.e. at 
the same locations using appropriately defined directions 
the directional values of $F$ coincide with the zero directions of $K_0$.

\begin{theorem}\label{thm:F_vs_K_values}
Assume that the realization of $F$ given by
\[
F(z) \sim 
\left(
\begin{array}{c|c} A &   B
\\
\hline \rule{0cm}{.42cm}
   C & D
\end{array}
\right)
\]
is minimal.
Consider maximal solutions $\left(\Pi_\text{max}, H_\text{max},
\Lambda_\text{max}\right)$ of (\ref{matrix_zero_eq_max}) and
$(R_0, \alpha_0)$ of (\ref{matrix_kernel_eq}) for
which $\ker \Pi_\text{max} = \left\{ 0 \right\}$.
Define the function $K_0$ according to Remark \ref{rem:K0}.

Assume that for some matrices $\lambda'$ and $h'$
the product
\[
\left(zI-\lambda'\right)^{-1}h' K_0(z) \quad
\text{is analytic on the spectrum of}\ \lambda'\;.
\]
If the spectrum of $\lambda'$ is disjoint from
that of $\Lambda_f$ and $A$, then 
there exists a matrix $Z'$ such that the product
\[
\left(zI-\lambda'\right)^{-1}
\left(Z' F(z)-h'\right)\quad
\text{is analytic on the spectrum of}\
\lambda'\;.
\]
\end{theorem}

\begin{proof}
The proof of the theorem is 
based on the following lemma which 
is valid under more general assumptions, as well.

\begin{lemma}\label{lem:sigma_h_lambda}
Assume that the realization of $F$ given by
\[
F(z) \sim 
\left(
\begin{array}{c|c} A &   B
\\
\hline \rule{0cm}{.42cm}
   C & D
\end{array}
\right)
\]
is minimal.
Consider maximal solutions 
$\left(\Pi_\text{max}, H_\text{max},
\Lambda_\text{max}\right)$ of 
(\ref{matrix_zero_eq_max}) and
$(R_0, \alpha_0)$ of (\ref{matrix_kernel_eq}) for
which $\ker \Pi_\text{max} = \left\{ 0 \right\}$.
(Let us recall that under the minimality
assumption the subscripts $.\text{max}$ and
$._{fzk}$ mean the same.)
Assume that the matrices $\sigma', h', \lambda'$ 
provide a
solution of the equations
\begin{equation}\label{eq:left_zero_K}
\left[ \sigma'\;, \ -h'\right]
\left[\begin{array}{cc} \Lambda_\text{max} & \alpha_0 \\
H_\text{max} & R_0 \end{array}\right] =
\left[ \lambda' \sigma'\;, \ 0\right]\;,
\end{equation}
Then there exist matrices $Y', Z'$ such that
$Y'\Pi_\text{max} = \sigma'$ and equations
(\ref{left_interpol}) hold.
\end{lemma}

\begin{proof}
Before proving the lemma let us observe that 
multiplying the equation 
(\ref{left_interpol}) from the right by 
$\left[\begin{array}{c} \Pi_{fzk} \\ H_{fzk}
\end{array}\right]$ and
by 
$\left[\begin{array}{c} 0 \\ R_0
\end{array}\right]$ we obtain the 
following equations
\begin{eqnarray*}
\lambda' Y'\Pi_{fzk} + h' H_{fzk} &=& Y'\Pi_{fzk} 
\Lambda_{fzk}\\
h' R_0 & = & Y'\Pi_{fzk} \alpha_0\;.
\end{eqnarray*}
In other words
\begin{equation}\label{right_zero2}
\left[ Y'\Pi_{fzk}\;, \ -h'\right]
\left[\begin{array}{cc} \Lambda_{fzk} & \alpha_0 \\
H_{fzk} & R_0 \end{array}\right] =
\left[ \lambda' Y'\Pi_{fzk}\;, \ 0\right]\;.
\end{equation}
Thus the present lemma essentially states 
-- using the assumption that the realization of $F$
is minimal --  that
equations (\ref{left_interpol}) and
(\ref{eq:left_zero_K}) are equivalent.

\bigskip
For proving the lemma first notice
that since according to our assumption $\ker (\Pi_\text{max}) = 
\left\{ 0 \right\}$ there exists a matrix
$Y'_1$ such that $Y'_1\Pi_\text{max} = \sigma'$.
Then $\sigma' \alpha_0 = Y'_1\Pi_\text{max} \alpha_0 =
Y'_1BR_0$. Thus
\[
\left(Y'_1B - h'\right) R_0 = 0
\]
The maximality of the solution of equation 
(\ref{matrix_kernel_eq}) gives that 
$\ker \left[\begin{array}{cc} B & \Pi_{max} \\
D & 0 \end{array}
\right] =
\Im \left[\begin{array}{c} R_0 \\ 0 \end{array}\right]$
thus any row vector
orthogonal to the columns of $R_0$ can be written 
in the form $\eta B +\xi D$, where $\eta \Pi_\text{fzk} = 0$.
Thus there exist matrices $Y'_2, Z'_1$ such that
\begin{equation}\label{pre_YZh}
Y'_1B - h' = Y'_2B + Z'_1 D\;, \quad Y'_2 \Pi_{fzk} = 0
\end{equation}
Also the first equation in (\ref{right_zero2}) gives that
\[
Y'_1\Pi_\text{fzk} \Lambda_\text{fzk} 
-h'H_\text{fzk} -\lambda'Y'_1\Pi_\text{fzk}=0\;.
\]
Expressing $\Pi_\text{fzk}\Lambda_\text{fzk} = A \Pi_{fzk} + B H_{fzk}$ 
and $h'$ from (\ref{pre_YZh}) we get that
\[
Y'_1 A \Pi_\text{fzk} + Y'_1 BH_\text{fzk} 
- Y'_1B H_\text{fzk} + Y'_2 B H_\text{fzk} + Z'_1DH_\text{fzk} - 
\lambda' Y'_1\Pi_\text{fzk} =0\;.
\]
Thus
\[
Y'_1 A\Pi_\text{fzk}  - \lambda' Y'_1\Pi_\text{fzk} 
+ Y'_2 \Pi_\text{fzk} \Lambda_\text{fzk} 
- Y'_2 A \Pi_\text{fzk} - Z'_1 C \Pi_\text{fzk} =0 \;. 
\]
Using the identity $Y'_2\Pi_\text{fzk} =0$ we arrive at 
the following equation
\begin{equation}\label{pre_YZlambda}
\left[
\left(Y'_1-Y'_2\right)A - Z'_1 C 
-\lambda'\left(Y'_1-Y'_2\right)
\right]\Pi_\text{fzk} = 0
\end{equation}
Now Theorem \ref{thm:left_right} implies that the row
vectors orthogonal to $\Im (\Pi_\text{fzk})$ are in 
$\cC_{\text{left}}^{*}\left(\Sigma\right)$. 
I.e. there exists an integer $l$ and sequences of matrices
$\xi_0=0, \xi_1, \dots ,\xi_l$ and 
$\mu_0, \mu_1, \dots , \mu_{l-1}$ such that
\begin{equation}\label{eq_ximu}
\left[ \xi_j\;,\ \mu_j\right]
\left[\begin{array}{cc} A & B \\
C & D \end{array}\right] =
\left[ \xi_{j+1}\;, \ 0\right]\;,
\quad j=0, 1, \dots , l-1\;.
\end{equation}
and
\[
\xi_l = \left(Y'_1-Y'_2\right)A - Z'_1 C 
-\lambda'\left(Y'_1-Y'_2\right)\;.
\]
This latter equation together with
(\ref{pre_YZh}) can be written as follows
\begin{equation}\label{eq:YZlambda}
\left[\left(Y'_1-Y'_2\right)\;,\ -Z'_1
\right]
\left[\begin{array}{cc} A & B \\
C & D \end{array}\right] =
\left[\xi_l + 
\lambda'\left(Y'_1-Y'_2\right) 
\;,\ h'
\right]
\end{equation}
Multiplying (\ref{eq_ximu}) from the left by
$\left(\lambda'\right)^{l-1-j}$ taking the sum
from $j=0$ to $l-1$ and subtracting 
it from 
(\ref{eq:YZlambda}) we obtain that
\begin{multline}
\left[
\left(Y'_1-Y'_2 - \sum_{j=0}^{l-1} 
\left(\lambda'\right)^{l-1-j} \xi_j\right)\;,
-Z'_1 - \sum_{j=0}^{l-1} 
\left(\lambda'\right)^{l-1-j} \mu_j
\right]
\left[\begin{array}{cc} A & B \\
C & D \end{array}\right] \\ 
=
\left[\xi_l + 
\lambda'\left(Y'_1-Y'_2\right) -
\sum_{j=0}^{l-1} \left(\lambda'\right)^{l-1-j}\xi_{j+1}
\;,\ h'
\right]
\end{multline}
Introducing the notation (using that $\xi_0 = 0$)
\begin{eqnarray*}
Y' &=& Y'_1 -Y'_2 - \sum_{j=1}^{l-1} 
\left(\lambda'\right)^{l-1-j} \xi_j \\
Z' &=& -Z'_1 - \sum_{j=0}^{l-1} 
\left(\lambda'\right)^{l-1-j} \mu_j
\end{eqnarray*}
we get that
\[
\left[ Y'\;,\ Z'
\right]
\left[\begin{array}{cc} A & B \\
C & D \end{array}\right] 
=
\left[
\lambda' Y'\;,\ h'
\right]
\]
concluding the proof of the lemma.
\qed\end{proof}

Let us return to the proof of the theorem.
We might assume w.l.o.g. that the matrices
are partitioned according to 
(\ref{special_form1}) and (\ref{special_form2}).
Then the function $K_0$ has the 
{\it minimal} realization
$K_0 = R_0 + H_k \left(zI-\Lambda_k\right)^{-1}\alpha_k$.
According to part (i) of Theorem \ref{zeros_description},
if for the pair $(\lambda', h')$ the product
$\left(zI-\lambda'\right)^{-1}h' K_0(z)$ is analytic 
at the eigenvalues of $\lambda'$
then there exists a solution $\sigma'$ of the equation
\[
\left[ \sigma',\ -h'\right]
\left[ \begin{array}{cc}
\Lambda_k & \alpha_k \\
H_k & R_0
\end{array}\right] = 
\left[ \lambda'\sigma', \ 0\right]\;. 
\]
Using the assumption that the spectra of
$\lambda'$ and $\Lambda_f$ are disjoint
we get that the Sylvester-equation
\[
\lambda' \sigma^{''} - \sigma^{''}\Lambda_f 
= \sigma' \Lambda_{kf} - h' H_f
\]
has also a solution in $\sigma^{''}$.
In other words the equation
\[
\left[ \sigma',\ \sigma^{''},\ -h'\right]
\left[\begin{array}{ccc}
\Lambda_k & \Lambda_{kf} & \alpha_k \\
0 & \Lambda_f & 0 \\
H_k & H_f & R_0
\end{array}
\right] =
\left[ \lambda' \sigma',\ 
\lambda'\sigma^{''},\ 0\right]
\]
holds. Applying Lemma \ref{lem:sigma_h_lambda}
we obtain that there exists matrices
$Y'$, $Z'$ such that $Y' \Lambda_\text{max} = 
\left[ \sigma',\ \sigma^{''}\right]$
and equation (\ref{left_interpol}) holds.

Invoking now part (ii) of Theorem 
\ref{zeros_description} 
(or directly computing the product) 
-- using that $\lambda'$ and $A$ have no 
common eigenvalues we obtain that 
\[
\left(zI-\lambda'\right)^{-1}
\left(Z' F(z)-h'\right) = 
-Y' \left( zI - A\right)^{-1} B
\]
is analytic on the spectra of $\lambda'$, concluding
the proof of the theorem. \qed
\end{proof}

\begin{remark}\label{rem:all_K}
Let us note that in the previous theorem instead
of $K_0$ any other function $K_\beta$ can be used, due to 
the fact that the matrices 
$\left(\Pi_{fzk}, H_{fzk}+\alpha_0\beta, 
\Lambda_{fzk}+R_0\beta\right)$
are also maximal solutions of the 
equation (\ref{matrix_zero_eq_max}) due to the assumption
that a minimal realization of $F$ was considered.
\end{remark}

\begin{corollary}
Consider a complex number $\lambda\in \CC$ which is not a
finite zero of $F$, and assume that the functions $F$ and $K_0$
are analytic at $\lambda$. Choosing $\lambda' = \lambda I$
(where $I$ has appropriate size) the previous theorem 
gives that the row-space spanned by the row-vectors
of $F(\lambda)$ generate the 
\underline{\bf orthogonal complement} of
the column space generated by the column-vectors of
$K_0(\lambda)$.
\end{corollary}

\bigskip
\subsection{Further elimination via factorization:
$\cW (\ker_{\rm left} F )$ }

In Section \ref{sct:elim_ker} a special factorization
of function $F$ of the form
$F = F_{\text{\bf r}} L_\beta^{*}$ was discussed, where 
the inner function $L_\beta$ was constructed via the
square inner extension of the function $K_\beta$.
(This latter one generates the module 
$\cW \left(\ker F\right)$. See Theorems 
\ref{thm:char_ker2} and 
\ref{inner_fnc_kernel}.)  

Applying the same idea we can eliminate the left kernel module of $F$, as well. 
But in order to eliminate both the left and right kernel modules of $F$ 
at the same time
we have to consider the left kernel-module of $F_{\br r} = F L_{\beta}^{*}$.
To this aim first we have to consider \underline{maximal} solution of the "left"
version of equation (\ref{matrix_zero_eq_max}) for the realization 
(\ref{f1_real}) of $F_{\bf r}$ given in Theorem \ref{thm:f1}. 
As we have seen earlier, this maximal solution is connected to the subspace
$\cV^{*}_{\rm left}\left(\Sigma_{\bf r}\right)$.
Theorem \ref{thm:f1} provides explicit 
connections between the various
subspaces used in geometric control theory
(maximal output-nulling 
controlled invariant subspace, 
minimal input-containing subspace, 
maximal output-nulling reachability subspace) 
determined by the given realizations of
the functions
$F$ and $F_{\text{\bf r}}$, especially showing that
while $\cV^{*}\left(\Sigma_{\text{\bf r}}\right) =
\cV^{*}\left(\Sigma\right)$, the minimal input-containing subspace
reduces, 
$\cC^{*}\left(\Sigma_{\text{\bf r}}\right) \subset
\cC^{*}\left(\Sigma\right)$, in such a way that the 
intersection 
$\cR^{*}\left(\Sigma_{\text{\bf r}}\right) = 
\cV^{*}\left(\Sigma_{\text{\bf r}}\right) \cap 
\cC^{*}\left(\Sigma_{\text{\bf r}}\right)$ 
becomes trivial. As a consequence of this
-- using Theorem \ref{thm:left_right} --
$\cV^{*}_\text{left}\left(\Sigma_{\bf r}\right)$ becomes
larger than $\cV^{*}_\text{left} \left(\Sigma\right)$.  

Theorem \ref{thm:fzk_left} provides a detailed
picture of this question in terms of state-space matrices
solving the "left" version of equation (\ref{matrix_zero_eq_max})
for the realization (\ref{f1_real}) of $F_{\bf r}$ given in Theorem
\ref{thm:f1}.

To formulate this theorem we first need an auxiliary statement formulated
as a corollary of the following version 
of Lemma \ref{lem:sigma_h_lambda}.

\begin{lemma}\label{lem:sigma_h_lambda_mod}
Assume that the realization of $F$ given by
\[
F(z) \sim 
\left(
\begin{array}{c|c} A &   B
\\
\hline \rule{0cm}{.42cm}
   C & D
\end{array}
\right)
\]
is minimal.
Consider maximal solutions 
$\left(\Pi_\text{max}, H_\text{max},
\Lambda_\text{max}\right)$ of 
(\ref{matrix_zero_eq_max}) and
$(R_0, \alpha_0)$ of (\ref{matrix_kernel_eq}) for
which $\ker \Pi_\text{max} = \left\{ 0 \right\}$
assuming that the column vectors of $R_0$ are orthonormal
and the matrices are partitioned according to 
(\ref{special_form1}) and (\ref{special_form2}).

Assume that the matrices $\sigma', h', \lambda'$ 
provide a
solution of the equations
\begin{equation}\label{eq:left_zero_K_mod}
\left[ \sigma'\;, \ -h'\right]
\left[\begin{array}{cc} \Lambda_k & \alpha_k \\
H_k & R_0 \end{array}\right] =
\left[ \lambda' \sigma'\;, \ 0\right]\;,
\end{equation}
Then there exist matrices $Y', Z', V'$ such that
$Y'\Pi_k = \sigma'$, 
the row vectors of $V'$ are in $\cV^{*}_\text{\rm left}\left(\Sigma\right)$,
and equations
\begin{equation}\label{left_interpol_mod}
\left[ Y'\;,\ Z'\right]
\left[\begin{array}{cc} A & B \\
C & D \end{array}\right] =
\left[ \lambda' Y'\;+ V' , h'\right]
\end{equation}
hold.
\end{lemma}

\begin{proof}
Let us make the obvious changes in the proof of 
Lemma \ref{lem:sigma_h_lambda}, i.e. instead of
considering $\Lambda_\text{max}, ...$ use
$\Lambda_k , ...$.  

Since according to
our assumption 
$\ker \left(\Pi_k\right) = \left\{ 0 \right\}$ there 
exists a matrix $Y_1^{'}$ such that 
$Y_1^{'} \Pi_k = \sigma'$. 

Following the steps in the 
previous proof we -- instead of 
equation (\ref{pre_YZlambda}) -- arrive at the equation
\[
\left[
\left(Y'_1-Y'_2\right)A - Z'_1 C 
-\lambda'\left(Y'_1-Y'_2\right)
\right]\Pi_k = 0\;,
\]
where -- as before --
$Y_2^{'} \Pi_\text{\rm max} = 0$, i.e. the rows
of $Y_2^{'}$ are in 
$\cC^{*}_\text{\rm left}\left(\Sigma\right)$.

Now according to Theorem \ref{thm:left_right} and its 
immediate consequence the row vectors
orthogonal to $\Im \left( \Pi_k\right)
= \cV^{*}\left(\Sigma\right) \cap 
\cC^{*}\left(\Sigma\right)$ are
in $\cC^{*}_\text{\rm left}\left(\Sigma\right) \vee 
\cV^{*}_\text{\rm left}\left(\Sigma\right)$
we have that
\[
\left(
\left(Y'_1-Y'_2\right)A - Z'_1 C 
-\lambda'\left(Y'_1-Y'_2\right)
\right) = \xi_l + V'
\]
for some matrices $\xi_l$ and $V'$ where
the rows of $\xi_l$ are in 
$\cC^{*}_\text{\rm left}\left(\Sigma\right)$,
while those of $V'$ are in 
$\cV^{*}_\text{\rm left}\left(\Sigma\right)$.

Thus instead of (\ref{eq:YZlambda})
we obtain that
\[
\left[\left(Y'_1-Y'_2\right)\;,\ -Z'_1
\right]
\left[\begin{array}{cc} A & B \\
C & D \end{array}\right] =
\left[\xi_l + 
\lambda'\left(Y'_1-Y'_2\right) 
+ V'\;,\ h'
\right]\;.
\]
Embedding the rows of $\xi_l$ into sequences
in $\cC^{*}_\text{\rm left}\left(\Sigma\right)$ and continuing
the proof as it was done in Lemma
\ref{lem:sigma_h_lambda} we get
that equation
\[
\left[ Y'\;,\ Z'
\right]
\left[\begin{array}{cc} A & B \\
C & D \end{array}\right] 
=
\left[
\lambda' Y' + V'\;,\ h'
\right]
\]
holds, 
concluding the proof of the present lemma.
\qed
\end{proof}

\begin{remark}\label{rem:R_compl_V}
Let us point out that the matrices $Y', Z'$ and $V'$
can be chosen in such a way that for the matrix $V'$
the following more stringent condition holds:
considering any (maximal) complementary subspace 
of $\cR^{*}_\text{\rm left}\left(\Sigma\right)$
in $\cV^{*}_\text{\rm left}\left(\Sigma\right)$ the row vectors of 
$V'$ are in this subspace.
\end{remark}

Let us observe that the equations (\ref{inner_eq1}) and
(\ref{inner_eq2}) (or equivalently the Riccati-equation
(\ref{ric}) and (\ref{beta_def})) can be written as
\[
\left[\sigma , \left(H_k+R_0\beta_k\right)^{*}
\right]
\left[
\begin{array}{cc}
\Lambda_k + \alpha_k\beta_k & \alpha_k \\
H_k+R_0\beta_k & R_0
\end{array}\right]
=
\left[ - \left(\Lambda_k+\alpha_k\beta_k\right)^{*} \sigma ,
0 \right].
\]
Thus Lemma \ref{lem:sigma_h_lambda_mod}
can be applied giving the following
corollary.

\begin{corollary}\label{cor:y1_z1}
Consider a {\bf minimal realization} of $F$ given as
\[
F(z) \sim 
\left(
\begin{array}{c|c} A &   B
\\
\hline \rule{0cm}{.42cm}
   C & D
\end{array}
\right)
\]
Assume that the columns of the function
$H_{fzk}\left(zI-\Lambda_{fzk}\right)^{-1}$
provide a basis in 
$Z(F)\oplus \cW\left(\text{ker} F\right)$.
Let $\Pi_{fzk}$ be the corresponding solution of
(\ref{matrix_zero_eq_fzk}). 
Consider a maximal solution -- in terms
of $\alpha_0$ and $R_0$ -- of the equation
(\ref{matrix_kernel_eq}) assuming 
-- w.l.o.g. -- that
the column-vectors of the matrix $R_0$ are orthonormal
and the matrices are partitioned according
to (\ref{special_form1}) and 
(\ref{special_form2}).

Denote by $\sigma$ the positive definite 
solution of the Riccati-equation (\ref{ric}). Set
\[
\beta_k^{*} = -H_k^{*} R_0 - \sigma \alpha_k\;.
\]
then there exists matrices $Y_k, Z_k$ and
$V_k$ such that the rows of $V_k$
are in $\cV^{*}_\text{\rm left}\left(\Sigma\right)$
\begin{equation} 
Y_k \Pi_k = \sigma\;,
\end{equation}
and
\begin{equation}
\left[ Y_k, \ Z_k\right]
\left[\begin{array}{cc}
A & B \\ C & D \end{array}\right]
=
\left[
-\left(\Lambda_k + \alpha_k\beta_k\right)^{*} Y_k
+ V_k , 
\ -\left(H_k+R_0\beta_k\right)^{*}
\right]\;.
\end{equation}
\end{corollary}

After these preliminary statements we can formulate the
theorem determining the subspace 
$\cV^{*}_\text{left}\left(\Sigma_{\text{\bf r}}\right)$ for the
given realization of $F_\text{\bf r}$.

\begin{theorem}\label{thm:fzk_left}
Consider a {\bf minimal} realization of $F$
given by
\[
F(z) \sim 
\left(
\begin{array}{c|c} A &   B
\\
\hline \rule{0cm}{.42cm}
   C & D
\end{array}
\right)
\]
Assume that the columns of the function
$H_{fzk}\left(zI-\Lambda_{fzk}\right)^{-1}$
provide a basis in 
$Z(F)\oplus \cW\left(\text{ker} F\right)$.
Let $\Pi_{fzk}$ be the corresponding solution of
(\ref{matrix_zero_eq_fzk}). 
Consider a maximal solution -- in terms
of $\alpha_0$ and $R_0$ -- of the equation
(\ref{matrix_kernel_eq}) assuming 
-- w.l.o.g. -- that
the column-vectors of the matrix $R_0$ are orthonormal
and the matrices are partitioned according
to (\ref{special_form1}) and 
(\ref{special_form2}).

Consider the function $F_{\text{\bf r}}$ determined
in Theorem \ref{thm:f1} with the realization
given in (\ref{eq:real_f1}).  Assume that
$\Pi'_{fzk}$, $\Lambda'_{fzk}$ 
and $H'_{fzk}$ define a
maximal solution of
\[
\left[
\Pi'_{fzk}\;, H'_{fzk}
\right]
\left[
\begin{array}{cc}
A & B \\
C & D
\end{array}
\right] =
\left[
\Lambda'_{fzk} \Pi'_{fzk} \;, 0
\right]\;,
\]
assuming that the left kernel of $\Pi'_{fzk}$ is trivial.

Then the  maximal solution of the
equation
\begin{equation}\label{eq:extended_eq}
\left[
\bar \Pi'\;, \bar H'
\right]
\left[
\begin{array}{cc}
A & \left(B+\Pi_k\sigma^{-1}H_k^{*}\right)L_0 \\
C & DL_0
\end{array}
\right] =
\left[
\bar \Lambda' \bar \Pi' \;, 0
\right]\;,
\end{equation}
is provided by
\begin{equation}\label{eq:extended_left_zero}
\bar\Pi' =
\left[\begin{array}{c}
\Pi'_{fzk} \\  Y_k
\end{array}
\right]\;,\quad
\bar H' = 
\left[
\begin{array}{c}
H'_{fzk} \\ Z_k
\end{array}
\right]\;,\quad
\bar\Lambda' =
\left[
\begin{array}{cc}
\Lambda'_{fzk} & 0 \\
\Delta' & -\left(\Lambda_k + \alpha_k\beta_k\right)^{*}
\end{array}
\right]\;,
\end{equation}
where the matrices $Y_k, Z_k$ are given in Corollary
\ref{cor:y1_z1} and $\Delta'$ is defined as the unique
solution of equation
\[
\Delta' \Pi'_{fzk} = V_k\;.
\]
\end{theorem}

\begin{proof}
Equations 
$Y_k\Pi_k = \sigma$, $R_0^{*}L_0 = 0$ and 
Corollary \ref{cor:y1_z1} 
imply that the matrices defined in 
(\ref{eq:extended_left_zero}) satisfy equation
(\ref{eq:extended_eq}). 

To prove that it gives amximal solution first let us determine 
the rank of the maximal solution. For a maximal solution we have that 
\begin{eqnarray*}
\text{rank}\left( \bar \Pi'\right)
 &=& \dim \cV^{*}_\text{left}\left(\Sigma_{\text{\bf r}}\right)\\
&=&
n - \dim \cC^{*}\left(\Sigma_{\text{\bf r}}\right) \\
&=&
n - \left(\dim\cC^{*}\left(\Sigma\right) - \text{rank}
\left( \Pi_k\right)\right)\\
&=&
\text{rank}\left( \Pi'_{fzk}\right) + 
\text{rank}\left(\Pi_k\right)\;,
\end{eqnarray*}
where $n$ is the dimension of the state space.

Equation $Y_k \Pi_k = \sigma > 0$ gives that 
$\text{rank} (Y_k) = \text{rank} (\sigma) = 
\text{rank} (\Pi_k)$.
Since $\Pi_{fzk}'\Pi_k = 0$,
while
$Y_k \Pi_k$ is positive definite, 
the left kernel of the matrix 
$\left[\begin{array}{c} \Pi'_{fzk} \\ Y_k
\end{array}\right]$
is trivial and its rank equals to the rank
of the maximal solution, 
concluding thus the proof of the
theorem.
\qed\end{proof}

\bigskip
Now let us return to the both sided factorization of $F$.
Let us apply the factorization ideas given in 
Section \ref{sct:elim_ker} for eliminating the
\underline{left} kernel of $F$ or equivalently of $F_{\text{\bf r}}$.
Since according to Theorem \ref{thm:f1} 
\[
\cV^{*}\left(\Sigma_{\text{\bf r}}\right) \vee 
\cC^{*}\left(\Sigma_{\text{\bf r}}\right) =
\cV^{*}\left(\Sigma\right) \vee 
\cC^{*}\left(\Sigma\right)
\]
Theorem \ref{thm:left_right} implies
that
\[
\cV^{*}_\text{left}\left(\Sigma_{\text{\bf r}}\right) \cap
\cC^{*}_\text{left}\left(\Sigma_{\text{\bf r}}\right) = 
\cV^{*}_\text{left}\left(\Sigma\right) \cap
\cC^{*}_\text{left}\left(\Sigma\right)
\]
giving that $F$ and $F_{\text{\bf r}}$ determine that
same ``left-kernel'' flat inner function,
denoted by $K'_{\beta'}$. (I.e. 
$K'_{\beta'} K_{\beta'}^{'*} = I$.)

With obvious notation:
\[
K'_{\beta'} (z) =
R_0' + \alpha'_k
\left(zI-\left(\Lambda'_k+\beta'_k\alpha'_k\right)
\right)^{-1}
\left( H'_k +\beta'_k R'_0\right)\;,
\]
where
\begin{equation}\label{beta_left}
\beta'_k = -\sigma' \alpha_k^{'*} - H_k'R_0^{'*}
\end{equation}
and $\sigma'$ is the positive definite solution
of the Riccati-equation
\begin{equation}
\label{ric_left}
\left(\Lambda_k' - H_k'R_0^{'*}\alpha_k'\right)\sigma' 
+
\sigma'\left(\Lambda_k' - H_k'R_0^{'*}\alpha_k'\right)^{*}
-\sigma' \alpha_k^{'*}\alpha_k'\sigma'
+H_k'\left(I-R_0^{'*}R_0'\right)H_k^{'*} = 0
 \;.
\end{equation}
Consider the square inner extension of $K'_{\beta'}$:
\begin{equation}\label{eq:Kleftinner}
K'_{\beta',\text{ext}} = \left[
\begin{array}{c}
K'_{\beta'} \\ L'_{\beta'}
\end{array}
\right]\;,
\end{equation}
where 
\[
L_{\beta'}' =
L'_0 - L'_0\left( H'_k+ \beta'_k R'_0 \right)^{*}
\left(zI-\left(\Lambda'_k+\beta'_k\alpha'_k\right)
\right)^{-1}
\left(H'_k+\beta'_kR'_0\right)
\]
and $\left[ \begin{array}{c} R'_0 \\ 
             L'_0\end{array}\right]$
is a unitary matrix.

Theorem \ref{thm:f1} applied to the 
left zero structure gives
the factorization summarized in the following theorem.

\begin{theorem}\label{thm:image_out}
Let a {\bf minimal realization} of $F$ be
given by
\[
F(z) \sim 
\left(
\begin{array}{c|c} A &   B
\\
\hline \rule{0cm}{.42cm}
   C & D
\end{array}
\right)\;.
\]
Consider maximal solutions of the equations 
(\ref{eq:matrix_kerright}) and 
(\ref{eq:matrix_kerleft}) assuming that
the columns (rows) of $R_0$ ($R_0'$)
are orthonormal and they are partitioned as it 
is described in Remark \ref{rem:partition}
(applying it also to the ``left'' structure,
as well). Denote by $\sigma$ ($\sigma'$)
the solutions of the Riccati-equations (\ref{ric})
((\ref{ric_left}) respectively). Define the 
functions $L_\beta$ and $L'_{\beta'}$ 
by (\ref{eq:Kinner}) and (\ref{eq:Kleftinner}).

Consider the function 
$F_{\text{\bf rl}} = 
L'_{\beta'} F L_\beta$. Then

\begin{itemize}
\item[(i)]
the function $F$ has the following factorization
\[
F = L^{'*}_{\beta'} F_{\text{\bf rl}} L_\beta^{*}\;,
\]
where $F_{\text{\bf rl}}$ has the realization
\begin{equation}\label{eq:f_rl_r}
F_{\text{\bf rl}}(z) \sim 
\left(
\begin{array}{c|c} A &  
\left( B+\Pi_k\sigma^{-1}H_k^{*}\right)L_0
\\
\hline \rule{0cm}{.42cm}
   L_0^{'}\left(C+H_k^{'*}\sigma^{'-1}\Pi_k^{'}\right)
 & L_0'DL_0
\end{array}
\right)
\end{equation}

\item[(ii)]
if 
\begin{itemize}
\item[(a)]
if all the eigenvalues of $A$ have non-positive real part , or
\item[(b)]
the matrices
\begin{eqnarray*}
A \quad \text{and}\quad -A^{*}&&
\text{have no common eigenvalues, and}\\
\text{the pair}\quad 
( A, \overline C^{*} )
 &&
\text{is stabilizable (in continuous time sense)} \\
\text{the pair} \quad (\overline B^{*}, A) && 
\text{is detectable (in continuous time sense),}
\end{eqnarray*}
where $\overline C = C P + D B^{*}$ and
$P$ is the solution of the Lyapunov-equation 
\begin{equation}
AP + PA^{*} + BB^{*} = 0\;,
\end{equation}
$\overline B = Q B + C^{*} D$ and $Q$ is the solution of the Lyapunov-equation
\begin{equation}
QA + A^{*}A + C^{*} C = 0
\end{equation} 
\end{itemize}
then the realization (\ref{eq:f_rl_r}) above of $F_{\bf r}$ is minimal.

\item[(iii)]
if the realization given in (\ref{eq:f_rl_r}) of $F_{\bf r}$
is minimal
then for the maximal output-nulling 
controlled invariant
subspace $\cV^{*}\left(\Sigma_{\text{\bf rl}}\right)$ 
and for the minimal input-containing
subspace $\cC^{*}\left(\Sigma_{\text{\bf rl}}\right)$ of 
the realization of $F_{\text{\bf rl}}$
the following identities hold:
\begin{eqnarray}
\cC^{*}\left(\Sigma_{\text{\bf rl}}\right)\cap 
\cV^{*}\left(\Sigma_{\text{\bf rl}}\right) &=& 
\left\{ 0 \right\}\label{eq:r_v_cap} \\
\cC^{*}\left(\Sigma_{\text{\bf rl}}\right)\vee 
\cV^{*}\left(\Sigma_{\text{\bf rl}}\right) &=& 
\CC^n \label{eq:r_v_vee} \\
\cC^{*}\left(\Sigma_{\text{\bf rl}}\right) &=& 
\cC^{*}\left(\Sigma_{\text{\bf r}}\right) \label{eq:r_kifo} \\
\left(\cC^{*}\left(\Sigma\right) \cap 
\cV^{*}\left(\Sigma\right)\right)
\vee \cC^{*}\left(\Sigma_{\text{\bf rl}}\right) &=& 
\cC^{*}\left(\Sigma\right) \label{eq:r_r_kifo} \\
\cV^{*}\left(\Sigma_{\text{\bf rl}}\right) \cap 
\left(\cC^{*}\left(\Sigma\right) \vee 
\cV^{*}\left(\Sigma\right)\right) 
&=& \cV^{*}\left(\Sigma\right)\label{eq:v_v_kifo}
\end{eqnarray}
and
\[
\cW(\ker F_{\text{\bf rl}}) = \left\{ 0 \right\}\;,\quad
\cW(\Im F_{\text{\bf rl}}) = \left\{ 0 \right\}\;,
\]
and the function $F_{\text{\bf rl}}$ is invertible;

\item[(iv)]
if the realization (\ref{eq:f_rl_r}) is minimal 
then the finite zero matrix of
$F_{\text{\bf rl}}$ is given by
\begin{equation}\label{eq:finite_zero_rl}
\left[
\begin{array}{ccc}
-\sigma^{-1}\left(\Lambda_k + \alpha_k\beta_k\right)^{*}
\sigma\;, & 
\Lambda_{kf} + \sigma^{-1} 
\left(H_k + R_0\beta_k\right)^{*} H_f\;, & 
\sigma^{-1} \left(H_k + R_0\beta_k\right)^{*} Z'_k \\
0 & \Lambda_f &  \Delta \\
0 & 0 & -\left(\Lambda'_k +\beta'_k\alpha'_k\right)^{*}
\end{array}
\right]
\end{equation}
for some matrices $Z'_k$ and $\Delta$.
\end{itemize}
\end{theorem}

Note that the eigenvalues of $\Lambda_f$ determine the
finite zeros of the function $F$,
while the matrices $\Lambda_k, \alpha_k$ and
$\Lambda'_k, \alpha'_k$
are connected to the right and left kernel spaces
-- 
$\cW\left(\ker F \right)$, 
$\cW \left(\ker_\text{left}F  \right)$ --
of $F$, respectively.

\begin{proof}
Substitute into 
$\left(L_{\beta'}^{'}\right)^{*} 
F_{\text{\bf rl}} L^{*}_{\beta}$ the
definition of $F_{\text{\bf rl}}$:
\[
L_{\beta'}^{'*} F_{\text{\bf rl}} L_\beta^{*} = 
L_{\beta'}^{'*} L_{\beta'}^{'}F L_\beta L_\beta^{*} =
\left(L_{\beta'}^{'*} L_{\beta'}^{'} + 
K_{\beta'}^{'*} K_{\beta'}^{'}\right) F
\left( L_\beta L_\beta^{*} + K_\beta K_\beta^{*}\right) =
F
\]
Straightforward computation (or immediate application of
Theorem \ref{thm:f1} gives that the function
$F_{\text{\bf rl}}$ has the realization:
\begin{multline*}
F_{\text{\bf rl}} (z) = L_{\beta'}^{'} F_{\text{\bf r}} \\
= \left(L_0^{'} - L_0^{'}H_k^{'*} \sigma^{'-1}
\left(zI - \left(\Lambda_k^{'} + \beta_k^{'} \alpha_k^{'}\right)
\right)^{-1}\left(H_k^{'}+\beta_k^{'}R_0^{'}\right)\right)\\
\left(DL_0 + C\left(zI-A\right)^{-1}
\left(B+\Pi_k\sigma^{-1}H_k^{*}\right)L_0\right) \\
=
L_0^{'} D L_0 + L_0^{'}
\left(C+H_k^{'*} \sigma^{'-1}\Pi_k^{'}\right)
\left(zI-A\right)^{-1}
\left(B+\Pi_k\sigma^{-1}H_k^{*}\right)
L_0\;,
\end{multline*}
using the identities
\begin{eqnarray*}
\left(H_k^{'} +\beta_k^{'} R_0^{'}\right) C &=&
\Pi_k^{'}(zI-A) - 
(zI-(\Lambda_k^{'}+\beta_k^{'}\alpha_k^{'}))
\Pi_k^{'} \\
\left(H_k^{'} +\beta_k^{'} R_0^{'}\right) DL_0&=&
- \Pi_k^{'} \left(B+\Pi_k\sigma^{-1}H_k^{*}\right) L_0\;,
\end{eqnarray*}
proving (i).

(ii) Both a) and b) parts can be proven using part (iii) Theorem \ref{thm:f1}. 
In fact, under the condition
 that all eigenvalues of $A$ have non-positive real part part (iii) a) of
Theorem \ref{thm:f1} gives that the reachability
subspaces $<A \mid B >$ and 
$< A \mid (B+\Pi_k \sigma^{-1} H_k^{*})L_0 >$ coincide.
Applying the "left" version of this result we obtain that the
non-observability subspaces of the pairs $(C, A)$ and
$(L_0' (C+H_k^{'*} \sigma^{'-1} \Pi_k^{'}), A)$ coincide.
But according to our assumption the realization $F$ is minimal, consequently 
the realization $F_{\bf rl}$ above is also minimal. 

Concerning the b) part of this statement now part (iii) b) of Theorem and 
its "left" version gives again that reachability subspaces above and
non-observability subspaces above coincide, giving again the minimality
of the realization (\ref{eq:f_rl_r}).

(iii)
Denote by $\cV^{*}_\text{left}\left(\Sigma_{\text{\bf rl}}\right)$, 
$\cC^{*}_\text{left}\left(\Sigma_{\text{\bf rl}}\right)$
the maximal output-nulling controlled invariant subspace 
and the minimal input-containing subspace of the 
realization of $F_{\text{\bf rl}}$ given in (\ref{eq:f_rl_r})
with respect to the left multiplication. 

Then Theorem \ref{thm:left_right} allows us to transform the 
results of Theorem \ref{thm:f1} to the 
left zero structure of $F_{\text{\bf r}}$. Consequently,
\begin{eqnarray*}
\cC^{*}_\text{left}\left(\Sigma_{\text{\bf r}}\right) &=&
\cC^{*}_\text{left}\left(\Sigma\right)\;, \\
\cC^{*}_\text{left}\left(\Sigma_{\text{\bf r}}\right) \vee
\cV^{*}_\text{left}\left(\Sigma_{\text{\bf r}}\right) &=& \CC^n
\;, \\
\left( \cC^{*}_\text{left}\left(\Sigma\right)\vee
  \cV^{*}_\text{left}\left(\Sigma\right)\right) \cap
\cV^{*}_\text{left}\left(\Sigma_{\text{\bf r}}\right) &=&
\cV^{*}_\text{left}\left(\Sigma\right)\;.
\end{eqnarray*}

The pair
$\left(A, \left(B+\Pi_k\sigma^{-1}H_k^{*}\right)L_0\right)$
is reachable due to the our assumption that 
the realization (\ref{eq:f_rl_r}) is minimal
Consequently, Theorem \ref{thm:f1}
can be applied to the "left" factorization of $F_{\text{\bf r}}$ 
yielding that
\begin{eqnarray*}
\cV^{*}_\text{left}\left(\Sigma_{\text{\bf rl}}\right) &=&
\cV^{*}_\text{left}\left(\Sigma_{\text{\bf r}}\right)\;, \\
\cV^{*}_\text{left}\left(\Sigma_{\text{\bf rl}}\right) \cap
\cC^{*}_\text{left}\left(\Sigma_{\text{\bf rl}}\right) 
&=& \left\{ 0 \right\}
\;, \\
\left( \cV^{*}_\text{left}\left(\Sigma_{\text{\bf r}}\right)\cap
  \cC^{*}_\text{left}\left(\Sigma_{\text{\bf r}}\right)\right) \vee
\cC^{*}_\text{left}\left(\Sigma_{\text{\bf rl}}\right) &=&
\cC^{*}_\text{left}\left(\Sigma_{\text{\bf r}}\right)\;.
\end{eqnarray*}
In the last equation taking on both sides
the generated subspace by
$\cV^{*}_\text{left}\left(\Sigma_{\text{\bf rl}}\right) =
 \cV^{*}_\text{left}\left(\Sigma_{\text{\bf r}}\right)$
we obtain that
\[
\cV^{*}_\text{left}\left(\Sigma_{\text{\bf rl}}\right)
\vee
\cC^{*}_\text{left}\left(\Sigma_{\text{\bf rl}}\right) = \CC^n\;.
\]
Invoking Theorem \ref{thm:left_right}
(i.e. taking the orthogonal complements
of these subspaces)
we obtain that
\begin{eqnarray*}
\cC^{*}\left(\Sigma_{\text{\bf rl}}\right) &=&
\cC^{*}\left(\Sigma_{\text{\bf r}}\right)\;, \\
\cC^{*}\left(\Sigma_{\text{\bf rl}}\right) \vee
\cV^{*}\left(\Sigma_{\text{\bf rl}}\right) &=& \CC^n
\;, \\
\cC^{*}\left(\Sigma_{\text{\bf rl}}\right) \cap
\cV^{*}\left(\Sigma_{\text{\bf rl}}\right)
&=& \left\{ 0 \right\}\\
\left( \cC^{*}\left(\Sigma_{\text{\bf r}}\right)\vee
  \cV^{*}\left(\Sigma_{\text{\bf r}}\right)\right) \cap
\cV^{*}\left(\Sigma_{\text{\bf rl}}\right) &=&
\cV^{*}\left(\Sigma_{\text{\bf r}}\right)
\;.
\end{eqnarray*}
The last equation can be written as
\begin{equation}
\left( \cC^{*}\left(\Sigma\right)\vee
  \cV^{*}\left(\Sigma\right)\right) \cap
\cV^{*}\left(\Sigma_{\text{\bf rl}}\right) =
\cV^{*}\left(\Sigma\right)
\end{equation}
The complementary property of the subspaces 
$\cC^{*}\left(\Sigma_{\text{\bf rl}}\right)$ and 
$\cV^{*}\left(\Sigma_{\text{\bf rl}}\right)$ 
gives that the zero modules
$\cW\left( \ker F_{\text{\bf rl}}\right)$ and
$\cW\left( \Im \,F_{\text{\bf rl}}\,\right)$
are trivial.
In fact, Corollary \ref{cor:ker_triv} and
Theorem \ref{thm:char_im} can be applied
(using that the realization of $F_{\text{\bf rl}}$ is
observable).

\bigskip
According Proposition 4 in \cite{AL-SC-84}
the invertibility of a proper transfer function
is equivalent to that the property
that the corresponding $\cV^{*}\left(\Sigma\right)$ and
$\cC^{*}\left(\Sigma\right)$ are complementary 
subspaces and the columns of
$\left[\begin{array}{c} B \\ D \end{array}\right]$
are linearly independent, the rows of $\left[ C ,\ D\right]$
are linearly independent.

In the present situation the complementary
property of $\cV^{*}\left(\Sigma_{\text{\bf rl}}\right)$ and 
$\cC^{*}\left(\Sigma_{\text{\bf rl}}\right)$ was just proved. 

Furthermore, in Remark
\ref{rem:F_r_left_inv} we have checked the left
invertibility of $F_{\text{\bf r}}$. Similar argument gives
the left-invertibility of $F_{\text{\bf rl}}$. In fact,
if for some vector $\xi$ both
$\left(B+\Pi_l\sigma^{-1}H_k^{*}\right)L_0 \xi = 0$ and
$L'_0 D L_0 \xi=0$ then the identity $R'_0 D =0$ implies
that $DL_0\xi = 0$. From the first equation we obviously get
that $B L_0 \xi \in \Im \Pi_k \subset \Im  \Pi$. Using the 
maximality of $R_0$ we obtain that $L_0\xi \in \Im R_0$, i.e.
$\xi =0$. The right invertibility of $F_{\text{\bf rl}}$ can be 
proved with obvious modification, 
concluding the proof of part (iii).

(iv) 
The Riccati-equation (\ref{ric_left}) using (\ref{beta_left}) can be written
as
\begin{equation}\label{eq:ric_ext_left}
\left[
\begin{array}{cc}
\Lambda'_k+\beta'_k\alpha'_k &
H'_k+\beta'_kR'_0\\
\alpha'_k & R'_0
\end{array}\right]
\left[
\begin{array}{c}
\sigma' \\
\left(H'_k + \beta'_kR'_0\right)^{*}
\end{array}\right]
=
\left[\begin{array}{c}
-\sigma'
\left(\Lambda'_k + \beta'_k\alpha'_k\right)^{*}
\\
0
\end{array}\right]
\end{equation}

Now invoking Lemma \ref{lem:sigma_h_lambda_mod} 
(for the left multiplication) and
Remark \ref{rem:R_compl_V} we get that 
there exist matrices
$Y'_k, Z'_k$ and $V'_k$ 
where the columns of $V'_k$ are in
$\Im \left(\Pi_f\right)$ such that
\[
\Pi'_k Y'_k = \sigma'
\] 
and
\begin{equation}\label{eq:left_y_z}
\left[\begin{array}{cc}
A & B \\
C & D
\end{array}\right]
\left[
\begin{array}{c}
Y'_k \\ Z'_k
\end{array}\right]
=
\left[\begin{array}{c}
-Y'_k \left(\Lambda'_k + \beta'_k\alpha'_k\right)^{*}
+ V'_k
\\
-\left(H'_k+\beta'_k R'_0\right)^{*}
\end{array}\right]\;.
\end{equation}
Straightforward calculation gives that
\begin{eqnarray*}
&& 
\left[
\begin{array}{cc}
A & \left(B+\Pi_k\sigma^{-1}H_k^{*}\right)L_0\\
L'_0\left(C+H'_k\sigma^{'-1}\Pi'_k\right) & 
L'_0 D L_0
\end{array}\right]
\left[\begin{array}{cc}
\Pi_{fzk} & Y'_k \\ 
L_0^{*}H_{fzk} & L_0^{*}Z'_k  
\end{array}\right]\\
&& =
\left[\begin{array}{cc}
\Pi_{fzk} \left(\Lambda_{fzk} + \Gamma H_{fzk}\right) &
-Y'_k \left(\Lambda'_k + \beta'_k\alpha'_k\right)^{*}
+ V'_k + \Pi_k\sigma^{-1}
\left(H_k^{*} + \beta_k^{*} R_0^{*}\right)Z'_k
\\
0 & 0 
\end{array}\right] \\
&=&
\left[\begin{array}{cc}
\Pi_{fzk} \left(\Lambda_{fzk} + \Gamma H_{fzk}\right) &
-Y'_k \left(\Lambda'_k + \beta'_k\alpha'_k\right)^{*}
+ \Pi_f \Delta + \Pi_k\sigma^{-1}
\left(H_k^{*} + \beta_k^{*} R_0^{*}\right)Z'_k
\\
0 & 0 
\end{array}\right]
\end{eqnarray*}
where $\Gamma$ is defined in (\ref{eq:def_gamma}) and
$\Delta$ is defined by $\Pi_f \Delta = V'_k$.

In fact, the term $(1,1)$ is the same equation
as the first term in (\ref{pi_in_pibar}).
The identity $\Pi'_k \Pi_{fzk} = 0$ implies that the term
$(2,1)$ is essentially identical to the second equation in 
(\ref{pi_in_pibar}). 

On the other hand --
using that $DL_0L_0^{*} = D$,
$\sigma' = \Pi'_k Y'_k$ and
$L'_0 R_0^{'*} = 0$ --
\begin{eqnarray*}
L'_0\left(C+H'_k\sigma^{'-1}\Pi'_k\right)Y'_k + 
L'_0 D L_0 L_0^{*} Z'_k &=& 
- L'_0 \left(H'_k+\beta'_kR'_0\right)^{*} 
+ 
L'_0 H'_k\sigma^{'-1}\Pi'_k Y'_k \\
&=& 0\;,   
\end{eqnarray*}
and finally -- from
$BR_0 = \Pi_k\alpha_k$ -- we get that 
\begin{eqnarray*}
A Y'_k + \left(B+ \Pi_k\sigma^{-1} H_k^{*}\right)
L_0 L_0^{*} Z'_k &=& 
A Y'_k + B Z'_k - B R_0R_0^{*} Z'_k +
\Pi_k\sigma^{-1} H_k^{*} L_0L_0^{*} Z'_k  \\
&=& -Y'_k\left(\Lambda'_k+ \beta'_k\alpha'_k\right)^{*}
-\Pi_k
\left(\alpha_k R_0^{*} - \sigma^{-1} H_k^{*} L_0L_0^{*}\right)
Z'_k + V_k^{"}\\
&=& -Y'_k \left(\Lambda'_k+ \beta'_k\alpha'_k\right)^{*}
+\Pi_k \sigma^{-1}
\left(H_k + R_0\beta_k\right)^{*} Z'_k + \Pi_f \Delta\;.
\end{eqnarray*}

The identities $\Pi'_k \Pi_{fzk} = 0$, $\Pi'_kY'_k = \sigma' > 0$
give that the columns of $\left[ \Pi_{fzk}, Y'_k\right]$
are linearly independent. 
Furthermore,
\begin{eqnarray*}
\text{rank} \left(\left[ \Pi_{fzk}, Y'_k\right]\right) &=& 
\dim \left( \cV^{*}\left(\Sigma\right)\right) + 
\text{rank} \left(\Pi'_k\right) \\
&=& 
\dim \left( \cV^{*}\left(\Sigma\right)\right) +
\dim \left( 
\cV_\text{left}^{*}\left(\Sigma\right)\cap
\cC_\text{left}^{*}\left(\Sigma\right)
\right) \\
&=& 
\dim \left( \cV^{*}\left(\Sigma_{\text{\bf r}}\right)\right)
+
\left(n-\dim 
\left( \cV^{*}\left(\Sigma_{\text{\bf r}}\right)\vee 
\cC^{*}\left(\Sigma_{\text{\bf r}}\right)\right)\right)
 \\
&=& n - \cC^{*}\left(\Sigma_{\text{\bf r}}\right) \\
&=& n - \cC^{*}\left(\Sigma_{\text{\bf rl}}\right)\;,
\end{eqnarray*} 
proving the maximality 
of 
$\left[ \Pi_{fzk}, Y'_k\right]$
using the observation that
$\cC^{*}\left(\Sigma_{\text{\bf r}}\right)$ and
$\cV^{*}\left(\Sigma_{\text{\bf r}}\right)$ are complementary subspaces
-- and giving that
\[
\cV^{*}\left(\Sigma_{\text{\bf rl}}\right) = 
\Im \left[ \Pi_{fzk}, \ Y'_k\right]\;.
\]
Using the minimality of the realization $F_{\text{\bf rl}}$ we obtain
that the finite zero matrix of $F_{\text{\bf rl}}$ 
is determined by the equation
(\ref{eq:finite_zero_rl}).
This concludes the proof of the theorem.
\qed\end{proof}

\vskip 1truecm
{\bf Acknowledgments}:  part of this research
took place while the author was visiting the
Royal Institute of Technology in January 2006. The
warm hospitality and support is gratefully
acknowledged.

\end{document}